\documentclass[11pt]{article}
\usepackage[latin1]{inputenc}
\usepackage{amsmath,amsthm,amssymb}
\usepackage{amsfonts}
\usepackage{amsmath,amsthm,amssymb,amscd}
\usepackage{latexsym}
\usepackage{color}
\usepackage{graphicx}
\usepackage{mathrsfs}
\usepackage{cite}

\usepackage{color,enumitem,graphicx}
\usepackage[colorlinks=true,urlcolor=black,
citecolor=black,linkcolor=black,linktocpage,pdfpagelabels,
bookmarksnumbered,bookmarksopen]{hyperref}

\makeatletter \@addtoreset{equation}{section} \makeatother

\newtheorem{theorem}{Theorem}[section]
\newtheorem{definition}{Definition}[section]
\newtheorem{proposition}{Proposition}[section]
\newtheorem{lemma}{Lemma}[section]
\newtheorem{remark}{Remark}[section]

\allowdisplaybreaks

\begin{document}
\title{Blow-up solutions concentrated along minimal submanifolds for asymptotically critical Lane-Emden systems on Riemannian manifolds}

\author{Wenjing Chen\footnote{Corresponding author.}\ \footnote{E-mail address:\, {\tt wjchen@swu.edu.cn} (W. Chen), {\tt zxwangmath@163.com} (Z. Wang).}\  \ and Zexi Wang\\
\footnotesize  School of Mathematics and Statistics, Southwest University,
Chongqing, 400715, P.R. China}

\date{ }
\maketitle

\begin{abstract}
{Let $(\mathcal{M},g)$ and $(\mathcal{K},\kappa)$ be two
Riemannian manifolds of dimensions $N$ and $m$, respectively. Let $\omega\in C^2(\mathcal{M})$, $\omega>0$. The warped product $\mathcal{M}\times _\omega \mathcal{K}$ is the $(N+m)$-dimensional product manifold $\mathcal{M}\times \mathcal{K}$ furnished with metric $g+\omega^2\kappa$.
We are concerned with the following elliptic system
\begin{align}\label{yuanshi}
 \left\{
  \begin{array}{ll}
  -\Delta_{g+\omega^2\kappa} u+h(x)u=v^{p-\alpha \varepsilon},
  \ \  &\mbox{in $(\mathcal{M}\times _\omega \mathcal{K},g+\omega^2\kappa)$},\\
   -\Delta_{g+\omega^2\kappa} v+h(x)v=u^{q-\beta \varepsilon},
  \ \  &\mbox{in $(\mathcal{M}\times _\omega \mathcal{K},g+\omega^2\kappa)$},\\
  u,v>0,
  \ \  &\mbox{in $(\mathcal{M}\times _\omega \mathcal{K},g+\omega^2\kappa)$},
    \end{array}
    \right.
  \end{align}
where $\Delta _{g+\omega^2\kappa}=div_{g+\omega^2\kappa} \nabla$ is the Laplace-Beltrami operator on $\mathcal{M}\times _\omega \mathcal{K}$, $h(x)$ is a $C^1$-function on $\mathcal{M}\times _\omega \mathcal{K}$, $\varepsilon>0$ is a small parameter, $\alpha,\beta>0$ are real numbers, $\varepsilon$ is a positive parameter, $(p,q)\in (1,+\infty)\times (1,+\infty)$ satisfies $\frac{1}{p+1}+\frac{1}{q+1}=\frac{N-2}{N}$.
For any given integer $k\geq2$, using the Lyapunov-Schmidt reduction, we prove that problem \eqref{yuanshi} has a $k$-peaks solution concentrated along a $m$-dimensional minimal submanifold of $(\mathcal{M}\times _\omega \mathcal{K})^k$.}


\end{abstract}

\section{Introduction}
Let $(\mathfrak{M},\mathfrak{g})$ be a $n$-dimensional smooth compact Riemannian manifold without boundary, where $\mathfrak{g}$ denotes the metric tensor. We consider the following elliptic system
\begin{align}\label{pro}
 \left\{
  \begin{array}{ll}
  -\Delta_\mathfrak{g} u+h(x)u=v^{p},
  \ \  &\mbox{in}\ (\mathfrak{M},\mathfrak{g}),\\
   -\Delta_\mathfrak{g} v+h(x)v=u^{q},
  \ \  &\mbox{in}\ (\mathfrak{M},\mathfrak{g}),\\
  u,v>0,
  \ \  &\mbox{in}\ (\mathfrak{M},\mathfrak{g}),
    \end{array}
    \right.
  \end{align}
 where $\Delta _\mathfrak{g}=div_\mathfrak{g} \nabla$ is the Laplace-Beltrami operator on $\mathfrak{M}$, $h(x)$ is a $C^1$-function on $\mathfrak{M}$,  $p,q>1$.

The starting point on the study of system \eqref{pro} is its scalar version
\begin{equation}\label{yamabe}
   -\Delta_\mathfrak{g} u+h(x)u=u^{p},\quad u>0,\quad \text{in $(\mathfrak{M},\mathfrak{g})$}.
 \end{equation}
If $p\in(1,2^*-1)$ ($2^*=\frac{2n}{n-2}$ if $n\geq3$ and $2^*=+\infty$ if $n=1,2$), by the compact embedding $H_\mathfrak{g}^1(\mathfrak{M})\hookrightarrow L^p_\mathfrak{g}(\mathfrak{M})$, one can obtain a solution of \eqref{yamabe}. In \cite{MP1}, Micheletti and Pistoia also considered the following subcritical problem
\begin{equation}\label{sub}
   -\varepsilon^2\Delta_\mathfrak{g} u+u=u^{p},\quad u>0,\quad \text{in $(\mathfrak{M},\mathfrak{g})$},
 \end{equation}
where $p\in(1,2^*-1)$, $n\geq2$, and $\varepsilon>0$ is a small parameter. By performing the Lyapunov-Schmidt reduction procedure, they obtained a single blowing-up solution for \eqref{sub}. Successively,  multiple blowing-up solutions and clustered solutions are constructed in \cite{MP2} and \cite{DMP}, respectively.

In the critical case $p=2^*-1$, the situation is more complicated, and
the existence of solutions for \eqref{yamabe} is related to the position of the potential $h$ with respect to the geometric potential
\begin{equation*}
  h_\mathfrak{g}=\frac{n-2}{4(n-1)}Scal_\mathfrak{g},
\end{equation*}
where $Scal_g$ is the scalar curvature of the manifold.
Particularly, if $h(x)\equiv h_\mathfrak{g}$, equation \eqref{yamabe} is referred as the {\em Yamabe problem} and it always has a solution, see
 e.g. \cite{Y,Au,Sch,Tru}.

The supercritical case $p>2^*-1$ is even more difficult to deal with.   Based on the Lyapunov-Schmidt reduction,
Micheletti et al. \cite{MPV} first constructed a single blowing-up solution for  \eqref{yamabe}  in asymptotically critical case (i.e., $p=2^*-1\pm \varepsilon$ with $\varepsilon>0$ small enough). Here, we say that a family of solutions $u_\varepsilon$ of \eqref{yamabe} blows up and concentrates at $\xi_0\in \mathfrak{M}$ if there exists a family of points $\xi_\varepsilon \in \mathfrak{M}$ such that $\xi_\varepsilon\rightarrow \xi_0$ and $u_\varepsilon(\xi_\varepsilon)\rightarrow+\infty$ as $\varepsilon\rightarrow0$. Since then, equation  \eqref{yamabe} has been studied extensively, see \cite{DMW,RV} for sign-changing blowing-up solutions, \cite{Deng} for multiple blowing-up solutions, \cite{Chen} for clustered solutions, \cite{CK,PV} for sign-changing bubble tower solutions, and so on. In particular, we are interested in the result due to Ghimenti et al. \cite{GMP}, where the authors obtained a single blowing-up solution concentrated along a minimal submanifold of $\mathfrak{M}$.

If  $\mathfrak{M}$ is either a smooth bounded domain or $\mathbb{R}^n$,  system \eqref{pro} reduces to the following elliptic system
 \begin{align}\label{back}
 \left\{
  \begin{array}{ll}
  -\Delta u=v^p,
  \quad  &\mbox{in}\ \Omega,\\
   -\Delta  v=u^q,
  \quad  &\mbox{in}\ \Omega,\\
   u,v>0\quad &\mbox{in}\  \Omega,
    \end{array}
    \right.
  \end{align}
called the Lane-Emden system. Here, $n\geq3$, $p,q>1$, $\Omega$ is either a smooth bounded domain or $\mathbb{R}^n$.
 In this case, the critical hyperbola
 \begin{equation}\label{ch}
    \frac{1}{p+1}+\frac{1}{q+1}=\frac{n-2}{n},
  \end{equation}
plays a similar role to the Sobolev exponent $2^*$ for the single equation.
System \eqref{back} has received remarkable attention for decades.   When $\Omega=\mathbb{R}^n$,
by applying the concentration compactness principle, Lions \cite{Lions} found a positive least energy solution of \eqref{back}-\eqref{ch}, which is radially symmetric and radially decreasing.  Wang \cite{Wang} and Hulshof and Van der Vorst \cite{HV} independently proved the uniqueness of the
positive least energy solution $(U_{1,0}(z),V_{1,0}(z))$ to \eqref{back}-\eqref{ch}.
Moreover, Frank et al. \cite{FKP} established the non-degeneracy of \eqref{back}-\eqref{ch} at each least energy solution.
 Using the Lyapunov-Schmidt reduction and the non-degeneracy result obtained in \cite{FKP}, Guo et al. \cite{GLP1} established the existence and non-degeneracy of multiple blowing-up solutions to \eqref{back}-\eqref{ch} with two potentials.
For more investigations about system \eqref{back} with $\Omega=\mathbb{R}^n$, we can see \cite{GM,CS}.

If $\Omega$ is a smooth bounded domain, Kim and Pistoia \cite{KP} obtained the existence of multiple blowing-up solutions for the following system
\begin{align*}
 \left\{
  \begin{array}{ll}
  -\Delta u=|v|^{p-1}v+\varepsilon(\alpha u+\beta_1 v),
  \quad  &\mbox{in}\ \Omega,\\
   -\Delta  v=|u|^{q-1}u+\varepsilon(\alpha v+\beta_2 u),
  \quad  &\mbox{in}\ \Omega,\\
   u,v=0,\quad &\mbox{on}\ \partial \Omega,
    \end{array}
    \right.
  \end{align*}
where $n\geq8$, $\varepsilon>0$, $\alpha,\beta_1,\beta_2\in \mathbb{R}$, $1<p<\frac{n-1}{n-2}$, and $(p,q)$ satisfies \eqref{ch}.
Meanwhile, they also found solutions to the following asymptotically critical system when $n\geq4$, $\alpha,\beta>0$
\begin{align*}
 \left\{
  \begin{array}{ll}
  -\Delta u=v^{p-\alpha \varepsilon},
  \quad  &\mbox{in}\ \Omega,\\
   -\Delta  v=u^{q-\beta \varepsilon},
  \quad  &\mbox{in}\ \Omega,\\
  u,v>0,\quad &\mbox{in}\ \Omega,\\
   u,v=0,\quad &\mbox{on}\ \partial \Omega.
    \end{array}
    \right.
  \end{align*}
Moreover, using the local Pohozaev identity, Guo et al. \cite{GHP} proved its non-degeneracy.
Recently,
Jin and Kim \cite{JK} studied the Coron's problem for the critical Lane-Emden system, and established the existence, qualitative properties of positive solutions.
Guo and Peng \cite{GP} considered sign-changing solutions to the asymptotically critical Lane-Emden system with Neumann boundary conditions. It is worth emphasizing that
Guo et al. \cite{GLP2} obtained positive solutions with boundary layers concentrating along one or several submanifolds for the asymptotically critical Lane-Emden system.
For more classical results regarding Hamiltonian systems in bounded domains,
we refer the readers to \cite{CK1,DY,PST,BMT,KM} and references therein.

Motivated by the results already mentioned above,  noticed that a point is a $0$-dimensional manifold, it is natural to ask that, {\em does problem \eqref{pro} have solutions blowing up and concentrating at a $m$-dimensional submanifold of $\mathfrak{M}$ when $\frac{1}{p+1}+\frac{1}{q+1}\rightarrow\frac{n-m-2}{n-m}$?} %
In this paper, we give a positive answer when $(\mathfrak{M},\mathfrak{g})$ is a warped product manifold.

We recall the notion of warped product manifold introduced by Bishop and O'Neill in \cite{BO}. Let $(\mathcal{M},g)$ and $(\mathcal{K},\kappa)$ be two
Riemannian manifolds of dimensions $N$ and $m$, respectively. Let $\omega\in C^2(\mathcal{M})$, $\omega>0$. The warped product $\mathfrak{M}:=\mathcal{M}\times _\omega \mathcal{K}$ is the $n:=N+m$ dimensional product manifold $\mathcal{M}\times \mathcal{K}$ furnished with metric $\mathfrak{g}=g+\omega^2\kappa$. The function $\omega$ is called a warping function.
If $u\in C^2(\mathfrak{M})$, it holds
\begin{equation}\label{lap}
  \Delta_{\mathfrak{g}}u=\Delta_g u+\frac{N}{\omega}\nabla_g \omega \cdot \nabla_g u+\frac{1}{\omega^2}\Delta _\kappa u.
\end{equation}
Assume that $h$ is invariant with respect to $\mathcal{K}$, i.e., $h(x,y)=h(x)$ for any $(x,y)\in \mathcal{M}\times \mathcal{K}$. We look for solutions of \eqref{pro} which are invariant with respect to $\mathcal{K}$, i.e., $(u(x,y),v(x,y))=(u_1(x),v_1(x))$, then by \eqref{lap}, we know $(u,v)$ solves \eqref{pro} if and only if $(u_1,v_1)$ solves
\begin{align*}
 \left\{
  \begin{array}{ll}
  -\Delta_g u_1-\frac{N}{\omega}\nabla_g \omega \cdot \nabla_g u_1+hu_1=v_1^{p},
  \ \  &\mbox{in}\ (\mathcal{M},g),\\
   -\Delta_g v_1-\frac{N}{\omega}\nabla_g \omega \cdot \nabla_g v_1+hv_1=u_1^{q},
  \ \  &\mbox{in}\ (\mathcal{M},g),\\
  u_1,v_1>0,
  \ \  &\mbox{in}\ (\mathcal{M},g),
    \end{array}
    \right.
  \end{align*}
or equivalently
 \begin{align}\label{pro2}
 \left\{
  \begin{array}{ll}
  -div_g(\omega^N\nabla_g u_1)+\omega^Nhu_1=\omega^Nv_1^{p},
  \ \  &\mbox{in}\ (\mathcal{M},g),\\
   -div_g(\omega^N\nabla_g v_1)+\omega^Nhv_1=\omega^Nu_1^{q},
  \ \  &\mbox{in}\ (\mathcal{M},g),\\
  u_1,v_1>0,
  \ \  &\mbox{in}\ (\mathcal{M},g).
    \end{array}
    \right.
  \end{align}
  It is clear that if $(u_1,v_1)$ is a pair of solutions  of \eqref{pro2} which blows up and concentrates at $\xi_0\in \mathcal{M}$, then $(u(x,y),v(x,y))=(u_1(x),v_1(x))$ is a  pair of solutions  of \eqref{pro} which blows up and concentrates along the fiber $\{\xi_0\}\times \mathcal{K}$, which is a $m$-dimensional submanifold of $\mathfrak{M}$. Moreover, we can see that  $\{\xi_0\}\times \mathcal{K}$ is a minimal submanifold of $\mathfrak{M}$ if $\xi_0$ is a critical point of $\omega$.

Therefore, we are led to study the following problem
\begin{align}\label{prob}
 \left\{
  \begin{array}{ll}
  -div(a(x)\nabla_gu)+a(x)hu=a(x)v^{p-\alpha \varepsilon},
  \ \  &\mbox{in}\ (\mathcal{M},g),\\
   -div(a(x)\nabla_gv)+a(x)hv=a(x)u^{q-\beta \varepsilon},
  \ \  &\mbox{in}\ (\mathcal{M},g),\\
  u,v>0,
  \ \  &\mbox{in}\ (\mathcal{M},g),
    \end{array}
    \right.
  \end{align}
where $a\in C^2(\mathcal{M})$, $\min\limits_{x\in \mathcal{M}}a(x)>0$, $h\in C^1(\mathcal{M})$, $\varepsilon>0$ is a small parameter, $\alpha,\beta>0$ are real numbers, $(p,q)\in (1,+\infty)\times (1,+\infty)$ is a pair of numbers satisfying
 \begin{equation}\label{pq}
    \frac{1}{p+1}+\frac{1}{q+1}=\frac{N-2}{N}.
  \end{equation}
Without loss of generality, we assume that $1<p\leq \frac{N+2}{N-2}\leq q$. %Moreover, by \eqref{ch}, we have $p>\frac{2}{N-2}$.

To state our main result, we give the following  definition.
\begin{definition}
{\rm For $k\geq2$ to be a positive integer, let $(u_\varepsilon,v_\varepsilon)$ be a family of solutions of \eqref{prob}, we say that
$(u_\varepsilon,v_\varepsilon)$ blows up and concentrates at $\bar{\xi^0}=(\xi_1^0,\xi_2^0,\cdots, \xi_k^0)\in \mathcal{M}^k$ if there exists $(\delta_1^\varepsilon,\delta_2^\varepsilon,\cdots, \delta_k^\varepsilon)\in (\mathbb{R}^+)^k$ and $\bar{\eta}^\varepsilon=(\eta_1^\varepsilon,\eta_2^\varepsilon,\cdots, \eta_k^\varepsilon)\in (\mathbb{R}^N)^k$ such that
$\delta_j^\varepsilon\rightarrow0$ and $\eta_j^\varepsilon\rightarrow 0$ as $\varepsilon\rightarrow0$ for $j=1,2,\cdots,k$, and
\begin{equation*}
  \Big\|(u_\varepsilon,v_\varepsilon)-\Big(\sum\limits_{j=1}^kW_{\delta_j^\varepsilon,\xi^0_j,\eta_j^\varepsilon},
  \sum\limits_{j=1}^kH_{\delta_j^\varepsilon,\xi^0_j,\eta_j^\varepsilon}\Big)\Big\|\rightarrow0,\quad \text{as $\varepsilon\rightarrow0$},
\end{equation*}
where $\|\cdot\|$ and $\big(W_{\delta,\xi,\eta},H_{\delta,\xi,\eta}\big)$ are defined in \eqref{norm} and \eqref{fun1}, respectively.}
\end{definition}

Let us recall $(U_{1,0}(z),V_{1,0}(z))$, which is the least energy solution of \eqref{back}-\eqref{ch} in $\mathbb{R}^N$ given by Wang \cite{Wang} and Hulshof and Van der Vorst \cite{HV}. Let $L_1,L_2,\cdots,L_7$ be positive numbers defined by
\begin{align}\label{chang}
 \left\{
  \begin{array}{ll}
  \displaystyle L_1=\int\limits_{\mathbb{R}^N}\nabla U_{1,0}\cdot \nabla V_{1,0}dz,\\
  \displaystyle L_2=\int\limits_{\mathbb{R}^N}|z|^2\nabla U_{1,0}\cdot \nabla V_{1,0}dz,\\
 \displaystyle L_3=\int\limits_{\mathbb{R}^N}U_{1,0}\cdot  V_{1,0}dz,
    \end{array}
    \right.
    \quad \text{and}\quad\ \
    \left\{
  \begin{array}{ll}
  \displaystyle L_4=\int\limits_{\mathbb{R}^N}|z|^2V_{1,0}^{p+1}dz,\\
   \displaystyle L_5=\int\limits_{\mathbb{R}^N}|z|^2U_{1,0}^{q+1}dz,\\
  \displaystyle L_6=\int\limits_{\mathbb{R}^N}V_{1,0}^{p+1}\log V_{1,0}dz,\\
  \displaystyle  L_7=\int\limits_{\mathbb{R}^N}U_{1,0}^{q+1}\log U_{1,0}dz.
    \end{array}
    \right.
  \end{align}

Our main results state as follows.
\begin{theorem}\label{th}
Let $(\mathcal{M},g)$ be a smooth compact Riemannian manifold of dimension $N\geq 8$, %let $h(x)$ be a $C^1$-function on $\mathcal{M}$, %$\alpha,\beta>0$,
 for any given integer $k\geq2$, set $\bar{\xi^0}=(\xi_1^0,\xi_2^0,\cdots, \xi_k^0)\in \mathcal{M}^k$, let $\xi_j^0$ be a non-degenerate critical point of $a(x)$, and
 \begin{equation}\label{condi1}
   h(\xi^0_j)>\Big(L_2-\frac{L_4}{p+1}-\frac{L_5}{q+1}\Big)\frac{Scal_g(\xi^0_j)}{6NL_3} -
  \Big(L_2-\frac{L_4}{p+1}-\frac{L_5}{q+1}\Big) \frac{\Delta_g a(\xi^0_j)}{2NL_3a(\xi^0_j)}
 \end{equation}
 for any $j=1,2,\cdots,k$.
 Assume that one of the following conditions holds:
 \begin{equation}\label{condi}
 (i)\, \frac{N}{N-2}<p<\frac{N+2}{N-2} \,\,\text{and}\,\, N\geq 8;\,\,\,
 (ii)\, p=\frac{N+2}{N-2}\,\, \text{and}\,\, N\geq10;
 \,\,\,
 (iii)\, 1<p<\frac{N}{N-2}\,\, \text{and} \,\,N\geq 8.
 \end{equation}
 Then for any $\varepsilon>0$ small enough, system \eqref{prob} admits a family of solutions $(u_\varepsilon,v_\varepsilon)$, which blows up and concentrates at $\bar{\xi}^0$ as $\varepsilon\rightarrow0$.
\end{theorem}
In particular, Theorem \ref{th} applies to the case $a=\omega^N$, where $\omega$ is the warping function. For any $j=1,2,\cdots,k$, let $\Gamma_j:=\{\xi^0_j\}\times \mathcal{K}$, and
\begin{equation*}
  \Sigma_\mathfrak{g}(\Gamma_j):=\Big(L_2-\frac{L_4}{p+1}-\frac{L_5}{q+1}\Big)\frac{Scal_g(\xi^0_j)}{6NL_3} -
  \Big(L_2-\frac{L_4}{p+1}-\frac{L_5}{q+1}\Big) \frac{\Delta_\mathfrak{g} \omega(\xi^0_j)}{2L_3\omega(\xi^0_j)}.
\end{equation*}
If $\xi_j^0$ is a critical point of $\omega(x)$, then $\Gamma_1\times\Gamma_2\times\cdots\times \Gamma_k$ is a $m$-dimensional minimal submanifold of $\mathfrak{M}^k$.
Moreover, by \eqref{lap} and Theorem \ref{th}, we immediately have the following result.
\begin{theorem}\label{th1}
For any given integer $k\geq2$, set $\bar{\xi^0}=(\xi_1^0,\xi_2^0,\cdots, \xi_k^0)\in \mathcal{M}^k$, let $\xi_j^0$ be a non-degenerate critical point of $\omega(x)$, if $h$ is invariant with respect to $\mathcal{K}$ and $h(\Gamma_j)> \Sigma_\mathfrak{g}(\Gamma_j)$, $j=1,2,\cdots,k$. Assume that one of the condition \eqref{condi} holds, then for any $\varepsilon>0$ small enough, system \eqref{pro} admits a family of solutions $(u_\varepsilon,v_\varepsilon)$, invariant with respect to $\mathcal{K}$, which blows up and concentrates along $\Gamma_1\times\Gamma_2\times\cdots\times \Gamma_k$ as $\varepsilon\rightarrow0$.
\end{theorem}

\begin{remark}
{\rm By \eqref{condi}, we have $L_i<+\infty$ for any $i=1,2,\cdots,7$, where $L_i$ is given in \eqref{chang}.}
\end{remark}

\begin{remark}
{\rm If $u=v$, $p=q=\frac{N+2}{N-2}$, $\alpha=\beta=1$, and $k=1$, then
Theorems \ref{th} and \ref{th1} are exactly the conclusions obtained in \cite[Theorems 1.2-1.3]{GMP}.
}
\end{remark}

\begin{remark}
{\rm Compared with the work \cite{CW}, which also considers the asymptotically critical Lane-Emden system on Riemannian manifolds, in this paper, we focus on the existence of $k$-peaks solutions concentrated along a minimal submanifold of $\mathfrak{M}^k$.}
\end{remark}

The proof of our result relies on a well known finite dimensional  Lyapunov-Schmidt reduction method, introduced in \cite{BC,FW}. The paper is organized as follows. In Section
\ref{sec2}, we introduce the framework and present some preliminary results. The proof of Theorem \ref{th} is given in Section \ref{sec3}. In Section \ref{sec4}, we perform the finite dimensional reduction, and Section \ref{sec5} is devoted to  the reduced problem. Throughout the paper, $C,C_i$, $i\in \mathbb{N}^+$ denote positive constants possibly different from line to line.

\section{The framework and preliminary results}\label{sec2}

We start with some properties of the least energy solution $(U_{1,0}(z),V_{1,0}(z))$ of \eqref{back}-\eqref{ch} in $\mathbb{R}^N$, which is given by Wang \cite{Wang} and Hulshof and Van der Vorst \cite{HV}.
\begin{lemma}\label{jian1}\cite[Theorem 2]{HV}
Assume that $p,q$ satisfies \eqref{pq} and $1<p\leq \frac{N+2}{N-2}$. If $r\rightarrow+\infty$, there hold
\begin{equation*}
 V_{1,0}(r)=O( r^{2-N}),
\end{equation*}
and
\begin{align*}
 U_{1,0}(r) =\left\{
  \begin{array}{ll}
  O( r^{2-N}),\quad &\text{if $p>\frac{N}{N-2}$;}
 \\ O( r^{2-N}\log r),\quad &\text{if $p=\frac{N}{N-2}$;}\\
 O( r^{2-(N-2)p}),\quad &\text{if $p<\frac{N}{N-2}$.}
    \end{array}
    \right.
  \end{align*}
\end{lemma}

\begin{lemma}\label{jian2}\cite[Lemma 2.2]{KM}
Assume that $p,q$ satisfies \eqref{pq} and $1<p\leq \frac{N+2}{N-2}$. If $r\rightarrow+\infty$, there hold
\begin{equation*}
 V'_{1,0}(r)=O( r^{1-N}),
\end{equation*}
and
\begin{align*}
 U'_{1,0}(r) =\left\{
  \begin{array}{ll}
  O( r^{1-N}),\quad &\text{if $p>\frac{N}{N-2}$;}
 \\ O( r^{1-N}\log r),\quad &\text{if $p=\frac{N}{N-2}$;}\\
 O( r^{1-(N-2)p}),\quad &\text{if $p<\frac{N}{N-2}$.}
    \end{array}
    \right.
  \end{align*}
\end{lemma}

\begin{lemma}\label{nonde}\cite[Theorem 1]{FKP}
Assume that $p,q$ satisfies \eqref{pq},
set
\begin{equation*}
  (\Psi_{1,0}^1,\Phi_{1,0}^1)=\Big(z\cdot \nabla U_{1,0}+\frac{N U_{1,0}}{q+1},z\cdot \nabla V_{1,0}+\frac{N V_{1,0}}{p+1}\Big)
\end{equation*}
and
\begin{equation*}
  (\Psi_{1,0}^l,\Phi_{1,0}^l)=\big(\partial _l U_{1,0},\partial _l V_{1,0} \big),\quad \text{for $l=1,2,\cdots,N$.}
\end{equation*}
Then the space of solutions for the linear system
\begin{align*}
 \left\{
  \begin{array}{ll}
  -\Delta\Psi=pV_{1,0}^{p-1}\Phi,
  \ \  \mbox{in}\ \mathbb{R}^N,\\
   -\Delta\Phi=qU_{1,0}^{q-1}\Psi,
  \ \  \mbox{in}\ \mathbb{R}^N,\\
  (\Psi,\Phi)\in \dot{W}^{2,\frac{p+1}{p}}(\mathbb{R}^N)\times \dot{W}^{2,\frac{q+1}{q}}(\mathbb{R}^N)
    \end{array}
    \right.
  \end{align*}
  is spanned by
  \begin{equation*}
     \big\{(\Psi_{1,0}^0,\Phi_{1,0}^0), (\Psi_{1,0}^1,\Phi_{1,0}^1),\cdots, (\Psi_{1,0}^N,\Phi_{1,0}^N)\big\}.
  \end{equation*}
\end{lemma}
Moreover, we have the following elementary inequality.
\begin{lemma}\cite[Lemma 2.1]{LN}\label{gs}
For any $a>0$, $b$ real, there holds
\begin{align*}
 \big||a+b|^\beta-b^\beta\big|\leq\left\{
  \begin{array}{ll}
  C(\beta)(a^{\beta-1}|b|+|b|^\beta),\quad &\text{if $\beta\geq1$},\\
  C(\beta)\min\big\{a^{\beta-1}|b|,|b|^\beta\big\},\quad &\text{if $0<\beta<1$}.
    \end{array}
    \right.
  \end{align*}
\end{lemma}
Now, we recall some definitions and results about the compact Riemannian manifold $(\mathcal{M},g)$.
\begin{definition}
{\rm Let $(\mathcal{M},g)$ be a smooth compact Riemannian manifold. On the tangent bundle of $\mathcal{M}$, define the exponential map $\exp: T \mathcal{M}\rightarrow \mathcal{M}$, which has the following properties:

(i) $\exp$ is of  class $C^\infty$;

(ii) there exists a constant $r_0>0$ such that $\exp_\xi|_{B(0,r_0)}\rightarrow B_g(\xi,r_0)$ is a diffeomorphism for all $\xi\in \mathcal{M}$.
}
\end{definition}

Fix such $r_0$ in this paper with $r_0<\min\big\{i_g,\min\limits_{j\neq m}\{d_g(\xi^0_j,\xi^0_m)\}\big\}$, where $i_g$ denotes the injectivity radius of $(\mathcal{M},g)$. For
 any $1<s<+\infty$ and $u\in L^s(\mathcal{M})$, we denote the $L^s$-norm of $u$ by $$ \|u\|_s=\Big(\int\limits_{\mathcal{M}}|u|^sd v_g\Big)^{1/{s}},$$
where $d v_g=\sqrt{\det g}dz$ is the volume element on $\mathcal{M}$ associated to the metric $g$.
We
introduce the Banach space
\begin{equation*}
  \mathcal{X}_{p,q}(\mathcal{M})=\dot{W}^{1,p^*}(\mathcal{M})\times \dot{W}^{1,q^*}(\mathcal{M})
\end{equation*}
equipped with the norm
\begin{equation}\label{norm}
  \|(u,v)\|=\Big(\int\limits_{\mathcal{M}}a(x)|\nabla_g u|^{p^*}dv_g\Big)^{1/{p^*}}+\Big(\int\limits_{\mathcal{M}}a(x)|\nabla_g v|^{q^*}dv_g\Big)^{1/{q^*}},
\end{equation}
where
\begin{equation*}
  \frac{1}{p^*}=\frac{p}{p+1}-\frac{1}{N}=\frac{1}{q+1}+\frac{1}{N},\quad \frac{1}{q^*}=\frac{q}{q+1}-\frac{1}{N}=\frac{1}{p+1}+\frac{1}{N}.
\end{equation*}
Denote by $\mathcal{I}^*$ the formal adjoint operator of the embedding $\mathcal{I}:\mathcal{X}_{q,p}(\mathcal{M})\hookrightarrow L^{p+1}(\mathcal{M})\times L^{q+1}(\mathcal{M})$. By the Calder\'{o}n-Zygmund estimate, the operator $\mathcal{I}^*$ maps $L^{\frac{p+1}{p}}(\mathcal{M})\times L^{\frac{q+1}{q}}(\mathcal{M})$ to $\mathcal{X}_{p,q}(\mathcal{M})$. Then we rewrite problem \eqref{prob} as
\begin{equation}\label{repro}
  (u,v)=\mathcal{I}^*\big(a(x)f_\varepsilon(v),a(x)g_\varepsilon(u)\big).
\end{equation}
where $f_\varepsilon(u):=u_+^{p-\alpha\varepsilon}$, $g_\varepsilon(u):=u_+^{q-\beta\varepsilon}$ and $u_+=\max \{u,0\}$.
Moreover, by the Sobolev embedding theorem, we have
\begin{equation}\label{em}
  \|\mathcal{I}^*\big(a(x)f_\varepsilon(v),a(x)g_\varepsilon(u)\big)\|\leq C\|a(x)f_\varepsilon(v)\|_{{\frac{p+1}{p}}}+
  C\|a(x)g_\varepsilon(u)\|_{{\frac{q+1}{q}}}.
\end{equation}

Let $\chi$ be a smooth cutoff function such that $0\leq \chi\leq1$ in $\mathbb{R}^N$, $\chi(z)=1$ if  $z\in B(0,r_0/2)$, and $\chi(z)=0$ if $z\in \mathbb{R}^N\backslash B(0,r_0)$. For any $\xi\in \mathcal{M}$, $\delta>0$ and $\eta\in \mathbb{R}^N$, we define the following functions on $\mathcal{M}$
\begin{equation}\label{fun1}
  \big(W_{\delta,\xi,\eta}(x),H_{\delta,\xi,\eta}(x)\big):=\big(\chi(d_g(x,\xi))\delta^{-\frac{N}{q+1}}U_{1,0}(\delta^{-1}\exp_\xi^{-1}(x)-\eta), \chi(d_g(x,\xi))\delta^{-\frac{N}{p+1}}V_{1,0}(\delta^{-1}\exp_\xi^{-1}(x)-\eta)\big)
\end{equation}
and
\begin{equation*}
  \big(\Psi^i_{\delta,\xi,\eta}(x),\Phi^i_{\delta,\xi,\eta}(x)\big):=\big(\chi(d_g(x,\xi))\delta^{-\frac{N}{q+1}}\Psi_{1,0}^i(\delta^{-1}\exp_\xi^{-1}(x)-\eta), \chi(d_g(x,\xi))\delta^{-\frac{N}{p+1}}\Phi_{1,0}^i(\delta^{-1}\exp_\xi^{-1}(x)-\eta)\big),
\end{equation*}
for $i=0,1,\cdots, N$,
where $(\Psi_{1,0}^i,\Phi_{1,0}^i)$ is given in Lemma \ref{nonde}.

For any $\varepsilon>0$ and $\bar{t}=(t_1,t_2,\cdots,t_k)\in (\mathbb{R}^+)^k$, we set
\begin{equation}\label{solu'}
  \bar{\delta}=(\delta_1,\delta_2,\cdots,\delta_k)\in (\mathbb{R}^+)^k,\quad \delta_j=\sqrt{\varepsilon t_j},\quad \varrho_1<t_j<\frac{1}{\varrho_1},\quad \bar{\eta}=(\eta_1,\eta_2,\cdots,\eta_k)\in (\mathbb{R}^N)^k
\end{equation}
for fixed small $\varrho_1>0$.
Let $\mathcal{Y}_{\bar{\delta},\bar{\xi^0},\bar{\eta}}$ and $\mathcal{Z}_{\bar{\delta},\bar{\xi^0},\bar{\eta}}$ be two subspaces of $\mathcal{X}_{p,q}(\mathcal{M})$ given as
\begin{equation*}
  \mathcal{Y}_{\bar{\delta},\bar{\xi^0},\bar{\eta}}=span\Big\{\big(\Psi^i_{\delta_j,\xi^0_j,\eta_j},\Phi^i_{\delta_j,\xi^0_j,\eta_j}\big):i=0,1,\cdots, N \,\, \text{and}\,\, j=1,2,\cdots,k\Big\}
\end{equation*}
and
\begin{equation*}
  \mathcal{Z}_{\bar{\delta},\bar{\xi^0},\bar{\eta}}=\Big\{(\Psi,\Phi)\in \mathcal{X}_{p,q}(\mathcal{M}): \big\langle(\Psi,\Phi),\big(\Psi^i_{\delta_j,\xi^0_j,\eta_j},\Phi^i_{\delta_j,\xi^0_j,\eta_j}\big)\big\rangle_h=0 \,\, \text{ for $i=0,1,\cdots, N$ and $j=1,2,\cdots,k$} \Big\},
\end{equation*}
where
\begin{equation*}
  \langle(u,v),(\varphi,\psi)\rangle_h=\int\limits_{\mathcal{M}}a(x)(\nabla _g u \cdot \nabla _g \psi+\nabla _g v \cdot \nabla _g \varphi )d v_g+\int\limits_{\mathcal{M}}a(x)(hu \psi +hv \varphi)dv_g
\end{equation*}
for any $(u,v),(\varphi,\psi)\in \mathcal{X}_{p,q}(\mathcal{M})$.

\begin{lemma}\label{topo}
There exists $\varepsilon_0>0$ such that for any $\varepsilon\in (0,\varepsilon_0)$, $\mathcal{X}_{p,q}(\mathcal{M})=\mathcal{Y}_{\bar{\delta},\bar{\xi^0},\bar{\eta}}\oplus \mathcal{Z}_{\bar{\delta},\bar{\xi^0},\bar{\eta}}$.
\end{lemma}
\begin{proof}
We shall prove that for any $(\Psi,\Phi)\in \mathcal{X}_{p,q}(\mathcal{M})$, there exists unique pair $(\Psi_0,\Phi_0)\in \mathcal{Z}_{\bar{\delta},\bar{\xi^0},\bar{\eta}}$ and coefficients $c_{01},c_{02},\cdots,c_{0k}, c_{11},c_{12},\cdots,c_{1k},\cdots,c_{N1},c_{N2},\cdots,c_{Nk}$ such that
\begin{equation}\label{coe}
  (\Psi,\Phi)=(\Psi_0,\Phi_0)+\sum\limits_{l=0}^N\sum\limits_{m=1}^kc_{lm}(\Psi^l_{\delta_m,\xi^0_m,\eta_m},\Phi^l_{\delta_m,\xi^0_m,\eta_m}).
\end{equation}
The requirement that $(\Psi_0,\Phi_0)\in \mathcal{Z}_{\bar{\delta},\bar{\xi^0},\bar{\eta}}$ is equivalent to demanding
\begin{align}\label{dem}
  &\int\limits_{\mathcal{M}}\big(a(x)\nabla _g \Psi \cdot\nabla _g \Phi^i_{\delta_j,\xi^0_j,\eta_j}+a(x)\nabla _g \Phi\cdot\nabla _g \Psi^i_{\delta_j,\xi^0_j,\eta_j}+a(x)h\Psi \Phi^i_{\delta_j,\xi^0_j,\eta_j}+a(x)h\Phi \Psi^i_{\delta_j,\xi^0_j,\eta_j}\big)d v_g \nonumber\\
  =&\sum\limits_{l=0}^N\sum\limits_{m=1}^kc_{lm}\int\limits_{\mathcal{M}}\big (a(x)\nabla _g \Psi^l_{\delta_m,\xi^0_m,\eta_m} \cdot\nabla _g \Phi^i_{\delta_j,\xi^0_j,\eta_j}+a(x)\nabla _g \Phi^l_{\delta_m,\xi^0_m,\eta_m} \cdot\nabla _g \Psi^i_{\delta_j,\xi^0_j,\eta_j}\nonumber\\
  &+a(x)h\Psi^l_{\delta_m,\xi^0_m,\eta_m} \Phi^i_{\delta_j,\xi^0_j,\eta_j}+a(x)h\Phi^l_{\delta_m,\xi^0_m,\eta_m}  \Psi^i_{\delta_j,\xi^0_j,\eta_j}\big) d v_g
\end{align}
for any $i=0,1,\cdots,N$ and $j=1,2,\cdots,k$.

We estimate the integral on the right-hand side of \eqref{dem}. By standard properties of the exponential map, there exists $C>0$ such that for any $\xi\in \mathcal{M}$, $\delta>0$, $z\in B(0,r_0/\delta)$, $\eta\in \mathbb{R}^N$ and $i,j,k\in \mathbb{N}^+$, there hold
\begin{equation*}
  |g_{\delta,\xi,\eta}^{ij}(z)-Eucl^{ij}|\leq C\delta^2|z+\eta|^2,\quad \text{and}\quad |g_{\delta,\xi,\eta}^{ij}(z) (\Gamma_{\delta,\xi,\eta})_{ij}^k(z)|\leq C\delta^2|z+\eta|,
\end{equation*}
where $g_{\delta,\xi,\eta}(z)=\exp_\xi^*g(\delta z+\delta \eta)$ and $(\Gamma_{\delta,\xi,\eta})_{ij}^k$ stand for the Christoffel symbols of the metric $g_{\delta,\xi,\eta}$.
Taking into account that there holds
\begin{equation*}
  \Delta _{g_{\delta,\xi,\eta}}=g_{\delta,\xi,\eta}^{ij}\Big(\frac{\partial ^2}{\partial x_i \partial x_j}-(\Gamma_{\delta,\xi,\eta})_{ij}^k\frac{\partial }{\partial x_k}\Big),
\end{equation*}
by Lemma \ref{nonde} and $d_g(\xi_j^0,\xi_m^0)>r_0$ for any $j\neq m$,  we have
\begin{align}\label{gu1}
  &\int\limits_{\mathcal{M}}a(x)\nabla _g \Psi^l_{\delta_m,\xi^0_m,\eta_m} \cdot\nabla _g \Phi^i_{\delta_j,\xi^0_j,\eta_j}d v_g=\delta_{jm}\int\limits_{\mathcal{M}}a(x)\nabla _g \Psi^l_{\delta_j,\xi^0_j,\eta_j} \cdot\nabla _g \Phi^i_{\delta_j,\xi^0_j,\eta_j}d v_g\nonumber\\
  =&\delta_{jm}\int\limits_{B(0,r_0/\delta_j)}a_{\delta_j,\xi^0_j,\eta_j}\nabla _{g_{\delta_j,\xi^0_j,\eta_j}} (\chi_{\delta_j,\eta_j}\Psi^l_{1,0})\cdot  \nabla _{g_{\delta_j,\xi^0_j,\eta_m}} (\chi_{\delta_j,\eta_j}\Phi^i_{1,0})dz \nonumber\\
  =&p\delta_{jm}\int\limits_{B(0,r_0/{\delta_j})}a_{\delta_j,\xi^0_j,\eta_j}\chi^2_{\delta_j,\eta_j} V_{1,0}^{p-1}\Phi_{1,0}^l\Phi_{1,0}^idz+O(\delta_j^2) \nonumber \\=&p\delta_{il}\delta_{jm}\int\limits_{B(0,r_0/{\delta_j})}a_{\delta_j,\xi^0_j,\eta_j}\chi^2_{\delta_j,\eta_j} V_{1,0}^{p-1}(\Phi_{1,0}^i)^2dz+O(\delta_j^2),
\end{align}
and
\begin{align}\label{gu2}
  &\int\limits_{\mathcal{M}}a(x)h\Psi^l_{\delta_m,\xi^0_m,\eta_m} \Phi^i_{\delta_j,\xi^0_j,\eta_j}d v_g=\delta_{jm}\int\limits_{\mathcal{M}}a(x)h\Psi^l_{\delta_j,\xi^0_j,\eta_j} \Phi^i_{\delta_j,\xi^0_j,\eta_j}d v_g
  \nonumber\\=&\delta_{jm}\delta_j^2 \int\limits_{B(0,r_0/\delta_j)}a_{\delta_j,\xi^0_j,\eta_j}h_{\delta_j,\xi^0_j,\eta_j}\chi^2_{\delta_j,\eta_j} \Psi_{1,0}^l\Phi_{1,0}^idz \nonumber\\
  =&-\delta_{jm}\delta_j^2 \int\limits_{B(0,r_0/\delta_j)}a_{\delta_j,\xi^0_j,\eta_j}h_{\delta_j,\xi^0_j,\eta_j}\chi^2_{\delta_j,\eta_j}\frac{\Delta \Phi_{1,0}^l}{qU_{1,0}^{q-1}}\Phi_{1,0}^idz+o(\delta_j^2)\nonumber\\
  =&\delta_{il}  \delta_{jm}  \delta_j^2 \int\limits_{B(0,r_0/\delta_j)}a_{\delta_j,\xi^0_j,\eta_j}h_{\delta_j,\xi^0_j,\eta_j} \chi^2_{\delta_j,\eta_j} \frac{(\nabla\Phi_{1,0}^i)^2}{qU_{1,0}^{q-1}}dz+o(\delta_j^2),
\end{align}
where
 $\chi_{\delta_j,\eta_j}(z)=\chi({\delta_jz+\delta_j \eta_j})$, $a_{\delta_j,\xi^0_j,\eta_j}(z)=a(\exp_{\xi^0_j}(\delta_j z+\delta_j\eta_j))$ and $h_{\delta_j,\xi^0_j,\eta_j}(z)=h(\exp_{\xi^0_j}(\delta_j z+\delta_j\eta_j))$.
Similarly, we have
\begin{align}\label{gu3}
  \int\limits_{\mathcal{M}}a(x)\nabla _g \Phi^l_{\delta_m,\xi^0_m,\eta_m} \cdot\nabla _g \Psi^i_{\delta_j,\xi^0_j,\eta_j}d v_g=q\delta_{il}\delta_{jm}\int\limits_{B(0,r_0/{\delta_j})}a_{\delta_j,\xi^0_j,\eta_j}\chi^2_{\delta_j,\eta_j} U_{1,0}^{q-1}(\Psi_{1,0}^i)^2dz+O(\delta_j^2),
\end{align}
and
\begin{align}\label{gu4}
  \int\limits_{\mathcal{M}}a(x)h\Phi^l_{\delta_m,\xi^0_m,\eta_m} \Psi^i_{\delta_j,\xi^0_j,\eta_j}d v_g=\delta_{il}\delta_{jm}\delta_j^2  \int\limits_{B(0,r_0/\delta_j)}a_{\delta_j,\xi^0_j,\eta_j}h_{\delta_j,\xi^0_j,\eta_j} \chi^2_{\delta_j,\eta_j}\frac{(\nabla\Psi_{1,0}^i)^2}{pV_{1,0}^{p-1}}dz+o(\delta_j^2).
\end{align}
By plugging \eqref{gu1}-\eqref{gu4} into \eqref{dem}, we can see that the coefficients $c_{lm}$ are uniquely determined for $l=0,1,\cdots,N$ and $m=1,2,\cdots,k$. By virtue of \eqref{coe}, so is $(\Psi_0,\Phi_0)$.

On the other hand, $\mathcal{Y}_{\bar{\delta},\bar{\xi^0},\bar{\eta}}$ and $\mathcal{Z}_{\bar{\delta},\bar{\xi^0},\bar{\eta}}$ are clearly closed subspaces of $\mathcal{X}_{p,q}(\mathcal{M})$, Therefore, they are topological complements of each other.
\end{proof}

\section{Scheme of the proof of Theorem \ref{th}}\label{sec3}

We look for solutions of system \eqref{prob}, or equivalently of  \eqref{repro}, of the form
\begin{equation}\label{solu}
  (u_\varepsilon,v_\varepsilon)=\big(\mathcal{W}_{\bar{\delta},\bar{\xi^0},\bar{\eta}}+\Psi_{\varepsilon,\bar{t},\bar{\xi^0},\bar{\eta}}
  ,\mathcal{H}_{\bar{\delta},\bar{\xi^0},\bar{\eta}}+\Phi_{\varepsilon,\bar{t},\bar{\xi^0},\bar{\eta}}\big),
\end{equation}
with
\begin{equation*}
\mathcal{W}_{\bar{\delta},\bar{\xi^0},\bar{\eta}}=\sum\limits_{j=0}^kW_{\delta_j,\xi^0_j,\eta_j}\ \ \  \mbox{and}\ \  \mathcal{H}_{\bar{\delta},\bar{\xi^0},\bar{\eta}}=\sum\limits_{j=0}^kH_{\delta_j,\xi^0_j,\eta_j},
\end{equation*}
where $\bar{\delta}$ is as in \eqref{solu'}, $\big(W_{\delta_j,\xi^0_j,\eta_j},H_{\delta_j,\xi^0_j,\eta_j}\big)$ is as in \eqref{fun1}, and $\big(\Psi_{\varepsilon,\bar{t},\bar{\xi^0},\bar{\eta}},\Phi_{\varepsilon,\bar{t},\bar{\xi^0},\bar{\eta}}\big)\in \mathcal{Z}_{\bar{\delta},\bar{\xi^0},\bar{\eta}}$. By Lemma \ref{topo}, we know $\mathcal{X}_{p,q}(\mathcal{M})=\mathcal{Y}_{\bar{\delta},\bar{\xi^0},\bar{\eta}}\oplus \mathcal{Z}_{\bar{\delta},\bar{\xi^0},\bar{\eta}}$. Then we define the projections $\Pi_{\bar{\delta},\bar{\xi^0},\bar{\eta}}$ and $\Pi_{\bar{\delta},\bar{\xi^0},\bar{\eta}}^\bot$ of the Sobolev space $\mathcal{X}_{p,q}(\mathcal{M})$ onto $\mathcal{Y}_{\bar{\delta},\bar{\xi^0},\bar{\eta}}$ and $\mathcal{Z}_{\bar{\delta},\bar{\xi^0},\bar{\eta}}$ respectively. Therefore, we have to solve the couples of equations
\begin{equation}\label{tou1}
  \Pi_{\bar{\delta},\bar{\xi^0},\bar{\eta}}\Big[\big(\mathcal{W}_{\bar{\delta},\bar{\xi^0},\bar{\eta}}+\Psi_{\varepsilon,\bar{t},\bar{\xi^0},\bar{\eta}}
  ,\mathcal{H}_{\bar{\delta},\bar{\xi^0},\bar{\eta}}+\Phi_{\varepsilon,\bar{t},\bar{\xi^0},\bar{\eta}}\big)
  -\mathcal{I}^*\big(a(x)f_\varepsilon(\mathcal{H}_{\bar{\delta},\bar{\xi^0},\bar{\eta}}+\Phi_{\varepsilon,\bar{t},\bar{\xi^0},\bar{\eta}}),a(x)g_\varepsilon(\mathcal{W}_{\bar{\delta},\bar{\xi^0},\bar{\eta}}+\Psi_{\varepsilon,\bar{t},\bar{\xi^0},\bar{\eta}})\big)\Big]=0,
\end{equation}
and
\begin{equation}\label{tou2}
  \Pi_{\bar{\delta},\bar{\xi^0},\bar{\eta}}^\bot\Big[\big(\mathcal{W}_{\bar{\delta},\bar{\xi^0},\bar{\eta}}+\Psi_{\varepsilon,\bar{t},\bar{\xi^0},\bar{\eta}}
  ,\mathcal{H}_{\bar{\delta},\bar{\xi^0},\bar{\eta}}+\Phi_{\varepsilon,\bar{t},\bar{\xi^0},\bar{\eta}}\big)
  -\mathcal{I}^*\big(a(x)f_\varepsilon(\mathcal{H}_{\bar{\delta},\bar{\xi^0},\bar{\eta}}+\Phi_{\varepsilon,\bar{t},\bar{\xi^0},\bar{\eta}}),a(x)g_\varepsilon(\mathcal{W}_{\bar{\delta},\bar{\xi^0},\bar{\eta}}+\Psi_{\varepsilon,\bar{t},\bar{\xi^0},\bar{\eta}})\big)\Big]=0.
\end{equation}

The first step in the proof consists in solving equation \eqref{tou2}. This requires Proposition \ref{propo1} below, whose proof is postponed to Section \ref{sec4}.
\begin{proposition}\label{propo1}
Under the assumptions of Theorem \ref{th}, if $\bar{\delta}$ is as in \eqref{solu'},
equation \eqref{tou2} admits a unique solution $(\Psi_{\varepsilon,\bar{t},\bar{\xi^0},\bar{\eta}},\Phi_{\varepsilon,\bar{t},\bar{\xi^0},\bar{\eta}})$ in $\mathcal{Z}_{\bar{\delta},\bar{\xi^0},\bar{\eta}}$, which is continuously differentiable with respect to $\bar{t}$ and $\bar{\eta}$, such that
\begin{equation*}
  \|(\Psi_{\varepsilon,\bar{t},\bar{\xi^0},\bar{\eta}},\Phi_{\varepsilon,\bar{t},\bar{\xi^0},\bar{\eta}})\|\leq C\varepsilon|\log \varepsilon|.
\end{equation*}
\end{proposition}
We now introduce the energy functional $\mathcal{J}_\varepsilon$ defined on $\mathcal{X}_{p,q}(\mathcal{M})$ by
\begin{align*}
  \mathcal{J}_\varepsilon(u,v)=&\int\limits_{\mathcal{M}}a(x)\nabla _g u\cdot \nabla _g vd v_g+\int\limits_{\mathcal{M}}a(x)huvd v_g\\
  &-\frac{1}{p+1-\alpha\varepsilon}\int\limits_{\mathcal{M}}
  a(x)v^{p+1-\alpha\varepsilon}d v_g-\frac{1}{q+1-\beta\varepsilon}\int\limits_{\mathcal{M}}a(x)
  u^{q+1-\beta\varepsilon}d v_g.
\end{align*}
It is clear that the critical points of $\mathcal{J}_\varepsilon$ are the solutions of system \eqref{prob}. Moreover,
\begin{align*}
  \mathcal{J}'_\varepsilon(u,v)(\varphi,\psi)=&\int\limits_{\mathcal{M}}a(x)(\nabla _g u\cdot \nabla _g \psi+\nabla _g v\cdot \nabla _g \varphi )d v_g+\int\limits_{\mathcal{M}}a(x)(hu\psi+hv\varphi)d v_g\\
  &-\int\limits_{\mathcal{M}}
  a(x)u^{q-\beta\varepsilon} \varphi d v_g-\int\limits_{\mathcal{M}}
  a(x)v^{p-\alpha\varepsilon} \psi d v_g,
\end{align*}
for any $(u,v),(\varphi,\psi)\in \mathcal{X}_{p,q}(\mathcal{M})$.
We also define the functional $\widetilde{\mathcal{J}}_\varepsilon:(\mathbb{R}^+)^k\times (\mathbb{R}^N)^k\rightarrow\mathbb{R}$
\begin{equation}\label{defj}
  \widetilde{\mathcal{J}}_\varepsilon(\bar{t},\bar{\eta}):=\mathcal{J}_\varepsilon\big(\mathcal{W}_{\bar{\delta},\bar{\xi^0},\bar{\eta}}+\Psi_{\varepsilon,\bar{t},\bar{\xi^0},\bar{\eta}}
  ,\mathcal{H}_{\bar{\delta},\bar{\xi^0},\bar{\eta}}+\Phi_{\varepsilon,\bar{t},\bar{\xi^0},\bar{\eta}}\big),
\end{equation}
where $\big(\mathcal{W}_{\bar{\delta},\bar{\xi^0},\bar{\eta}},\mathcal{H}_{\bar{\delta},\bar{\xi^0},\bar{\eta}}\big)$ is as \eqref{solu}, $\big(\Psi_{\varepsilon,\bar{t},\bar{\xi^0},\bar{\eta}},\Phi_{\varepsilon,\bar{t},\bar{\xi^0},\bar{\eta}}\big)$ is given in Proposition \ref{propo1}.
\begin{definition}
{\rm For a given $C^1$-function $\varphi_\varepsilon$, we say that the estimate $\varphi_\varepsilon=o(\varepsilon)$ is $C^1$-uniform if there hold $\varphi_\varepsilon=o(\varepsilon)$ and $\nabla \varphi_\varepsilon=o(\varepsilon)$ as $\varepsilon\rightarrow0$.}
\end{definition}
We solve equation \eqref{tou1} in Proposition \ref{propo2} below whose proof is postponed to Section \ref{sec5}.
\begin{proposition}\label{propo2}
(i) Under the assumptions of Theorem \ref{th}, if $\bar{\delta}$ is as in \eqref{solu'}, for any $\varepsilon>0$ small enough, if $(\bar{t},\bar{\eta})$ is a critical point of the functional $\widetilde{\mathcal{J}}_\varepsilon$, then $\big(\mathcal{W}_{\bar{\delta},\bar{\xi^0},\bar{\eta}}+\Psi_{\varepsilon,\bar{t},\bar{\xi^0},\bar{\eta}}
  ,\mathcal{H}_{\bar{\delta},\bar{\xi^0},\bar{\eta}}+\Phi_{\varepsilon,\bar{t},\bar{\xi^0},\bar{\eta}}\big)$ is a  solution of system \eqref{prob}, or equivalently of  \eqref{repro}.

(ii) Under the assumptions of Theorem \ref{th},
there holds
\begin{align*}
  \widetilde{\mathcal{J}}_\varepsilon(\bar{t},\bar{\eta})=&
  \sum\limits_{j=1}^ka(\xi^0_j)\big[\frac{2}{N}L_1+c_1\varepsilon-c_2\varepsilon \log \varepsilon+\Psi(t_j,\eta_j)\varepsilon\big]+o(\varepsilon)
\end{align*}
as $\varepsilon\rightarrow0$, $C^1$-uniformly with respect to $\bar{\eta}$ in $(\mathbb{R}^N)^k$ and to $\bar{t}$ in compact subsets of $(\mathbb{R}^+)^k$, where
\begin{equation}\label{defc1c2}
c_1=\Big(\frac{L_6\alpha}{p+1}+\frac{L_7\beta}{q+1}\Big)
  -\Big(\frac{\alpha}{(p+1)^2}+\frac{\beta}{(q+1)^2}\Big)
  L_1,\ \ \ c_2=\frac{NL_1}{2}\Big(\frac{\alpha}{(p+1)^2}
  +\frac{\beta}{(q+1)^2}\Big),
\end{equation}
and
\begin{align}\label{deffai}
  \Psi(t_j,\eta_j)=&\Big\{L_3h(\xi^0_j)-\Big(L_2-\frac{L_4}{p+1}-\frac{L_5}{q+1}\Big)\frac{Scal_g(\xi^0_j)}{6N} +
  \Big(L_2-\frac{L_4}{p+1}-\frac{L_5}{q+1}\Big) \frac{\Delta_g a(\xi^0_j)}{2Na(\xi^0_j)}\nonumber\\&+\frac{L_1D_g^2a(\xi^0_j)[\eta_j,\eta_j]}{Na(\xi^0_j)}
  \Big\}t_j-c_2\log t_j,
\end{align}
with $L_i$ are positive constants given in \eqref{chang}, $i=1,2,\cdots,7$.
\end{proposition}

We now prove Theorem \ref{th} by using Propositions \ref{propo1} and \ref{propo2}.\\ %together with the assumption that the function $\Lambda(\xi)$ has a $C^1$-stable critical  point $\xi_0$ with $\Lambda(\xi_0)>0$.\\
{\bf Proof of Theorem \ref{th}.} Define $\widetilde{\mathcal{J}}:(\mathbb{R}^+)^k\times (\mathbb{R}^N)^k\rightarrow \mathbb{R}$ by
\begin{equation*}
  \widetilde{\mathcal{J}}(\bar{t},\bar{\eta}):=\sum\limits_{j=1}^k\Phi(t_j,\eta_j),\quad \text{with $\Phi(t_j,\eta_j)=\frac{a(\xi^0_j)\Psi(t_j,\eta_j)}{L_3}$},
\end{equation*}
where $L_3>0$ is given in \eqref{chang}.
Since $\xi_j^0$ is a non-degenerate critical  point of $a(x)$ with \eqref{condi1} holds, set
\begin{equation*}
  \Theta(\xi^0_j):=h(\xi^0_j)-\Big(L_2-\frac{L_4}{p+1}-\frac{L_5}{q+1}\Big)\frac{Scal_g(\xi^0_j)}{6NL_3} +
  \Big(L_2-\frac{L_4}{p+1}-\frac{L_5}{q+1}\Big) \frac{\Delta_g a(\xi^0_j)}{2NL_3a(\xi^0_j)}\quad \text{and}\quad t^0_j:=\frac{c_2}{\Theta(\xi^0_j)},
\end{equation*}
then $t_j^0>0$ and $(t_j^0,0)$ is a non-degenerate critical point of $\Phi(t_j,\eta_j)$, $j=1,2,\cdots,k$. Hence $(\bar{t^0},0)$ is a non-degenerate critical point of $\widetilde{\mathcal{J}}(\bar{t},\bar{\eta})$.
Using Proposition \ref{propo2}, we have
\begin{equation*}
  \big|\partial _{\bar{t}}\big(\varepsilon^{-1}L_3^{-1}\widetilde{\mathcal{J}}_\varepsilon-\widetilde{\mathcal{J}}\big)\big|+\big|\partial_{\bar{\eta}} \big(\varepsilon^{-1}L_3^{-1}\widetilde{\mathcal{J}}_\varepsilon-\widetilde{\mathcal{J}}\big)\big|\rightarrow0,
\end{equation*}
as $\varepsilon\rightarrow0$, uniformly with respect to $\bar{\eta}$ in $(\mathbb{R}^N)^k$ and to $\bar{t}$ in compact subsets of $(\mathbb{R}^+)^k$. It follows that there exists a family of critical points $(\bar{t^\varepsilon},\bar{\eta^\varepsilon})$ of $\widetilde{\mathcal{J}}_\varepsilon$ converging to $(\bar{t^0},0)$ as $\varepsilon\rightarrow0$. Using Proposition \ref{propo2} again, we can see that the function $(u_\varepsilon,v_\varepsilon)=\big(\mathcal{W}_{\bar{\delta^\varepsilon},\bar{\xi^0},\bar{\eta^\varepsilon}}+\Psi_{\varepsilon,\bar{t^\varepsilon},\bar{\xi^0},\bar{\eta^\varepsilon}}
  ,\mathcal{H}_{\bar{\delta^\varepsilon},\bar{\xi^0},\bar{\eta^\varepsilon}}+\Phi_{\varepsilon,\bar{t^\varepsilon},\bar{\xi^0},\bar{\eta^\varepsilon}}\big)$ is a pair of solutions of system \eqref{prob} for any $\varepsilon>0$ small enough, where $\bar{\delta^\varepsilon}$ is as in \eqref{solu'}. Moreover,
$(u_\varepsilon,v_\varepsilon)$ blows up and concentrates at $\bar{\xi^0}$ at $\varepsilon\rightarrow0$. This ends the proof. \qed

\section{Proof of Proposition \ref{propo1}}\label{sec4}

This section is devoted to the proof of Proposition \ref{propo1}. For any $\varepsilon>0$, $\bar{t}\in (\mathbb{R}^+)^k$, and $\bar{\eta}\in (\mathbb{R}^N)^k$, if $\bar{\delta}$ is as in \eqref{solu'}, we introduce the map
$\mathcal{L}_{\varepsilon,\bar{t},\bar{\xi^0},\bar{\eta}}:\mathcal{Z}_{\bar{\delta},\bar{\xi^0},\bar{\eta}}\rightarrow \mathcal{Z}_{\bar{\delta},\bar{\xi^0},\bar{\eta}}$ defined by
\begin{equation}\label{defl}
\mathcal{L}_{\varepsilon,\bar{t},\bar{\xi^0},\bar{\eta}}(\Psi,\Phi)=
  \Pi_{\bar{\delta},\bar{\xi^0},\bar{\eta}}^\bot\Big[
  (\Psi,\Phi)-\mathcal{I}^*\big(a(x)f'_\varepsilon(\mathcal{H}_{\bar{\delta},\bar{\xi^0},\bar{\eta}})\Phi,a(x)g'_\varepsilon(\mathcal{W}_{\bar{\delta},\bar{\xi^0},\bar{\eta}})\Psi\big)
  \Big].
\end{equation}
It's easy to check that $\mathcal{L}_{\varepsilon,\bar{t},\bar{\xi^0},\bar{\eta}}$ is well defined in $\mathcal{Z}_{\bar{\delta},\bar{\xi^0},\bar{\eta}}$. Next, we prove the invertibility of this map.
\begin{lemma}\label{line}
Under the assumptions of Theorem \ref{th}, if $\bar{\delta}$ is as in \eqref{solu'}, then for any $\varepsilon>0$ small enough, and $(\Psi,\Phi)\in \mathcal{Z}_{\bar{\delta},\bar{\xi^0},\bar{\eta}}$, there holds
\begin{equation*}
  \|\mathcal{L}_{\varepsilon,\bar{t},\bar{\xi^0},\bar{\eta}}(\Psi,\Phi)\|\geq C\|(\Psi,\Phi)\|,
\end{equation*}
where $\mathcal{L}_{\varepsilon,\bar{t},\bar{\xi^0},\bar{\eta}}(\Psi,\Phi)$ is as in \eqref{defl}.
\end{lemma}
\begin{proof}
We assume by contradiction that there exist a sequence $\varepsilon_n\rightarrow 0$ as $n\rightarrow+\infty$, $\bar{t_n}=(t_{1n},t_{2n},\cdots,t_{kn})\in (\mathbb{R}^+)^k$, $\bar{\eta_n}=(\eta_{1n},\eta_{2n},\cdots,\eta_{kn})\in (\mathbb{R}^N)^k$, and a sequence of functions $(\Psi_n,\Phi_n)\in\mathcal{Z}_{\bar{\delta_n},\bar{\xi^0_n},\bar{\eta_n}}$ such that
\begin{equation*}
  \|(\Psi_n,\Phi_n)\|=1,\quad \|\mathcal{L}_{\varepsilon_n,\bar{t_n},\bar{\xi^0_n},\bar{\eta_n}}(\Psi_n,\Phi_n)\|\rightarrow0, \quad \text{as $n\rightarrow+\infty$}.
\end{equation*}

{\bf Step 1:}
For any $n\in \mathbb{N}^+$ and $j=1,2,\cdots,k$, let
\begin{align*}
  &(\widetilde{\Psi}_n(z),\widetilde{\Phi}_n(z))\\
  =&\big(\chi(\delta_{jn}z+\delta_{jn}\eta_{jn})\delta_{jn}^{\frac{N}{q+1}}\Psi_n(\exp _{\xi^0_{jn}}(
  \delta_{jn} z+\delta_{jn} \eta_{jn})),\chi(\delta_{jn}z+\delta_{jn}\eta_{jn})\delta_{jn}^{\frac{N}{p+1}}\Phi_n(\exp _{\xi^0_{jn}}(\delta_{jn} z+\delta_{jn}\eta_{jn}))\big),
\end{align*}
where $\chi$ is a cutoff function as in \eqref{fun1}. A direct computations shows %that
 \begin{align*}
  \big\|\nabla \widetilde{\Psi}_n\big\|^{p^*}_{L^{p^*}(\mathbb{R}^N)}&\leq \int\limits_{B(0,r_0/\delta_{jn})}| \delta_{jn}^{\frac{N}{q+1}}\nabla\Psi_n(\exp _{\xi^0_{jn}}(\delta_{jn} z+\delta_{jn}\eta))|^{p^*}dz\\
  &=\int\limits_{B(0,r_0)}\delta_{jn}^{-N}|\delta_{jn}^{1+\frac{N}{q+1}}\nabla\Psi_n(\exp _{\xi_{jn}}(z+\eta))|^{p^*}dz\\
  &=\int\limits_{B_g(\xi_{jn},r_0)}|\nabla_g\Psi_n|^{p^*}d v_g=\int\limits_{\mathcal{M}}|\nabla_g\Psi_n|^{p^*}d v_g\leq C,
\end{align*}
and
\begin{align*}
  \big\|\nabla \widetilde{\Phi}_n\big\|^{q^*}_{L^{q^*}(\mathbb{R}^N)}&\leq \int\limits_{B(0,r_0/\delta_{jn})}| \delta_{jn}^{\frac{N}{p+1}}\nabla\Phi_n(\exp _{\xi^0_{jn}}(\delta_{jn} z+\delta_{jn}\eta))|^{q^*}dz\\&=\int\limits_{B(0,r_0)}\delta_{jn}^{-N}|\delta_{jn}^{1+\frac{N}{p+1}}\nabla\Phi_n(\exp _{\xi^0_{jn}}(z+\eta))|^{q^*}dz\\
  &=\int\limits_{B_g(\xi_{jn},r_0)}|\nabla_g\Phi_n|^{q^*}d v_g=\int\limits_{\mathcal{M}}|\nabla_g\Phi_n|^{q^*}d v_g\leq C.
\end{align*}
Hence, $(\widetilde{\Psi}_n,\widetilde{\Phi}_n)$ is bounded in $ \dot{W}^{1,p^*}(\mathbb{R}^N)\times \dot{W}^{1,q^*}(\mathbb{R}^N)$.
Up to a subsequence, there exists $(\widetilde{\Psi},\widetilde{\Phi})\in \dot{W}^{1,p^*}(\mathbb{R}^N)\times \dot{W}^{1,q^*}(\mathbb{R}^N)$ such that $(\widetilde{\Psi}_n,\widetilde{\Phi}_n)\rightharpoonup (\widetilde{\Psi},\widetilde{\Phi})$ in $ \dot{W}^{1,p^*}(\mathbb{R}^N)\times \dot{W}^{1,q^*}(\mathbb{R}^N)$, $(\widetilde{\Psi}_n,\widetilde{\Phi}_n)\rightarrow (\widetilde{\Psi},\widetilde{\Phi})$ in $L_{loc}^{s}(\mathbb{R}^N)\times L_{loc}^{t}(\mathbb{R}^N)$ for any $(s,t)\in [1,q+1]\times[1,p+1]$, and $(\widetilde{\Psi}_n,\widetilde{\Phi}_n)\rightarrow (\widetilde{\Psi},\widetilde{\Phi})$ almost everywhere in $\mathbb{R}^N$.
   For convenience, we denote $(P_n,K_n)=\mathcal{L}_{\varepsilon_n,\bar{t_n},\bar{\xi^0_n},\bar{\eta_n}}(\Psi_n,\Phi_n)$. Furthermore, by $(P_n,K_n) \in \mathcal{Z}_{\bar{\delta_{n}},\bar{\xi^0_n},\bar{\eta_n}}$, there exist $c_{1n}^0,c_{2n}^0,\cdots,c_{kn}^0$, $c_{1n}^1,c_{2n}^1,\cdots,c_{kn}^1$, $\cdots$, $c_{1n}^ N,c_{2n}^N,\cdots,c_{kn}^N$ such that
  \begin{align}\label{line1}
 & (\Psi_n,\Phi_n)-\mathcal{I}^*\big(a(x)f'_{\varepsilon_n}(\mathcal{H}_{\bar{\delta_n},\bar{\xi^0_n},\bar{\eta_n}})\Phi_n,a(x)g'_{\varepsilon_n}(\mathcal{W}_{\bar{\delta_n},\bar{\xi^0_n},\bar{\eta_n}})\Psi_n\big)\nonumber\\
 =&(P_n,K_n)+
  \sum\limits
  _{l=0}^N\sum\limits
  _{m=1}^kc_{mn}^l(\Psi_{\delta_{mn},\xi^0_{mn},\eta_{mn}}^l,\Phi_{\delta_{mn},\xi^0_{mn},\eta_{mn}}^l),
  \end{align}
  which also reads
  \begin{align}\label{line1'}
 \left\{
  \begin{array}{ll}
 \Psi_n-\mathcal{I}^*\big(a(x)f'_{\varepsilon_n}(\mathcal{H}_{\bar{\delta_n},\bar{\xi^0_n},\bar{\eta_n}})\Phi_n\big)-P_n=\sum\limits
  _{l=0}^N\sum\limits
  _{m=1}^kc_{mn}^l\Psi_{\delta_{mn},\xi^0_{mn},\eta_{mn}}^l,
    \quad\mbox{in $\mathbb{R}^N$},\\
   \Phi_n-\mathcal{I}^*\big(a(x)g'_{\varepsilon_n}(\mathcal{W}_{\bar{\delta_n},\bar{\xi^0_n},\bar{\eta_n}})\Psi_n\big)-K_n=\sum\limits
  _{l=0}^N\sum\limits
  _{m=1}^kc_{mn}^l\Phi_{\delta_{mn},\xi^0_{mn},\eta_{mn}}^l,
   \quad\mbox{in $\mathbb{R}^N$}.
    \end{array}
    \right.
  \end{align}
  Using $(\Psi_n,\Phi_n) \in \mathcal{Z}_{\bar{\delta_{n}},\bar{\xi^0_n},\bar{\eta_n}}$ again, by an easy change of variable, for $i=0,1,\cdots,N$ and $j=1,2\cdots,k$, we have
  \begin{align*}
    0=&\int\limits_{\mathcal{M}} (a(x)\nabla _g \Psi_n\cdot \nabla _g \Phi _{\delta_{jn},\xi^0_{jn},\eta_{jn}}^i+a(x)\nabla _g \Phi_n\cdot \nabla _{g} \Psi _{\delta_{jn},\xi^0_{jn},\eta_{jn}}^i+a(x)h\Psi_n \Phi _{\delta_{jn},\xi^0_{jn},\eta_{jn}}^i+h \Phi_n \Psi _{\delta_{jn},\xi^0_{jn},\eta_{jn}}^i)    d v_g\\
    =&\int\limits_{B(0,r_0/\delta_{jn})} \Big[\delta_{jn}^{N-2-\frac{N}{p+1}}a_n\nabla _{g_n} \Psi_n(\exp_{\xi^0_{jn}}(\delta_{jn} z+\delta_{jn}\eta_{jn}))\cdot \nabla _{g_n} (\chi_n\Phi _{1,0}^i)\\&+\delta_{jn}^{N-2-\frac{N}{q+1}}a_n\nabla _{g_n} \Phi_n(\exp_{\xi^0_{jn}}( \delta_{jn} z+\delta_{jn}\eta_{jn}))\cdot \nabla _{g_n} (\chi_n\Psi _{1,0}^i)\\
  &+\delta_{jn}^{N-\frac{N}{p+1}}a_nh_n\Psi_n (\exp_{\xi^0_{jn}}( \delta_{jn} z+\delta_{jn}\eta_{jn})) \chi_n\Phi _{1,0}^i+\delta_{jn}^{N-\frac{N}{q+1}}a_nh_n \Phi_n(\exp_{\xi^0_{jn}}( \delta_{jn} z+\delta_{jn}\eta_{jn})) \chi_n\Psi _{1,0}^i\Big]    dz\\
    =&\int\limits_{B(0,r_0/\delta_{jn})} \big[a_n\nabla _{g_n} \widetilde{\Psi}_n\cdot \nabla _{g_n} (\chi_n\Phi _{1,0}^i)+a_n\nabla _{g_n} \widetilde{\Phi}_n\cdot \nabla _{g_n} (\chi_n \Psi _{1,0}^i)+\delta_{jn}^2a_nh_n\widetilde{\Psi}_n \Phi _{1,0}^i+\delta_{jn}^2a_nh_n \widetilde{\Phi}_n \Psi _{1,0}^i\big]    dz,
  \end{align*}
  where $g_{n}(z)=\exp_{\xi^0_{jn}}^*g(\delta_{jn}z+\delta_{jn}\eta_{jn})$, $\chi_{n}(z)=\chi({\delta_{jn}z+\delta_{jn}\eta_{jn}})$, $a_{n}(z)=a(\exp_{\xi^0_{jn}}(\delta_{jn}z+\delta_{jn}\eta_{jn}))$ and $h_{n}(z)=h(\exp_{\xi^0_{jn}}(\delta_{jn}z+\delta_{jn}\eta_{jn}))$. By Lemma \ref{nonde},  passing to the limit for the above equality, we obtain
  \begin{equation}\label{line2}
  \lim\limits_{n\rightarrow+\infty}\int\limits_{\mathbb{R}^N}a_n\big(pV_{1,0}^{p-1}\Phi_{1,0}^i\widetilde{\Phi}+qU_{1,0}^{q-1}\Psi_{1,0}^i\widetilde{\Psi}\big)dz=  \lim\limits_{n\rightarrow+\infty}\int\limits_{\mathbb{R}^N} a_n\big(\nabla  \widetilde{\Psi}\cdot \nabla \Phi _{1,0}^i+\nabla  \widetilde{\Phi}\cdot \nabla  \Psi _{1,0}^i\big)dz=0.
  \end{equation}

  {\bf Step 2:} For any $l=0,1,\cdots,N$ and $m=1,2,\cdots,k$, $c_{mn}^l\rightarrow 0$ as $n\rightarrow \infty$. For any $n\in \mathbb{N}^+$, since $(\Psi_n,\Phi_n)$ and $ (P_n,K_n)$ belong to $ \mathcal{Z}_{\bar{\delta_{n}},\bar{\xi^0_n},\bar{\eta_n}}$, multiplying \eqref{line1} by $(\Psi_{\delta_{jn},\xi^0_{jn},\eta_{jn}}^i,\Phi_{\delta_{jn},\xi^0_{jn},\eta_{jn}}^i)$, $0\leq i\leq N$, $1\leq j\leq k$, using \eqref{gu1}-\eqref{gu4}, we have
  \begin{align}\label{line3}
    & -\int\limits _{\mathcal{M}} \big(a(x)f'_{\varepsilon_n}(\mathcal{H}_{\bar{\delta_n},\bar{\xi^0_n},\bar{\eta_n}})\Phi_n \Phi_{\delta_{jn},\xi^0_{jn},\eta_{jn}}^i+a(x)g'_{\varepsilon_n}(\mathcal{W}_{\bar{\delta_n},\bar{\xi^0_n},\bar{\eta_n}})\Psi_n \Psi_{\delta_{jn},\xi^0_{jn},\eta_{jn}}^i\big)d v_g\nonumber\\
    =&\sum\limits_{l=0}^N\sum\limits_{m=1}^kc_{mn}^l\delta_{il}\delta_{jm}\int\limits_{B(0,r_0/{\delta_{jn}})}\big(pa_n\chi^2_n V_{1,0}^{p-1}(\Phi_{1,0}^i)^2+qa_n\chi^2_n U_{1,0}^{q-1}(\Psi_{1,0}^i)^2
  \big)dz+O(\delta_{jn}^2).
  \end{align}
  Moreover, by \eqref{line2}, we have
  \begin{align}\label{line4}
  &\int\limits _{\mathcal{M}} \big(a(x)f'_{\varepsilon_n}(\mathcal{H}_{\bar{\delta_n},\bar{\xi^0_n},\bar{\eta_n}})\Phi_n \Phi_{\delta_{jn},\xi^0_{jn},\eta_{jn}}^i+a(x)g'_{\varepsilon_n}(\mathcal{W}_{\bar{\delta_n},\bar{\xi^0_n},\bar{\eta_n}})\Psi_n \Psi_{\delta_{jn},\xi^0_{jn},\eta_{jn}}^i\big)d v_g \nonumber\\
    =&\int\limits _{\mathcal{M}} \big((p-\alpha{\varepsilon_n})a(x)\mathcal{H}_{\bar{\delta_{n}},\bar{\xi^0_{n},\bar{\eta_n}}}^{p-1-\alpha{\varepsilon_n}}\Phi_n \Phi_{\delta_{jn},\xi^0_{jn},\eta_{jn}}^i+(q-\beta{\varepsilon_n})a(x)\mathcal{W}_{\bar{\delta_{n}},\bar{\xi^0_{n},\bar{\eta_n}}}^{q-1-\beta{\varepsilon_n}}\Psi_n \Psi_{\delta_{jn},\xi^0_{jn},\eta_{jn}}^i\big)d v_g\nonumber\\
    =&\sum\limits_{m=1}^k\int\limits _{\mathcal{M}} \big((p-\alpha{\varepsilon_n})a(x)H_{\delta_{mn},\xi^0_{mn},\eta_{mn}}^{p-1-\alpha{\varepsilon_n}}\Phi_n \Phi_{\delta_{jn},\xi_{jn}}^i+(q-\beta{\varepsilon_n})a(x)W_{\delta_{mn},\xi^0_{mn},\eta_{mn}}^{q-1-\beta{\varepsilon_n}}\Psi_n \Psi_{\delta_{jn},\xi_{jn},\eta_{jn}}^i\big)d v_g\nonumber\\
    =&\sum\limits_{m=1}^k\delta_{jm}\int\limits _{B(0,r_0/\delta_{jn})} \Big[(p-\alpha{\varepsilon_n})\delta_{jn}^{N-\frac{N(p-\alpha{\varepsilon_n})}{p+1}-\frac{N}{p+1}}a_n(\chi_nV_{1,0})^{p-1-\alpha{\varepsilon_n}}\chi_n\delta_{jn}^
    {\frac{N}{p+1}}\Phi_n(\exp_{\xi^0_{jn}}(\delta_{jn} z+\delta_{jn}\eta_{jn})) \Phi_{1,0}^i\nonumber\\
    &+(q-\beta{\varepsilon_n})\delta_{jn}^{N-\frac{N(q-\beta{\varepsilon_n})}{q+1}-\frac{N}{q+1}}a_n(\chi_nU_{1,0})^{q-1-\beta{\varepsilon_n}}\chi_n\delta_{jn}^{\frac{N}{q+1}}\Psi_n(\exp_{\xi^0_{jn}}(\delta_{jn} z+\delta_{jn}\eta_{jn})) \Psi_{1,0}^i\Big]dz\nonumber\\
    =&\sum\limits_{m=1}^k\delta_{jm}\int\limits _{B(0,r_0/\delta_{jn})} \Big[(p-\alpha{\varepsilon_n})\delta_{jn}^{N-\frac{N(p-\alpha{\varepsilon_n})}{p+1}-\frac{N}{p+1}}a_n(\chi_nV_{1,0})^{p-1-\alpha{\varepsilon_n}}
    \widetilde{\Phi}_n(z) \Phi_{1,0}^i\nonumber\\
    &+(q-\beta{\varepsilon_n})\delta_{jn}^{N-\frac{N(q-\beta{\varepsilon_n})}{q+1}-\frac{N}{q+1}}a_n(\chi_nU_{1,0})^{q-1-\beta{\varepsilon_n}}
    \widetilde{\Psi}_n(z) \Psi_{1,0}^i\Big]dz
    \nonumber
    \\ \rightarrow &\sum\limits_{m=1}^k\delta_{jm}\int\limits_{\mathbb{R}^N}a_n(pV_{1,0}^{p-1}\Phi_{1,0}^i\widetilde{\Phi}+qU_{1,0}^{q-1}\Psi_{1,0}^i\widetilde{\Psi})dz=0,\quad \text{as $n\rightarrow+\infty$}.
  \end{align}
It follows from \eqref{line3} and \eqref{line4} that for any $l=0,1,\cdots,N$ and $m=1,2,\cdots,k$, $c_{mn}^l\rightarrow 0$ as $n\rightarrow \infty$.

{\bf Step 3:} $(\widetilde{\Psi},\widetilde{\Phi})=(0,0)$.
For any $(\varphi,\psi)\in C_0^\infty(\mathbb{R}^N)\times C_0^\infty(\mathbb{R}^N)$ and $j=1,2,\cdots,k$, by the dominated convergence theorem, we obtain
\begin{equation*}
  (p-\alpha\varepsilon)\int\limits_{\{z\in \mathbb{R}^N :\varphi(z)\neq0 \}}a_n\big(\chi_n\delta_{jn}^{-\frac{N}{p+1}}V_{1,0}\big)^{p-1-\alpha\varepsilon}\Phi_n(\exp_{\xi_{jn}}(\delta_{jn}z+\delta_{jn}\eta_{jn}))\varphi dz\rightarrow p\int\limits_{\{z\in \mathbb{R}^N :\varphi(z)\neq0 \}}a_nV_{1,0}^{p-1}\widetilde{\Phi}\varphi dz,
\end{equation*}
and
\begin{equation*}
  (q-\beta\varepsilon)\int\limits_{\{z\in \mathbb{R}^N :\psi(z)\neq0 \}}a_n\big(\chi_n\delta_{jn}^{-\frac{N}{q+1}}U_{1,0}\big)^{q-1-\beta\varepsilon}\Psi_n(\exp_{\xi_{jn}}(\delta_{jn}z+\delta_{jn}\eta_{jn}))\psi dx\rightarrow q\int\limits_{\{z\in \mathbb{R}^N :\psi(z)\neq0 \}}a_nU_{1,0}^{q-1}\widetilde{\Psi}\psi dz.
\end{equation*}
as $n\rightarrow+\infty$.
Using \eqref{line1'}, $\|(P_n,K_n)\|\rightarrow0$, $c_{mn}^l\rightarrow 0$ as $n\rightarrow \infty$ for any $l=0,1,\cdots,N$ and $m=1,2,\cdots,k$, we deduce that
$(\widetilde{\Psi},\widetilde{\Phi}) $ satisfies
  \begin{align*}
 \left\{
  \begin{array}{ll}
  -\Delta\widetilde{\Psi}=pV_{1,0}^{p-1}\widetilde{\Phi},
    \quad\mbox{in $\mathbb{R}^N$},\\
   -\Delta\widetilde{\Phi}=qU_{1,0}^{q-1}\widetilde{\Psi},
   \quad\mbox{in $\mathbb{R}^N$}.
    \end{array}
    \right.
  \end{align*}
 This together with \eqref{line2} and Lemma \ref{nonde} yields that $(\widetilde{\Psi},\widetilde{\Phi})=(0,0)$.

  {\bf Step 4:}
  $\|\mathcal{I}^*\big(a(x)f'_{\varepsilon_n}(\mathcal{H}_{\bar{\delta_n},\bar{\xi^0_n},\bar{\eta_n}})\Phi_n,a(x)g'_{\varepsilon_n}(\mathcal{W}_{\bar{\delta_n},\bar{\xi^0_n},\bar{\eta_n}})\Psi_n\big)\|\rightarrow 0$ as $n\rightarrow \infty$. By \eqref{em}, we know
  \begin{align*}
    &\|\mathcal{I}^*\big(a(x)f'_{\varepsilon_n}(\mathcal{H}_{\bar{\delta_n},\bar{\xi^0_n},\bar{\eta_n}})\Phi_n,a(x)g'_{\varepsilon_n}(\mathcal{W}_{\bar{\delta_n},\bar{\xi^0_n},\bar{\eta_n}})\Psi_n\big)\|
    \\ \leq& C\big\|a(x)f'_{\varepsilon_n}(\mathcal{H}_{\bar{\delta_n},\bar{\xi^0_n},\bar{\eta_n}})\Phi_n\big\|_{\frac{p+1}{p}}+C\big\|a(x)g'_{\varepsilon_n}(\mathcal{W}_{\bar{\delta_n},\bar{\xi^0_n},\bar{\eta_n}})\Psi_n\big\|_{\frac{q+1}{q}}.
  \end{align*}
  For any fixed $R>0$ and $j=1,2,\cdots,k$, by the H\"{o}lder inequality, $\widetilde{\Phi}_n\rightarrow0$ in $L_{loc}^{\frac{p+1}{1+\alpha{\varepsilon_n}}}({\mathbb{R}^N})$ and $\widetilde{\Psi}_n\rightarrow0$ in $L_{loc}^{\frac{q+1}{1+\beta{\varepsilon_n}}}({\mathbb{R}^N})$, we have
  \begin{align*}
   & \big\|a(x)f'_{\varepsilon_n}(\mathcal{H}_{\bar{\delta_n},\bar{\xi_n}})\Phi_n\big\|_{\frac{p+1}{p}}^{\frac{p+1}{p}}\\
    =&\int\limits_{\mathcal{M}}\big|(p-\alpha{\varepsilon_n})a(x)\mathcal{H}_{\bar{\delta_n},\bar{\xi^0_n},\bar{\eta_n}}^{p-1-\alpha{\varepsilon_n}}\Phi_n\big|^{\frac{p+1}{p}}d v_g\\
    =&\sum\limits_{j=1}^k
    \delta_{jn}^{\frac{N \alpha {\varepsilon_n}}{p}}\int\limits_{B(0,r_0/\delta_{jn})}\big|(p-\alpha{\varepsilon_n})a_n\chi_n^{p-2-\alpha\varepsilon}V_{1,0}^{p-1-\alpha{\varepsilon_n}} \chi_n\delta_{jn}^{\frac{N}{p+1}}\Phi_n(\exp_{\xi^0_{jn}}(\delta_{jn} z+\delta_{jn} \eta_{jn}))\big|^{\frac{p+1}{p}}dz\\
    =&\sum\limits_{j=1}^k\delta_{jn}^{\frac{N \alpha {\varepsilon_n}}{p}}\int\limits_{B(0,r_0/\delta_{jn})}\big|(p-\alpha{\varepsilon_n})a_n\chi_n^{p-2-\alpha\varepsilon}V_{1,0}^{p-1-\alpha{\varepsilon_n}} \widetilde{\Phi}_n(z)\big|^{\frac{p+1}{p}}dz\\
    \leq & C\Big(\int\limits_{B(0,r_0/\delta_{jn})}V_{1,0} ^{p+1}dz\Big)^{\frac{p-1-\alpha{\varepsilon_n}}{p}}\Big(\int\limits_{B(0,r_0/\delta_{jn})}|\widetilde{\Phi}_n(z)|^{\frac{p+1}{1+\alpha{\varepsilon_n}}}\Big)^{\frac{1+\alpha{\varepsilon_n}}{p}}\\
    \leq &C\Big(\int\limits_{B(0,R)}|\widetilde{\Phi}_n(z)|^{\frac{p+1}{1+\alpha{\varepsilon_n}}}\Big)^{\frac{1+\alpha{\varepsilon_n}}{p}}
    +C\varepsilon_n^{\frac{[(N-2)p-2](p-1-\alpha\varepsilon_n)}{2p}}\rightarrow0,\quad \text{as $n\rightarrow+\infty$},
  \end{align*}
  and
  \begin{align*}
   & \big\|a(x)g'_{\varepsilon_n}(\mathcal{W}_{\bar{\delta_n},\bar{\xi^0_n},\bar{\eta_n}})\Psi_n\big\|_{\frac{q+1}{q}}^{\frac{q+1}{q}}\\
    =&\int\limits_{\mathcal{M}}\big|(q-\beta{\varepsilon_n})a(x)\mathcal{W}_{\bar{\delta_n},\bar{\xi^0_n},\bar{\eta_n}}^{q-1-\beta{\varepsilon_n}}\Psi_n\big|^{\frac{q+1}{q}}d v_g\\
    =&\sum\limits_{j=1}^k
    \delta_{jn}^{\frac{N \beta {\varepsilon_n}}{q}}\int\limits_{B(0,r_0/\delta_{jn})}\big|(q-\beta{\varepsilon_n})a_n\chi_n^{q-2-\beta\varepsilon}U_{1,0}^{q-1-\beta{\varepsilon_n}} \chi_n\delta_{jn}^{\frac{N}{q+1}}\Psi_n(\exp_{\xi_{jn}}(\delta_{jn} z+\delta_{jn}\eta_{jn}))\big|^{\frac{q+1}{q}}dz\\
    =&\sum\limits_{j=1}^k\delta_{jn}^{\frac{N \beta {\varepsilon_n}}{q}}\int\limits_{B(0,r_0/\delta_{jn})}\big|(q-\beta{\varepsilon_n})a_n\chi_n^{q-2-\beta\varepsilon}U_{1,0}^{q-1-\beta{\varepsilon_n}} \widetilde{\Psi}_n(z)\big|^{\frac{q+1}{q}}dz\\
    \leq & C\Big(\int\limits_{B(0,r_0/\delta_{jn})}U_{1,0} ^{q+1}dz\Big)^{\frac{q-1-\beta{\varepsilon_n}}{q}}\Big(\int\limits_{B(0,r_0/\delta_{jn})}|\widetilde{\Psi}_n(z)|^{\frac{q+1}{1+\beta{\varepsilon_n}}}\Big)^
    {\frac{1+\beta{\varepsilon_n}}{q}}\\
    \leq &
    \left\{
    \begin{array}{ll}
    \displaystyle C\Big(\int\limits_{B(0,R)}|\widetilde{\Psi}_n(z)|^{\frac{q+1}{1+\beta{\varepsilon_n}}}\Big)^{\frac{1+\beta{\varepsilon_n}}{q}
    }+C\varepsilon_n^{\frac{[(N-2)q-2](q-1-\beta\varepsilon_n)}{2q}}\rightarrow0,\quad &\text{as $n\rightarrow+\infty$,\ \ if $p>\frac{N}{N-2}$},\\
    \displaystyle C\Big(\int\limits_{B(0,R)}|\widetilde{\Psi}_n(z)|^{\frac{q+1}{1+\beta{\varepsilon_n}}}\Big)^{\frac{1+\beta{\varepsilon_n}}{q}
    }+C\varepsilon_n^{\frac{[(N-3)q-3](q-1-\beta\varepsilon_n)}{2q}}\rightarrow0,\quad &\text{as $n\rightarrow+\infty$,\ \ if $p=\frac{N}{N-2}$},\\
    \displaystyle C\Big(\int\limits_{B(0,R)}|\widetilde{\Psi}_n(z)|^{\frac{q+1}{1+\beta{\varepsilon_n}}}\Big)^{\frac{1+\beta{\varepsilon_n}}{q}
    }+C\varepsilon_n^{\frac{Np(q-1-\beta\varepsilon_n)}{2q}}\rightarrow0,\quad &\text{as $n\rightarrow+\infty$, \ \ if $p<\frac{N}{N-2}$}.\\
    \end{array}
    \right.
  \end{align*}

From the above arguments,  we get $\|(\Psi_n,\Phi_n)\|\rightarrow0$ as $n\rightarrow+\infty$, which is an absurd. Thus, we complete the proof.
\end{proof}

For any $\varepsilon>0$ small enough,  $\bar{t}\in (\mathbb{R}^+)^k$, and $\bar{\eta}\in (\mathbb{R}^N)^k$, if $\bar{\delta}$ is as in \eqref{solu'}, then equation \eqref{tou2} is equivalent to
\begin{equation*}
  \mathcal{L}_{\varepsilon,\bar{t},\bar{\xi^0},\bar{\eta}}(\Psi,\Phi)=\mathcal{N}_{\varepsilon,\bar{t},\bar{\xi^0},\bar{\eta}}(\Psi,\Phi)+\mathcal{R}_{\varepsilon,\bar{t},\bar{\xi^0},\bar{\eta}},
\end{equation*}
where
\begin{align}\label{defn}
\mathcal{N}_{\varepsilon,\bar{t},\bar{\xi^0},\bar{\eta}}(\Psi,\Phi)=
  \Pi_{\bar{\delta},\bar{\xi^0},\bar{\eta}}^\bot \mathcal{I}^* \Big\{
  &a(x)\big[f_\varepsilon(\mathcal{H}_{\bar{\delta},\bar{\xi^0},\bar{\eta}}+\Phi)-f_\varepsilon(\mathcal{H}_{\bar{\delta},\bar{\xi^0},\bar{\eta}})-f'_\varepsilon(\mathcal{H}_{\bar{\delta},\bar{\xi^0},\bar{\eta}})\Phi\big], \nonumber \\& a(x)\big[ g_\varepsilon(\mathcal{W}_{\bar{\delta},\bar{\xi^0},\bar{\eta}}+\Psi)-g_\varepsilon(\mathcal{W}_{\bar{\delta},\bar{\xi^0},\bar{\eta}})-g'_\varepsilon(\mathcal{W}_{\bar{\delta},\bar{\xi^0},\bar{\eta}})\Psi\big]
  \Big\},
\end{align}
and
\begin{equation}\label{defr}
\mathcal{R}_{\varepsilon,\bar{t},\bar{\xi^0},\bar{\eta}}=
  \Pi_{\bar{\delta},\bar{\xi^0},\bar{\eta}}^\bot\Big[
  \mathcal{I}^*\big(a(x)f_\varepsilon(\mathcal{H}_{\bar{\delta},\bar{\xi^0},\bar{\eta}}),a(x)g_\varepsilon(\mathcal{W}_{\bar{\delta},\bar{\xi^0},\bar{\eta}})\big)-\big(\mathcal{W}_{\bar{\delta},\bar{\xi^0},\bar{\eta}},
  \mathcal{H}_{\bar{\delta},\bar{\xi^0},\bar{\eta}}\big)
  \Big].
\end{equation}
In the following lemma, we estimate the reminder term $\mathcal{R}_{\varepsilon,\bar{t},\bar{\xi^0},\bar{\eta}}$.
\begin{lemma}\label{error}
Under the assumptions of Theorem \ref{th}, if  $\bar{\delta}$ is as in \eqref{solu'}, then for any $\varepsilon>0$ small enough, there holds
\begin{equation*}
  \|\mathcal{R}_{\varepsilon,\bar{t},\bar{\xi^0},\bar{\eta}}\|\leq C\varepsilon|\log \varepsilon|,
\end{equation*}
where $\mathcal{R}_{\varepsilon,\bar{t},\bar{\xi^0},\bar{\eta}}$ is as in \eqref{defr}.
\end{lemma}
\begin{proof}
 Since $d_g(\xi^0_j,\xi^0_m)>r_0$ for any $j\neq m$, by \eqref{em} and the definition of  the function $\chi(z)$, there exists $C>0$ such that %we know there exists $C>0$ such that for $\varepsilon>0$ small enough, $\bar{t}\in (\mathbb{R}^+)^k$, and $\bar{\eta}\in (\mathbb{R}^N)^k$, there holds
\begin{align*}
  \|\mathcal{R}_{\varepsilon,\bar{t},\bar{\xi^0},\bar{\eta}}\|\leq & C\big\|a(x)f_\varepsilon(\mathcal{H}_{\bar{\delta},\bar{\xi^0},\bar{\eta}})+a(x)\Delta_g  \mathcal{W}_{\bar{\delta},\bar{\xi^0},\bar{\eta}}+\nabla_g a(x)\cdot\nabla_g \mathcal{W}_{\bar{\delta},\bar{\xi^0},\bar{\eta}}-a(x)h\mathcal{W}_{\bar{\delta},\bar{\xi^0},\bar{\eta}}\big\|_{{\frac{p+1}{p}}}
  \\&+C\big\|a(x)g_\varepsilon(\mathcal{W}_{\bar{\delta},\bar{\xi^0},\bar{\eta}})+a(x)\Delta_g \mathcal{H}_{\bar{\delta},\bar{\xi^0},\bar{\eta}}+\nabla_ga(x)\cdot\nabla_g \mathcal{H}_{\bar{\delta},\bar{\xi^0},\bar{\eta}}-a(x)h\mathcal{H}_{\bar{\delta},\bar{\xi^0},\bar{\eta}}\big\|_{{\frac{q+1}{q}}}\\
  \leq &C\big\|f_\varepsilon(\mathcal{H}_{\bar{\delta},\bar{\xi^0},\bar{\eta}})+\Delta_g  \mathcal{W}_{\bar{\delta},\bar{\xi^0},\bar{\eta}}-h\mathcal{W}_{\bar{\delta},\bar{\xi^0},\bar{\eta}}\big\|_{{\frac{p+1}{p}}}+\big\|\nabla_g a(x)\cdot\nabla_g \mathcal{W}_{\bar{\delta},\bar{\xi^0},\bar{\eta}}\big\|_{{\frac{p+1}{p}}}
  \\&+C\big\|g_\varepsilon(\mathcal{W}_{\bar{\delta},\bar{\xi^0},\bar{\eta}})+\Delta_g \mathcal{H}_{\bar{\delta},\bar{\xi^0},\bar{\eta}}-h\mathcal{H}_{\bar{\delta},\bar{\xi^0},\bar{\eta}}\big\|_{{\frac{q+1}{q}}}+\big\|\nabla_ga(x)\cdot\nabla_g \mathcal{H}_{\bar{\delta},\bar{\xi^0},\bar{\eta}}\big\|_{{\frac{q+1}{q}}}\\
  =&C\sum\limits_{j=1}^k\big\|f_\varepsilon(H_{\delta_j,\xi^0_j,\eta_j})+\Delta_g W_{\delta_j,\xi^0_j,\eta_j}-hW_{\delta_j,\xi^0_j,\eta_j}\big\|_{{\frac{p+1}{p}}}+C\sum\limits_{j=1}^k
  \big\|\nabla_g a(x)\cdot\nabla_g {W}_{\delta_j,\xi^0_j,\eta_j}\big\|_{{\frac{p+1}{p}}}
  \\&+C\sum\limits_{j=1}^k\big\|g_\varepsilon(W_{\delta_j,\xi^0_j,\eta_j})+\Delta_g H_{\delta_j,\xi^0_j,\eta_j}-hH_{\delta_j,\xi^0_j,\eta_j}\big\|_{{\frac{q+1}{q}}}+C\sum\limits_{j=1}^k
  \big\|\nabla_g a(x)\cdot\nabla_g {H}_{\delta_j,\xi^0_j,\eta_j}\big\|_{{\frac{q+1}{q}}}\\
  =&:C\sum\limits_{j=1}^k(I_j+II_j+III_j+IV_j).
\end{align*}
Similar to \cite[Lemma 4.2]{CW}, we can prove that $I_j=III_j=O(\varepsilon |\log \varepsilon|)$. It remains to estimate $II_j$ and $IV_j$, $j=1,2,\cdots,k$.

For any fixed $R>0$ and $j=1,2,\cdots,k$, since $\xi^0_j$ is a non-degenerate critical point of $a(x)$, by Lemmas \ref{jian1} and \ref{jian2}, we have
\begin{align}\label{new1}
  II_j^{{\frac{p+1}{p}}}\leq& C\max\limits_{1\leq l\leq N}\int\limits_{B(0,r_0)}|y|^{\frac{p+1}{p}}\big|\partial _{y_l}\big(\chi(y)\delta_j^{-\frac{N}{q+1}}U_{1,0}(\delta_j^{-1}y-\eta_j)\big)\big|^{\frac{p+1}{p}}dy\nonumber\\
  \leq& C\int\limits_{B(0,r_0/\delta_j)\backslash B(0,r_0/2\delta_j)}|\delta_j z|^{\frac{p+1}{p}}\big|\delta_j^{-\frac{N}{q+1}}U_{1,0}(z-\eta_j)\big|^{\frac{p+1}{p}}\delta_j^Ndz\nonumber\\
  &+C\max\limits_{1\leq l\leq N}\int\limits_{B(0,r_0/\delta_j)}|\delta_jz|^{\frac{p+1}{p}}\big|\delta_j^{-\frac{N}{q+1}-1}\partial _{y_l}U_{1,0}(z-\eta_j)\big|^{\frac{p+1}{p}}\delta_j^Ndz\nonumber\\
  \leq&
  C\delta_j^{N+\frac{p+1}{p}-\frac{N(p+1)}{p(q+1)}}\int\limits_{B(0,r_0/\delta_j)\backslash B(0,r_0/2\delta_j)}| z|^{\frac{p+1}{p}}\big|U_{1,0}(z-\eta_j)\big|^{\frac{p+1}{p}}dz\nonumber\\
  &+C\delta_j^{N-\frac{N(p+1)}{p(q+1)}}\max\limits_{1\leq l\leq N}\int\limits_{B(0,R)}|z|^{\frac{p+1}{p}}\big|\partial _{y_l}U_{1,0}(z-\eta_j)\big|^{\frac{p+1}{p}}dz\nonumber\\&+
  C\delta_j^{N-\frac{N(p+1)}{p(q+1)}}\max\limits_{1\leq l\leq N}\int\limits_{B(0,r_0/\delta_j)\backslash B(0,R)}|z|^{\frac{p+1}{p}}\big|\partial _{y_l}U_{1,0}(z-\eta_j)\big|^{\frac{p+1}{p}}dz
  \nonumber\\
    = &
    \left\{
    \begin{array}{ll}
    O(\delta_j^{\frac{N}{p}})+O\big(\delta_j^{N-\frac{N(p+1)}{p(q+1)}}\big),\quad &\text{if $p>\frac{N}{N-2}$;}\nonumber\\
    O(\delta_j^{\frac{N-p-1}{p}})+O\big(\delta_j^{N-\frac{N(p+1)}{p(q+1)}}\big),\quad &\text{if $p=\frac{N}{N-2}$;}\nonumber\\
    O(\delta_j^{\frac{N(p+1)}{q+1}})+O\big(\delta_j^{N-\frac{N(p+1)}{p(q+1)}}\big),\quad &\text{if $p<\frac{N}{N-2}$,}
    \end{array}
    \right.
    \nonumber\\=&O(\delta_j^{\frac{2(p+1)}{p}})=O(\varepsilon^{\frac{p+1}{p}}).
\end{align}
Similarly, it holds
\begin{align}\label{new2}
  IV_j^{{\frac{q+1}{q}}}\leq&
  C\delta_j^{N+\frac{q+1}{q}-\frac{N(q+1)}{q(p+1)}}\int\limits_{B(0,r_0/\delta_j)\backslash B(0,r_0/2\delta_j)}| z|^{\frac{q+1}{q}}\big|V_{1,0}(z-\eta_j)\big|^{\frac{q+1}{q}}dz\nonumber\\
  &+C\delta_j^{N-\frac{N(q+1)}{q(p+1)}}\max\limits_{1\leq l\leq N}\int\limits_{B(0,R)}|z|^{\frac{q+1}{q}}\big|\partial _{y_l}V_{1,0}(z-\eta_j)\big|^{\frac{q+1}{q}}dz\nonumber\\&+
  C\delta_j^{N-\frac{N(q+1)}{q(p+1)}}\max\limits_{1\leq l\leq N}\int\limits_{B(0,r_0/\delta_j)\backslash B(0,R)}|z|^{\frac{q+1}{q}}\big|\partial _{y_l}V_{1,0}(z-\eta_j)\big|^{\frac{q+1}{q}}dz
 \nonumber \\
    = &
    O(\delta_j^{\frac{N}{q}})+O\big(\delta_j^{N-\frac{N(q+1)}{q(p+1)}}\big)=O(\varepsilon^{\frac{q+1}{q}}).
\end{align}
This ends the proof.
\end{proof}
\noindent{\bf Proof of Proposition \ref{propo1}.}
By using Lemmas  \ref{line} and \ref{error}, a similar discussion of \cite[Proposition 3.1]{CW} completes the proof.
\qed

\section{Proof of Proposition \ref{propo2}}\label{sec5}
This section is devoted to the proof of Proposition \ref{propo2}. As a first step, we have
\begin{lemma}
Under the assumptions of Theorem \ref{th}, if $\bar{\delta}$ is as in \eqref{solu'}, then for any $\varepsilon>0$ small enough, if $(\bar{t},\bar{\eta})$ is a critical point of the functional $\widetilde{\mathcal{J}}_\varepsilon$, then $\big(\mathcal{W}_{\bar{\delta},\bar{\xi^0},\bar{\eta}}+\Psi_{\varepsilon,\bar{t},\bar{\xi^0},\bar{\eta}}
  ,\mathcal{H}_{\bar{\delta},\bar{\xi^0},\bar{\eta}}+\Phi_{\varepsilon,\bar{t},\bar{\xi^0},\bar{\eta}}\big)$ is a  solution of system \eqref{prob}, or equivalently of  \eqref{repro}.
\end{lemma}
\begin{proof}
Let $(\bar{t},\bar{\eta})$ be a critical point of  $\widetilde{\mathcal{J}}_\varepsilon$, where $\bar{t}=(t_1,t_2,\cdots,t_k)\in (\mathbb{R}^+)^k$ and $\bar{\eta}=(\eta_1,\eta_2,\cdots,\eta_k)\in (\mathbb{R}^N)^k$.  Since $(\bar{t},\bar{\eta})$ be a critical point of  $\widetilde{\mathcal{J}}_\varepsilon$, for any  $l=1,2,\cdots,N$ and $m=1,2,\cdots,k$, there hold
\begin{equation*}
  \mathcal{J}'_{\varepsilon}\big(\mathcal{W}_{\bar{\delta},\bar{\xi^0},\bar{\eta}}+\Psi_{\varepsilon,\bar{t},\bar{\xi^0},\bar{\eta}},\mathcal{H}_{\bar{\delta},\bar{\xi^0},\bar{\eta}}+\Phi_{\varepsilon,\bar{t},\bar{\xi^0},\bar{\eta}}\big)
  \big(\partial _{t_m}\mathcal{W}_{\bar{\delta},\bar{\xi^0},\bar{\eta}}+\partial _{t_m}\Psi_{\varepsilon,\bar{t},\bar{\xi^0},\bar{\eta}},\partial _{t_m}\mathcal{H}_{\bar{\delta},\bar{\xi^0},\bar{\eta}}+\partial _{t_m}\Phi_{\varepsilon,\bar{t},\bar{\xi^0},\bar{\eta}}\big)=0,
\end{equation*}
and
\begin{equation*}
  \mathcal{J}'_{\varepsilon}\big(\mathcal{W}_{\bar{\delta},\bar{\xi^0},\bar{\eta}}+\Psi_{\varepsilon,\bar{t},\bar{\xi^0},\bar{\eta}},\mathcal{H}_{\bar{\delta},\bar{\xi^0},\bar{\eta}}+\Phi_{\varepsilon,\bar{t},\bar{\xi^0},\bar{\eta}}\big)
  \big(\partial _{\eta_{ml}}\mathcal{W}_{\bar{\delta},\bar{\xi^0},\bar{\eta}}+\partial _{\eta_{ml}}\Psi_{\varepsilon,\bar{t},\bar{\xi^0},\bar{\eta}},\partial _{\eta_{ml}}\mathcal{H}_{\bar{\delta},\bar{\xi^0},\bar{\eta}}+\partial _{\eta_{ml}}\Phi_{\varepsilon,\bar{t},\bar{\xi^0},\bar{\eta}}\big)=0.
\end{equation*}
For any $(\varphi,\psi)\in \mathcal{X}_{p,q}(\mathcal{M})$, by Proposition \ref{propo1}, there exist some constants $c_{01},c_{02},\cdots,c_{0k}$, $c_{11},c_{12},\cdots,c_{1k}$, $\cdots$, $c_{N1},c_{N2},\cdots,c_{Nk}$ such that
\begin{equation*}
  \mathcal{J}_\varepsilon'(\mathcal{W}_{\bar{\delta},\bar{\xi^0},\bar{\eta}}+\Psi_{\varepsilon,\bar{t},\bar{\xi^0},\bar{\eta}},\mathcal{H}_{\bar{\delta},\bar{\xi^0},\bar{\eta}}+
  \Phi_{\varepsilon,\bar{t},\bar{\xi^0},\bar{\eta}})(\varphi,\psi)=
  \sum\limits_{l=0}^N\sum\limits_{m=1}^kc_{lm}\big\langle(\Psi_{\delta_m,\xi^0_m,\eta_m}^l,\Phi^l_{\delta_m,\xi^0_m,\eta_m}),(\varphi,\psi)\big\rangle_h.
\end{equation*}
Let $\partial _s$ denote $\partial_ {t_m}$ or $\partial _{\eta_{ml}}$ for any $l=1,2,\cdots,N$ and $m=1,2,\cdots,k$. Then
\begin{align}\label{com}
0=&  \partial_s \widetilde{\mathcal{J}}_{\varepsilon}(\bar{t},\bar{\xi^0},\bar{\eta}) =\mathcal{J}'_{\varepsilon}\big(\mathcal{W}_{\bar{\delta},\bar{\xi^0},\bar{\eta}}+\Psi_{\varepsilon,\bar{t},\bar{\xi^0},\bar{\eta}},\mathcal{H}_{\bar{\delta},\bar{\xi^0},\bar{\eta}}+\Phi_{\varepsilon,\bar{t},\bar{\xi^0},\bar{\eta}}\big)
  \big(\partial _{s}\mathcal{W}_{\bar{\delta},\bar{\xi^0},\bar{\eta}}+\partial _{s}\Psi_{\varepsilon,\bar{t},\bar{\xi^0},\bar{\eta}},\partial _{s}\mathcal{H}_{\bar{\delta},\bar{\xi^0},\bar{\eta}}+\partial _{s}\Phi_{\varepsilon,\bar{t},\bar{\xi^0},\bar{\eta}}\big)\nonumber\\
  =& \big\langle \big(\mathcal{W}_{\bar{\delta},\bar{\xi^0},\bar{\eta}}+\Psi_{\varepsilon,\bar{t},\bar{\xi^0},\bar{\eta}}
  ,\mathcal{H}_{\bar{\delta},\bar{\xi^0},\bar{\eta}}+\Phi_{\varepsilon,\bar{t},\bar{\xi^0},\bar{\eta}}\big)
  -\mathcal{I}^*\big(a(x)f_\varepsilon(\mathcal{H}_{\bar{\delta},\bar{\xi^0},\bar{\eta}}+\Phi_{\varepsilon,\bar{t},\bar{\xi^0},\bar{\eta}}),a(x)g_\varepsilon(\mathcal{W}_{\bar{\delta},\bar{\xi^0},\bar{\eta}}+\Psi_{\varepsilon,\bar{t},\bar{\xi^0},\bar{\eta}})\big),\nonumber\\
  &\big(\partial _{s}\mathcal{W}_{\bar{\delta},\bar{\xi^0},\bar{\eta}}+\partial _{s}\Psi_{\varepsilon,\bar{t},\bar{\xi^0},\bar{\eta}},\partial _{s}\mathcal{H}_{\bar{\delta},\bar{\xi^0},\bar{\eta}}+\partial _{s}\Phi_{\varepsilon,\bar{t},\bar{\xi^0},\bar{\eta}}\big)\big\rangle \nonumber\\
  =&\sum\limits_{i=0}^N\sum\limits_{j=1}^kc_{ij}\big\langle\big(\Psi_{\delta_j,\xi^0_j,\eta_j}^i,\Phi^i_{\delta_j,\xi^0_j,\eta_j}\big),\big(\partial _{s}\mathcal{W}_{\bar{\delta},\bar{\xi^0},\bar{\eta}}+\partial _{s}\Psi_{\varepsilon,\bar{t},\bar{\xi^0},\bar{\eta}},\partial _{s}\mathcal{H}_{\bar{\delta},\bar{\xi^0},\bar{\eta}}+\partial _{s}\Phi_{\varepsilon,\bar{t},\bar{\xi^0},\bar{\eta}}\big)\big\rangle_h.
\end{align}

We prove that for any $\varepsilon>0$ small enough, there holds
\begin{equation*}
  c_{ij}=0,\quad \text{for any  $i=0,1,\cdots,N$ and $j=1,2,\cdots,k$}.
\end{equation*}
For any $l=1,2,\cdots,N$ and $m=1,2,\cdots,k$,
  we can easily check that there hold
  \begin{equation}\label{ch1}
    \big(\partial_{t_m}\mathcal{W}_{\bar{\delta},\bar{\xi^0},\bar{\eta}},\partial_{t_m}\mathcal{H}_{\bar{\delta},\bar{\xi^0},\bar{\eta}}\big)=-\frac{1}{t_m}\Big(\Psi^0_{\delta_m,\xi^0_m,\eta_m}+\sum\limits_{l=1}^N\eta_{ml}\Psi^l_{\delta_m,\xi^0_m,\eta_m},\Phi^0_{\delta_m,\xi^0_m,\eta_m}+\sum\limits_{l=1}^N\eta_{ml}\Phi^l_{\delta_m,\xi^0_m,\eta_m}\Big),
  \end{equation}
  and
  \begin{equation}\label{ch2}
    \big(\partial_{\eta_{ml}} (\mathcal{W}_{\bar{\delta},\bar{\xi^0},\bar{\eta}}),
    \partial_{\eta_{ml}} (\mathcal{H}_{\bar{\delta},\bar{\xi^0},\bar{\eta}})\big)=\big(\Psi^l_{\delta_m,\xi^0_m,\eta_m},\Phi^l_{\delta_m,\xi^0_m,\eta_m}\big).
  \end{equation}
   Using \eqref{gu1}-\eqref{gu4} and \eqref{ch1}-\eqref{ch2}, we have
\begin{align}\label{e51}
  &\sum\limits_{i=0}^N\sum\limits_{j=1}^kc_{ij}\big\langle\big(\Psi_{\delta_j,\xi^0_j,\eta_j}^i,\Phi^i_{\delta_j,\xi^0_j,\eta_j}\big),\big(\partial _{t_m}\mathcal{W}_{\bar{\delta},\bar{\xi^0},\bar{\eta}},\partial _{t_m}\mathcal{H}_{\bar{\delta},\bar{\xi^0},\bar{\eta}}\big)\big\rangle_h\nonumber\\
  =&-\frac{1}{t_m}\sum\limits_{i=0}^N\sum\limits_{j=1}^kc_{ij}\big\langle\big(\Psi_{\delta_j,\xi^0_j,\eta_j}^i,\Phi^i_{\delta_j,\xi^0_j,\eta_j}\big),\big(\Psi^0_{\delta_m,\xi^0_m,\eta_m},\Phi^0_{\delta_m,\xi^0_m,\eta_m}\big)\big\rangle_h\nonumber\\
  &-\frac{1}{t_m}\sum\limits_{i=0}^N\sum\limits_{j=1}^k\sum\limits_{l=1}^Nc_{ij}\eta_{ml}\big\langle\big(\Psi_{\delta_j,\xi^0_j,\eta_j}^i,\Phi^i_{\delta_j,\xi^0_j,\eta_j}\big),\big(\Psi^l_{\delta_m,\xi^0_m,\eta_m},\Phi^l_{\delta_m,\xi^0_m,\eta_m}\big)\big\rangle_h\nonumber\\
  =&-\frac{1}{t_m}\sum\limits_{i=0}^N\sum\limits_{j=1}^kc_{ij}\delta_{i0}\delta_{jm}\int\limits_{B(0,r_0/{\delta_m})}\big(pa_{\delta_m,\xi^0_m,\eta_m}\chi^2_{\delta_m,\eta_m} V_{1,0}^{p-1}(\Phi_{1,0}^0)^2+qa_{\delta_m,\xi^0_m,\eta_m}\chi^2_{\delta_m,\eta_m} U_{1,0}^{q-1}(\Psi_{1,0}^0)^2
  \big)dx\nonumber\\
  &-\frac{1}{t_m}\sum\limits_{i=0}^N\sum\limits_{j=1}^k\sum\limits_{l=1}^Nc_{ij}\eta_{ml}\delta_{il}\delta_{jm}\int\limits_{B(0,r_0/{\delta_m})}a_{\delta_m,\xi^0_m,\eta_m}\chi^2_{\delta_m,\eta_m}\big(p V_{1,0}^{p-1}(\Phi_{1,0}^l)^2+q U_{1,0}^{q-1}(\Psi_{1,0}^l)^2
  \big)dx
  +O(\delta_m^2),
\end{align}
\begin{align}\label{e52}
  &\sum\limits_{i=0}^N\sum\limits_{j=1}^kc_{ij}\big\langle\big(\Psi_{\delta_j,\xi^0_j,\eta_j}^i,\Phi^i_{\delta_j,\xi^0_j,\eta_j}\big),\big(\partial _{\eta_{ml}}\mathcal{W}_{\bar{\delta},\bar{\xi^0},\bar{\eta}},\partial _{\eta_{ml}}\mathcal{H}_{\bar{\delta},\bar{\xi^0},\bar{\eta}}\big)\big\rangle_h\nonumber\\
  =&\sum\limits_{i=0}^N\sum\limits_{j=1}^kc_{ij}\big\langle\big(\Psi_{\delta_j,\xi^0_j,\eta_j}^i,\Phi^i_{\delta_j,\xi^0_j,\eta_j}\big),\big(\Psi^l_{\delta_m,\xi^0_m,\eta_m},\Phi^l_{\delta_m,\xi^0_m,\eta_m}\big)\big\rangle_h\nonumber\\
  =&\sum\limits_{i=0}^N\sum\limits_{j=1}^kc_{ij}\delta_{il}\delta_{jm}\int\limits_{B(0,r_0/{\delta_m})}a_{\delta_m,\xi^0_m,\eta_m}\chi^2_{\delta_m,\eta_m}\big(p
  V_{1,0}^{p-1}(\Phi_{1,0}^l)^2+q U_{1,0}^{q-1}(\Psi_{1,0}^l)^2
  \big)dx+O(\delta_m^2),
\end{align}
and
\begin{align}\label{e53}
  &\sum\limits_{i=0}^N\sum\limits_{j=1}^kc_{ij}\big\langle\big(\Psi_{\delta_j,\xi^0_j,\eta_j}^i,\Phi^i_{\delta_j,\xi^0_j,\eta_j}\big),\big(\partial _{s}\Psi_{\varepsilon,\bar{t},\bar{\xi^0},\bar{\eta}},\partial _{s}\Phi_{\varepsilon,\bar{t},\bar{\xi^0},\bar{\eta}}\big)\big\rangle_h\nonumber\\
 =&-\sum\limits_{i=0}^N\sum\limits_{j=1}^kc_{ij}\big\langle\big(\partial_s\Psi_{\delta_j,\xi^0_j,\eta_j}^i,\partial_s\Phi^i_{\delta_j,\xi^0_j,\eta_j}\big),
  \big(\Psi_{\varepsilon,\bar{t},\bar{\xi^0},\bar{\eta}},\Phi_{\varepsilon,\bar{t},\bar{\xi^0},\bar{\eta}}\big)\big\rangle_h,
\end{align}
where $a_{\delta_m,\xi^0_m,\eta_m}(z)=a(\exp_{\xi^0_m}(\delta_m z+\delta_m \eta_m))$ and $\chi_{\delta_m,\eta_m}(z)=\chi(\delta_mz+\delta_m \eta_m)$.

For any $\vartheta\in (0,1)$, with the aid of Proposition \ref{propo1}, by the H\"{o}lder inequality, we have
\begin{align}\label{e54}
 &\sum\limits_{i=0}^N\sum\limits_{j=1}^kc_{ij}\big\langle\big(\partial _{t_m}\Psi_{\delta_j,\xi^0_j,\eta_j}^i,\partial _{t_m}\Phi^i_{\delta_j,\xi^0_j,\eta_j}\big),\big(\Psi_{\varepsilon,\bar{t},\bar{\xi^0},\bar{\eta}},\Phi_{\varepsilon,\bar{t},\bar{\xi^0},\bar{\eta}}\big)\big\rangle_h\nonumber\\
  \leq&C\sum\limits_{i=0}^N\sum\limits_{j=1}^kc_{ij}\delta_{jm}\Big(\big\|\partial_\delta\big(\delta^{-\frac{N}{q+1}}\Psi^i_{1,0}(\delta^{-1}z-\eta)\big)\big|_{\delta=1}\big\|_{\dot{W}^{1,p^*}(\mathbb{R}^N)}\|\nabla _g \Phi_{\varepsilon,\bar{t},\bar{\xi^0},\bar{\eta}}\|_{q^*}\nonumber\\
 &+\big\|\partial_\delta\big(\delta^{-\frac{N}{p+1}}\Phi^i_{1,0}(\delta^{-1}z-\eta)\big)\big|_{\delta=1}\big\|_{\dot{W}^{1,q^*}(\mathbb{R}^N)}\|\nabla _g \Psi_{\varepsilon,\bar{t},\bar{\xi^0},\bar{\eta}}\|_{p^*}\Big)+O(\varepsilon^{2}\log \varepsilon)=o\big(\varepsilon^{\vartheta}\big),
\end{align}
and
\begin{align}\label{e55}
&\sum\limits_{i=0}^N\sum\limits_{j=1}^kc_{ij}\big\langle\big(\partial _{\eta_{ml}}\Psi_{\delta_j,\xi^0_j,\eta_j}^i,\partial _{\eta_{ml}}\Phi^i_{\delta_j,\xi^0_j,\eta_j}\big),
   \big(\Psi_{\varepsilon,\bar{t},\bar{\xi^0},\bar{\eta}},\Phi_{\varepsilon,\bar{t},\bar{\xi^0},\bar{\eta}}\big)\big\rangle_h\nonumber\\
  \leq&\sum\limits_{i=0}^N\sum\limits_{j=1}^kc_{ij}\delta_{jm}\Big(\big\|\partial_{\eta_l} \Psi^i_{1,0} (z-\eta)\big\|_{\dot{W}^{1,p^*}(\mathbb{R}^N)}\|\nabla _g \Phi_{\varepsilon,\bar{t},\bar{\xi^0},\bar{\eta}}\|_{q^*}+\big\|\partial _{\eta_l}\Phi^i_{1,0}(z-\eta)\big\|_{\dot{W}^{1,q^*}(\mathbb{R}^N)}\|\nabla _g \Psi_{\varepsilon,\bar{t},\bar{\xi^0},\bar{\eta}}\|_{p^*}\Big)\nonumber\\&+O(\varepsilon^{2}\log \varepsilon)
 =o\big(\varepsilon^{\vartheta}\big).
\end{align}

Therefore, by \eqref{e51}-\eqref{e55}, we deduce that the linear system in \eqref{com} has only a trivial solution provided that $\varepsilon>0$ small enough. This ends the proof.
\end{proof}

In the next lemma, we give the asymptotic expansion of $ \mathcal{J}_{\varepsilon}(\mathcal{W}_{\bar{\delta},\bar{\xi^0},\bar{\eta}},\mathcal{H}_{\bar{\delta},\bar{\xi^0},\bar{\eta}})$ as $\varepsilon\rightarrow0$.
\begin{lemma}\label{e5.1}
Under the assumptions of Theorem \ref{th}, if  $\bar{\delta}$ is as in \eqref{solu'}, then %for any real numbers $0<a<b$,
there holds
\begin{align*}
  \mathcal{J}_{\varepsilon}(\mathcal{W}_{\bar{\delta},\bar{\xi^0},\bar{\eta}},\mathcal{H}_{\bar{\delta},\bar{\xi^0},\bar{\eta}})=&
  \sum\limits_{j=1}^ka(\xi^0_j)\big[\frac{2}{N}L_1+c_1\varepsilon-c_2\varepsilon \log \varepsilon+\Psi(t_j,\eta_j)\varepsilon\big]+o(\varepsilon)
\end{align*}
as $\varepsilon\rightarrow0$, $C^1$-uniformly with respect to $\bar{\eta}$ in $(\mathbb{R}^N)^k$ and to $\bar{t}$ in compact subsets of $(\mathbb{R}^+)^k$, where $L_1$ is given in \eqref{chang}, $c_1$ and $c_2$ are given in \eqref{defc1c2}, and $\Psi(t_j,\eta_j)$ is defined as \eqref{deffai}.
\end{lemma}
\begin{proof}
Since $d_g(\xi^0_j,\xi^0_m)>r_0$ for any $j\neq m$, by the definition of  the function $\chi(z)$, there holds
\begin{align*}
  &\mathcal{J}_{\varepsilon}(\mathcal{W}_{\bar{\delta},\bar{\xi^0},\bar{\eta}},\mathcal{H}_{\bar{\delta},\bar{\xi^0},\bar{\eta}})=\sum \limits_{j=1}^k\mathcal{J}_{\varepsilon}(W_{\delta_j,\xi^0_j,\eta_j},H_{\delta_j,\xi^0_j,\eta_j})\\
  =&\sum \limits_{j=1}^k
  a(\xi^0_j)\Big\{\int\limits_{\mathcal{M}}\nabla _g W_{\delta_j,\xi^0_j,\eta_j}\cdot \nabla _g H_{\delta_j,\xi^0_j,\eta_j}d v_g+\int\limits_{\mathcal{M}}hW_{\delta_j,\xi^0_j,\eta_j}H_{\delta_j,\xi^0_j,\eta_j}d v_g\\
  &-\frac{1}{p+1-\alpha\varepsilon}\int\limits_{\mathcal{M}}
  H_{\delta_j,\xi^0_j,\eta_j}^{p+1-\alpha\varepsilon}d v_g-\frac{1}{q+1-\beta\varepsilon}\int\limits_{\mathcal{M}}  W_{\delta_j,\xi^0_j,\eta_j}^{q+1-\beta\varepsilon}d v_g\Big\}\\
  &+\sum \limits_{j=1}^k
  \Big\{\int\limits_{\mathcal{M}}\big(a(x)-a(\xi^0_j)\big)\nabla _g W_{\delta_j,\xi^0_j,\eta_j}\cdot \nabla _g H_{\delta_j,\xi^0_j,\eta_j}d v_g+\int\limits_{\mathcal{M}}\big(a(x)-a(\xi^0_j)\big)hW_{\delta_j,\xi^0_j,\eta_j}H_{\delta_j,\xi^0_j,\eta_j}d v_g\\
  &-\frac{1}{p+1-\alpha\varepsilon}\int\limits_{\mathcal{M}}
 \big(a(x)-a(\xi^0_j)\big) H_{\delta_j,\xi^0_j,\eta_j}^{p+1-\alpha\varepsilon}d v_g-\frac{1}{q+1-\beta\varepsilon}\int\limits_{\mathcal{M}}  \big(a(x)-a(\xi^0_j)\big)W_{\delta_j,\xi^0_j,\eta_j}^{q+1-\beta\varepsilon}d v_g\Big\}\\
 =&:\sum \limits_{j=1}^k
  a(\xi^0_j)(I_1+I_2-I_3-I_4)+\sum \limits_{j=1}^k(I_5+I_6-I_7-I_8).
\end{align*}
First of all, we estimate $I_1$, $I_2$, $I_3$ and $I_4$:
\begin{align*}
  I_1=&\int\limits_{\mathbb{R}^N}\sum\limits_{a,b=1}^Ng^{ab}_{\xi^0_j}(\delta_j z+\delta_j \eta_j)\partial_{z_a}\big(\chi_{\delta_j,\eta_j}U_{1,0}(z)\big)\partial_{z_b}
  \big(\chi_{\delta_j,\eta_j}V_{1,0}(z)\big)\big|g_{\xi^0_j}(\delta_j z+\delta_j \eta_j)\big|^{1/2}dz\\
  =&\int\limits_{\mathbb{R}^N}\sum\limits_{a,b=1}^Ng^{ab}_{\xi^0_j}(\delta_j z+\delta_j \eta_j)\partial_{z_a}U_{1,0}(z)\partial_{z_b}V_{1,0}(z)\big|g_{\xi^0_j}(\delta_j z+\delta_j \eta_j)\big|^{1/2}dz+o(\delta_j^2)\\
  =&\int\limits_{\mathbb{R}^N}\sum\limits_{a,b=1}^N\Big(\delta_{ab}+\frac{\delta_j^2}{2}\sum\limits_{s,t=1}^N\frac{\partial^2g_{\xi^0_j}^{ab}}{\partial{y_s}\partial{y_t}}(0)(z_s+\eta_{js})(z_t+\eta_{jt})\Big)
  \partial_{z_a}U_{1,0}(z)\partial_{z_b}V_{1,0}(z)\\&\times\Big(1-\frac{\delta_j^2}{4}\sum\limits_{s,t,r=1}^N\frac{\partial^2g_{\xi^0_j}^{rr}}{\partial{y_s}\partial{y_t}}(0)(z_s+\eta_{js})(z_t+\eta_{jt})\Big)dz+o(\delta_j^2)
\\=&\int\limits_{\mathbb{R}^N}\nabla U_{1,0}\cdot \nabla V_{1,0}dz+
  \frac{\delta_j^2}{2}\sum\limits_{a,b,s,t=1}^N\frac{\partial^2g_{\xi^0_j}^{ab}}{\partial{y_s}\partial{y_t}}(0)\int\limits_{\mathbb{R}^N}z_sz_t\partial_{z_a}U_{1,0}(z)\partial_{z_b}V_{1,0}(z)dz\\
  &+\frac{\delta_j^2}{2}\sum\limits_{a,b,s,t=1}^N\frac{\partial^2g_{\xi^0_j}^{ab}}{\partial{y_s}\partial{y_t}}(0)\eta_{js}\eta_{jt}\int\limits_{\mathbb{R}^N}\partial_{z_a}U_{1,0}(z)\partial_{z_b}V_{1,0}(z)dz
  \\&-\frac{\delta_j^2}{4}\sum\limits_{s,r=1}^N\frac{\partial^2g_{\xi^0_j}^{rr}}{\partial{y_s}^2}(0)\int\limits_{\mathbb{R}^N}z_s^2\nabla U_{1,0}\cdot \nabla V_{1,0}dz-\frac{\delta_j^2}{4}\sum\limits_{s,t,r=1}^N\frac{\partial^2g_{\xi^0_j}^{rr}}{\partial{y_s}\partial{y_t}}(0)\eta_{js} \eta_{jt}\int\limits_{\mathbb{R}^N}\nabla U_{1,0}\cdot \nabla V_{1,0}dz+o(\delta_j^2)\\
  =&\int\limits_{\mathbb{R}^N}\nabla U_{1,0}\cdot \nabla V_{1,0}dz+
  \frac{\varepsilon t_j}{2}\sum\limits_{a,b,s,t=1}^N\frac{\partial^2g_{\xi^0_j}^{ab}}{\partial{y_s}\partial{y_t}}(0)\int\limits_{\mathbb{R}^N}\frac{U'_{1,0}(z)V'_{1,0}(z)}{|z|^2}z_az_bz_sz_tdz\\
  &+\frac{\varepsilon t_j}{2}\sum\limits_{a,s,t=1}^N\frac{\partial^2g_{\xi^0_j}^{aa}}{\partial{y_s}\partial{y_t}}(0)\eta_{js}\eta_{jt}\int\limits_{\mathbb{R}^N}\frac{U'_{1,0}(z)V'_{1,0}(z)}{|z|^2}z_a^2dz
  -\frac{\varepsilon t_j}{4}\sum\limits_{s,r=1}^N\frac{\partial^2g_{\xi^0_j}^{rr}}{\partial{y_s}^2}(0)\int\limits_{\mathbb{R}^N}U'_{1,0}(z)V'_{1,0}(z)z_s^2dz\\&-\frac{\varepsilon t_j}{4}\sum\limits_{s,t,r=1}^N\frac{\partial^2g_{\xi^0_j}^{rr}}{\partial{y_s}\partial{y_t}}(0)\eta_{js} \eta_{jt}\int\limits_{\mathbb{R}^N}U'_{1,0}(z) V'_{1,0}(z)dz+o(\varepsilon ),
\end{align*}
and
\begin{align*}
  I_2=&\delta_j^2\int\limits_{\mathbb{R}^N}h_{\delta_j,\xi^0_j,\eta_j}\chi^2_{\delta_j,\eta_j}U_{1,0}(z)V_{1,0}(z)\big|g_{\xi^0_j}(\delta_j z+\delta_j \eta_j)\big|^{1/2}dz\\
  =&\delta_j^2\int\limits_{\mathbb{R}^N}\big(h(\xi^0_j)+O(\delta_j)\big)U_{1,0}(z)V_{1,0}(z)\big(1+O(\delta_j^2)\big)dz+o(\delta_j^2)\\
  =&\varepsilon t_jh(\xi^0_j)\int\limits_{\mathbb{R}^N}U_{1,0}(z)V_{1,0}(z)dz+o(\varepsilon),
\end{align*}
where $\chi_{\delta_j,\eta_j}(z)=\chi(\delta_j z+ \delta_j\eta_j)$ and $h_{\delta_j,\xi^0_j,\eta_j}(z)=h(\exp_{\xi^0_j}(\delta_j z+\delta_j \eta_j))$.
Using the Taylor formula, we have
\begin{align*}
  I_3=&\frac{1}{p+1}\int\limits_{\mathcal{M}}
  H_{\delta_j,\xi^0_j,\eta_j}^{p+1}d v_g+\alpha\varepsilon \int\limits_{\mathcal{M}}
  \Big[\frac{H_{\delta_j,\xi^0_j,\eta_j}^{p+1}}{(p+1)^2}-\frac{H_{\delta_j,\xi^0_j,\eta_j}^{p+1}\log H_{\delta_j,\xi^0_j,\eta_j}}{p+1}\Big]d v_g+o(\delta^2_j)\\
  =&\Big(\frac{1}{p+1}+\frac{\alpha\varepsilon}{(p+1)^2}\Big)\int\limits_{\mathbb{R}^N}V_{1,0}^{p+1}\Big(1-\frac{\delta_j^2}{4}\sum\limits_{s,t,r=1}^N\frac{\partial^2g_{\xi^0_j}^{rr}}{\partial{y_s}\partial{y_t}}(0)(z_s+\eta_{js})(z_t+\eta_{jt})\Big)dz
  \\&-\frac{\alpha\varepsilon}{p+1}\int\limits_{\mathbb{R}^N}V_{1,0}^{p+1}\log \big(\delta_j^{-\frac{N}{p+1}}V_{1,0}\big)\Big(1-\frac{\delta_j^2}{4}\sum\limits_{s,t,r=1}^N\frac{\partial^2g_{\xi^0_j}^{rr}}{\partial{y_s}\partial{y_t}}(0)(z_s+\eta_{js})(z_t+\eta_{jt})\Big)dz
 +o(\delta_j^2)\\
 =&\frac{1}{p+1}\int\limits_{\mathbb{R}^N}V_{1,0}^{p+1}dz+\frac{\alpha \varepsilon}{p+1}\Big(\frac{1}{p+1}\int\limits_{\mathbb{R}^N}V_{1,0}^{p+1}dz-\int\limits_{\mathbb{R}^N}V_{1,0}^{p+1}\log V_{1,0}dz\Big)\\
 &+\frac{N\alpha \varepsilon}{2(p+1)^2}\log (\varepsilon t_j)\int_{\mathbb{R}^N}V_{1,0}^{p+1}dz-\frac{\varepsilon t_j}{4(p+1)}\sum\limits_{s,r=1}^N\frac{\partial^2g_{\xi^0_j}^{rr}}{\partial{y_s}^2}(0)\int\limits_{\mathbb{R}^N}V_{1,0}^{p+1}z_s^2dz\\
 &-\frac{\varepsilon t_j}{4(p+1)}\sum\limits_{s,t,r=1}^N\frac{\partial^2g_{\xi^0_j}^{rr}}{\partial{y_s}\partial{y_t}}(0)\eta_{js}\eta_{jt}\int\limits_{\mathbb{R}^N}V_{1,0}^{p+1}dz+o(\varepsilon),
\end{align*}
and
\begin{align*}
  I_4=&\frac{1}{q+1}\int\limits_{\mathcal{M}}
  W_{\delta_j,\xi^0_j,\eta_j}^{q+1}d v_g+\beta\varepsilon \int\limits_{\mathcal{M}}
  \Big[\frac{W_{\delta_j,\xi^0_j,\eta_j}^{q+1}}{(q+1)^2}-\frac{W_{\delta_j,\xi^0_j,\eta_j}^{q+1}\log W_{\delta_j,\xi^0_j,\eta_j}}{q+1}\Big]d v_g+o(\delta^2_j)\\
  =&\Big(\frac{1}{q+1}+\frac{\beta\varepsilon}{(q+1)^2}\Big)\int\limits_{\mathbb{R}^N}U_{1,0}^{q+1}\Big(1-\frac{\delta_j^2}{4}\sum\limits_{s,t,r=1}^N\frac{\partial^2g_{\xi^0_j}^{rr}}{\partial{y_s}\partial{y_t}}(0)(z_s+\eta_{js})(z_t+\eta_{jt})\Big)dz
  \\&-\frac{\beta\varepsilon}{q+1}\int\limits_{\mathbb{R}^N}U_{1,0}^{q+1}\log \big(\delta_j^{-\frac{N}{q+1}}U_{1,0}\big)\Big(1-\frac{\delta_j^2}{4}\sum\limits_{s,t,r=1}^N\frac{\partial^2g_{\xi^0_j}^{rr}}{\partial{y_s}\partial{y_t}}(0)(z_s+\eta_{js})(z_t+\eta_{jt})\Big)dz
 +o(\delta_j^2)\\
 =&\frac{1}{q+1}\int\limits_{\mathbb{R}^N}U_{1,0}^{q+1}dz+\frac{\beta \varepsilon}{q+1}\Big(\frac{1}{q+1}\int\limits_{\mathbb{R}^N}U_{1,0}^{q+1}dz-\int\limits_{\mathbb{R}^N}U_{1,0}^{q+1}\log U_{1,0}dz\Big)\\
 &+\frac{N\beta \varepsilon}{2(q+1)^2}\log (\varepsilon t_j)\int_{\mathbb{R}^N}U_{1,0}^{q+1}dz-\frac{\varepsilon t_j}{4(q+1)}\sum\limits_{s,r=1}^N\frac{\partial^2g_{\xi^0_j}^{rr}}{\partial{y_s}^2}(0)\int\limits_{\mathbb{R}^N}U_{1,0}^{q+1}z_s^2dz\\
 &-\frac{\varepsilon t_j}{4(q+1)}\sum\limits_{s,t,r=1}^N\frac{\partial^2g_{\xi^0_j}^{rr}}{\partial{y_s}\partial{y_t}}(0)\eta_{js}\eta_{jt}\int\limits_{\mathbb{R}^N}U_{1,0}^{q+1}dz+o(\varepsilon).
\end{align*}
Now, we estimate $I_5,I_6,I_7$ and $I_8$. Let $\tilde{a}(z)=a(\exp_{\xi^0_j}(z))$, since $\xi^0_j$
is a non-degenerate critical point of $a(x)$, then
\begin{align*}
  I_5=&\int\limits_{\mathbb{R}^N}\big(\tilde{a}(\delta_j z+\delta_j \eta_j)-\tilde{a}(0)\big)\sum\limits_{a,b=1}^Ng^{ab}_{\xi^0_j}(\delta_j z+\delta_j \eta_j)\partial_{z_a}\big(\chi_{\delta_j,\eta_j}U_{1,0}(z)\big)\partial_{z_b}
  \big(\chi_{\delta_j,\eta_j}V_{1,0}(z)\big)\big|g_{\xi^0_j}(\delta_j z+\delta_j \eta_j)\big|^{1/2}dz\\
  =&\frac{\delta_j^2}{2}\int\limits_{\mathbb{R}^N}\sum\limits_{s,t=1}^N\frac{\partial^2\tilde{a}}{\partial y_s \partial y_t}(0)(z_s+\eta_{js})(z_t+\eta_{jt})\nabla U_{1,0}\cdot \nabla V_{1,0}dz+o(\delta_j^2)\\
  =&\frac{\varepsilon t_j}{2}\sum\limits_{s=1}^N\frac{\partial^2\tilde{a}}{\partial y_s^2 }(0)\int\limits_{\mathbb{R}^N}z_s^2\nabla U_{1,0}\cdot \nabla V_{1,0}dz+\frac{\varepsilon t_j}{2}\sum\limits_{s,t=1}^N\frac{\partial^2\tilde{a}}{\partial y_s \partial y_t}(0)\eta_{js}\eta_{jt}\int\limits_{\mathbb{R}^N}\nabla U_{1,0}\cdot \nabla V_{1,0}dz+o(\varepsilon),
\end{align*}
\begin{align*}
  I_6=&\delta_j^2\int\limits_{\mathbb{R}^N}\big(\tilde{a}(\delta_j z+\delta_j \eta_j)-\tilde{a}(0)\big)h_{\delta_j,\xi^0_j,\eta_j}\chi^2_{\delta_j,\eta_j}U_{1,0}(z)V_{1,0}(z)\big|g_{\xi^0_j}(\delta_j z+\delta_j \eta_j)\big|^{1/2}dz\\
  =&\frac{\delta_j^4}{2}\int\limits_{\mathbb{R}^N}\sum\limits_{s,t=1}^N\frac{\partial^2\tilde{a}}{\partial y_s \partial y_t}(0)(z_s+\eta_{js})(z_t+\eta_{jt})h_{\delta_j,\xi^0_j,\eta_j}\chi^2_{\delta_j,\eta_j}U_{1,0}(z)V_{1,0}(z)dz+o(\delta_j^4)=O(\varepsilon^2)=o(\varepsilon),
\end{align*}
and
\begin{align*}
  I_7=&\frac{1}{p+1}\int\limits_{\mathcal{M}}
  \big(\tilde{a}(\delta_j z+\delta_j \eta_j)-\tilde{a}(0)\big)H_{\delta_j,\xi^0_j,\eta_j}^{p+1}d v_g\\
  &+\alpha\varepsilon \int\limits_{\mathcal{M}}
  \big(\tilde{a}(\delta_j z+\delta_j \eta_j)-\tilde{a}(0)\big)\Big[\frac{H_{\delta_j,\xi^0_j,\eta_j}^{p+1}}{(p+1)^2}-\frac{H_{\delta_j,\xi^0_j,\eta_j}^{p+1}\log H_{\delta_j,\xi^0_j,\eta_j}}{p+1}\Big]d v_g+o(\delta^2_j)\\
  =&\frac{\varepsilon t_j}{2(p+1)}\sum\limits_{s=1}^N\frac{\partial^2\tilde{a}}{\partial y_s^2 }(0)\int\limits_{\mathbb{R}^N}V_{1,0}^{p+1}z_s^2dz+\frac{\varepsilon t_j}{2(p+1)}\sum\limits_{s,t=1}^N\frac{\partial^2\tilde{a}}{\partial y_s \partial y_t}(0)\eta_{js}\eta_{jt}\int\limits_{\mathbb{R}^N}V_{1,0}^{p+1}dz+o(\varepsilon),
  \end{align*}
\begin{align*}
  I_8=&\frac{1}{q+1}\int\limits_{\mathcal{M}}
  \big(\tilde{a}(\delta_j z+\delta_j \eta_j)-\tilde{a}(0)\big)W_{\delta_j,\xi^0_j,\eta_j}^{q+1}d v_g\\
  &+\beta\varepsilon \int\limits_{\mathcal{M}}
  \big(\tilde{a}(\delta_j z+\delta_j \eta_j)-\tilde{a}(0)\big)\Big[\frac{W_{\delta_j,\xi^0_j,\eta_j}^{q+1}}{(q+1)^2}-\frac{W_{\delta_j,\xi^0_j,\eta_j}^{q+1}\log W_{\delta_j,\xi^0_j,\eta_j}}{q+1}\Big]d v_g+o(\delta^2_j)\\
  =&\frac{\varepsilon t_j}{2(q+1)}\sum\limits_{s=1}^N\frac{\partial^2\tilde{a}}{\partial y_s^2 }(0)\int\limits_{\mathbb{R}^N}U_{1,0}^{q+1}z_s^2dz+\frac{\varepsilon t_j}{2(q+1)}\sum\limits_{s,t=1}^N\frac{\partial^2\tilde{a}}{\partial y_s \partial y_t}(0)\eta_{js}\eta_{jt}\int\limits_{\mathbb{R}^N}U_{1,0}^{q+1}dz+o(\varepsilon).
  \end{align*}

Therefore, taking into account that
\begin{equation}\label{ge1}
  \sum\limits_{a,b=1}^N\frac{\partial g_{\xi^0_j}^{aa}}{\partial y_b^2}(0)-\sum\limits_{a,b=1}^N\frac{\partial g_{\xi^0_j}^{ab}}{\partial y_a \partial y_b}(0)=Scal_g(\xi^0_j),
\end{equation}
and
\begin{equation}\label{ge2}
  \Delta_g a(\xi^0_j)=\sum\limits_{s=1}^N\frac{\partial^2 \tilde{a}}{\partial y_s^2}(0),\quad D^2_ga(\xi^0_j)[\eta_j,\eta_j]=\sum\limits_{s,t=1}^N \frac{\partial^2 \tilde{a}}{\partial y_s \partial y_t}(0)\eta_{js}\eta_{jt},
\end{equation}
we get the $C^0$-estimate.

Next, we prove the $C^1$-estimate. For any $j=1,2,\cdots,k$, define
 \begin{equation*}
   A(\delta_j,\xi^0_j,\eta_j)=\int\limits_{\mathcal{M}}a(x)\nabla _g W_{\delta_j,\xi^0_j,\eta_j}\cdot \nabla _g H_{\delta_j,\xi^0_j,\eta_j}d v_g,\quad B(\delta_j,\xi^0_j,\eta_j)=\int\limits_{\mathcal{M}}a(x)hW_{\delta_j,\xi^0_j,\eta_j}H_{\delta_j,\xi^0_j,\eta_j}d v_g,
 \end{equation*}
 and
 \begin{equation*}
   C(\delta_j,\xi^0_j,\eta_j)=\frac{1}{p+1-\alpha\varepsilon}\int\limits_{\mathcal{M}}
  a(x)H_{\delta_j,\xi^0_j,\eta_j}^{p+1-\alpha\varepsilon}d v_g,\quad D(\delta_j,\xi^0_j,\eta_j)=\frac{1}{q+1-\beta\varepsilon}\int\limits_{\mathcal{M}}  a(x)W_{\delta_j,\xi^0_j,\eta_j}^{q+1-\beta\varepsilon}d v_g.
 \end{equation*}
Let $\delta'_j=\partial_{t_j}\delta_j$, then $\delta'_j\delta_j=\frac{\varepsilon}{2}$. Set $\tilde{h}(z)=h(\exp_{\xi^0_j}(z))$, since $\xi^0_j$
is a non-degenerate critical point of $a(x)$, we have
\begin{align*}
  \partial _{t_j}A(\delta_j,\xi^0_j,\eta_j)=&\delta'_j\partial_{\delta_j}A(\delta_j,\xi^0_j,\eta_j)\\
  =&\delta'_j\frac{\partial}{\partial{\delta_j}}\int\limits_{\mathbb{R}^N}\tilde{a}(\delta_jz+\delta_j\eta_j)\sum\limits_{a,b=1}^Ng^{ab}_{\xi^0_j}(\delta_j z+\delta_j \eta_j)\partial_{z_a}\big(\chi_{\delta_j,\eta_j}U_{1,0}(z)\big)\partial_{z_b}
  \big(\chi_{\delta_j,\eta_j}V_{1,0}(z)\big)\\
  &\times \big|g_{\xi^0_j}(\delta_j z+\delta_j \eta_j)\big|^{1/2}dz\\
  =&\delta'_j\int\limits_{\mathbb{R}^N}\sum\limits_{s=1}^N\frac{\partial\tilde{a}}{\partial y_s}(\delta_jz+\delta_j\eta_j)(z_s+\eta_{js})\sum\limits_{a,b=1}^Ng^{ab}_{\xi^0_j}(\delta_j z+\delta_j \eta_j)\partial_{z_a}U_{1,0}(z)\partial_{z_b}
  V_{1,0}(z)dz\\
  &+\delta'_j\int\limits_{\mathbb{R}^N}\tilde{a}(\delta_jz+\delta_j\eta_j)\sum\limits_{a,b,s=1}^N\frac{\partial g^{ab}_{\xi^0_j}}{\partial y_s}(\delta_j z+\delta_j \eta_j)(z_s+\eta_{js})\partial_{z_a}U_{1,0}(z)\partial_{z_b}
  V_{1,0}(z)dz\\
  &-\delta'_j\int\limits_{\mathbb{R}^N}\tilde{a}(\delta_jz+\delta_j\eta_j)\sum\limits_{a,b=1}^Ng^{ab}_{\xi^0_j}(\delta_j z+\delta_j \eta_j)\partial_{z_a}U_{1,0}(z)\partial_{z_b}
  V_{1,0}(z)\\
  &\times \frac{\delta_j}{2}\sum\limits_{s,t,r=1}^N\frac{\partial^2g_{\xi^0_j}^{rr}}{\partial{y_s}\partial{y_t}}(0)(z_s+\eta_{js})(z_t+\eta_{jt})dz+o(\varepsilon)\\
  =&\delta'_j\delta_j\int\limits_{\mathbb{R}^N}\sum\limits_{s,t=1}^N\frac{\partial^2\tilde{a}}{\partial y_s \partial y_t}(0)(z_s+\eta_{js})(z_t+\eta_{jt})\nabla U_{1,0}\cdot \nabla V_{1,0}dz\\
  &+\delta'_j\delta_j\int\limits_{\mathbb{R}^N}\tilde{a}(0)\sum\limits_{a,b,s,t=1}^N\frac{\partial ^2g^{ab}_{\xi^0_j}}{\partial y_s\partial y_t}(0)(z_s+\eta_{js})( z_t+\eta_{jt})\partial_{z_a}U_{1,0}(z)\partial_{z_b}
  V_{1,0}(z)dz\\
  &-\frac{\delta'_j \delta_j}{2}\int\limits_{\mathbb{R}^N}\tilde{a}(0)\nabla U_{1,0}\cdot \nabla
  V_{1,0}\sum\limits_{s,t,r=1}^N\frac{\partial^2g_{\xi^0_j}^{rr}}{\partial{y_s}\partial{y_t}}(0)(z_s+\eta_{js})(z_t+\eta_{jt})dz+o(\varepsilon)\\
  =&\frac{\varepsilon}{2}\sum\limits_{s=1}^N\frac{\partial^2\tilde{a}}{\partial y_s ^2}(0)\int\limits_{\mathbb{R}^N}z_s^2\nabla U_{1,0}\cdot \nabla V_{1,0}dz+\frac{\varepsilon}{2}\sum\limits_{s,t=1}^N\frac{\partial^2\tilde{a}}{\partial y_s \partial y_t}(0)\eta_{js}\eta_{jt}\int\limits_{\mathbb{R}^N}\nabla U_{1,0}\cdot \nabla V_{1,0}dz\\
  &+\frac{\varepsilon}{2}\tilde{a}(0)\sum\limits_{a,b,s,t=1}^N\frac{\partial^2 g^{ab}_{\xi^0_j}}{\partial y_s\partial y_t}(0)\int\limits_{\mathbb{R}^N}
  \frac{U'_{1,0}(z)V'_{1,0}(z)}{|z|^2}z_az_bz_sz_tdz\\
  &+\frac{\varepsilon}{2}\tilde{a}(0)\sum\limits_{a,s,t=1}^N\frac{\partial^2 g^{aa}_{\xi^0_j}}{\partial y_s\partial y_t}(0)\eta_{js}\eta_{jt}\int\limits_{\mathbb{R}^N}
  \frac{U'_{1,0}(z)V'_{1,0}(z)}{|z|^2}z_a^2dz\\
  &-\frac{\varepsilon}{4}\tilde{a}(0)\sum\limits_{s,r=1}^N\frac{\partial^2g_{\xi^0_j}^{rr}}{\partial{y_s}^2}(0)\int\limits_{\mathbb{R}^N}z_s^2\nabla U_{1,0}\cdot \nabla
  V_{1,0}dz\\
  &-\frac{\varepsilon}{4}\tilde{a}(0)\sum\limits_{s,t,r=1}^N\frac{\partial^2g_{\xi^0_j}^{rr}}{\partial{y_s}\partial{y_t}}(0)\eta_{js}\eta_{jt}\int\limits_{\mathbb{R}^N}\nabla U_{1,0}\cdot \nabla
  V_{1,0}dz+o(\varepsilon),
\end{align*}
\begin{align*}
  \partial _{t_j}B(\delta_j,\xi^0_j,\eta_j)=&\delta'_j\partial_{\delta_j}B(\delta_j,\xi^0_j,\eta_j)\\
  =&\delta'_j\frac{\partial}{\partial{\delta_j}}\int\limits_{\mathbb{R}^N}\delta_j^2\tilde{a}(\delta_jz+\delta_j\eta_j)\tilde{h}(\delta_jz+\delta_j\eta_j)\chi_{\delta_j,\eta_j}^2U_{1,0}(z)V_{1,0}(z)\big|g_{\xi^0_j}(\delta_j z+\delta_j \eta_j)\big|^{1/2}dz\\
  =&2\delta'_j\delta_j\int\limits_{\mathbb{R}^N}\tilde{a}(0)\tilde{h}(0)U_{1,0}(z)V_{1,0}(z)dz+o(\varepsilon)
  =\varepsilon\tilde{a}(0)\tilde{h}(0)\int\limits_{\mathbb{R}^N}U_{1,0}(z)V_{1,0}(z)dz+o(\varepsilon),
\end{align*}
and
\begin{align*}
&\partial _{t_j}C(\delta_j,\xi^0_j,\eta_j)\\=&\delta'_j\partial_{\delta_j}C(\delta_j,\xi^0_j,\eta_j)\\
  =&\delta'_j\Big(\frac{1}{p+1}+\frac{\alpha\varepsilon}{(p+1)^2}\Big)\frac{\partial}{\partial{\delta_j}}\int\limits_{\mathbb{R}^N}\tilde{a}(\delta_jz+\delta_j\eta_j)V_{1,0}^{p+1}\Big(1-\frac{\delta_j^2}{4}\sum\limits_{s,t,r=1}^N\frac{\partial^2g_{\xi^0_j}^{rr}}{\partial{y_s}\partial{y_t}}(0)(z_s+\eta_{js})(z_t+\eta_{jt})\Big)dz\\
  &-\frac{\alpha \delta'_j\varepsilon}{p+1}\frac{\partial}{\partial{\delta_j}}\int\limits_{\mathbb{R}^N}\tilde{a}(\delta_jz+\delta_j\eta_j)V_{1,0}^{p+1}\log \big(\delta_j^{-\frac{N}{p+1}}V_{1,0}\big)\Big(1-\frac{\delta_j^2}{4}\sum\limits_{s,t,r=1}^N\frac{\partial^2g_{\xi^0_j}^{rr}}{\partial{y_s}\partial{y_t}}(0)(z_s+\eta_{js})(z_t+\eta_{jt})\Big)dz+o(\varepsilon)\\
  =&\frac{\delta'_j}{p+1}\int\limits_{\mathbb{R}^N}\sum\limits_{s=1}^N\frac{\partial\tilde{a}}{\partial y_s}(\delta_jz+\delta_j\eta_j)(z_s+\eta_{js})V_{1,0}^{p+1}dz+
  \frac{N\alpha \delta'_j\varepsilon}{(p+1)^2\delta_j}\int\limits_{\mathbb{R}^N}\tilde{a}(\delta_jz+\delta_j\eta_j)V_{1,0}^{p+1}dz
  \\&-\frac{\delta_j'\delta_j}{2(p+1)}\int\limits_{\mathbb{R}^N}\tilde{a}(\delta_jz+\delta_j\eta_j)V_{1,0}^{p+1}\sum\limits_{s,t,r=1}^N\frac{\partial^2g_{\xi^0_j}^{rr}}{\partial{y_s}\partial{y_t}}(0)(z_s+\eta_{js})(z_t+\eta_{jt})dz
+o(\varepsilon)\\
=&\frac{\varepsilon}{2(p+1)}\sum\limits_{s=1}^N\frac{\partial\tilde{a}}{\partial y_s^2}(0)\int\limits_{\mathbb{R}^N}z_s^2V_{1,0}^{p+1}dz+\frac{\varepsilon}{2(p+1)}\sum\limits_{s,t=1}^N\frac{\partial\tilde{a}}{\partial y_s\partial y_t}(0)\eta_{js}\eta_{jt}\int\limits_{\mathbb{R}^N}V_{1,0}^{p+1}dz+\frac{N\alpha \tilde{a}(0)\varepsilon}{2(p+1)^2t_j}\int\limits_{\mathbb{R}^N}V_{1,0}^{p+1}dz\\&-\frac{\tilde{a}(0)\varepsilon}{4(p+1)}\sum\limits_{s,r=1}^N\frac{\partial^2g_{\xi^0_j}^{rr}}{\partial{y_s}^2}(0)\int\limits_{\mathbb{R}^N}V_{1,0}^{p+1}z_s^2dz-\frac{\tilde{a}(0)\varepsilon}{4(p+1)}\sum\limits_{s,t,r=1}^N\frac{\partial^2g_{\xi^0_j}^{rr}}{\partial{y_s}\partial{y_t}}(0)\eta_{js}\eta_{jt}\int\limits_{\mathbb{R}^N}V_{1,0}^{p+1}dz
+o(\varepsilon),
\end{align*}
\begin{align*}
&\partial _{t_j}D(\delta_j,\xi^0_j,\eta_j)\\=&\delta'_j\partial_{\delta_j}D(\delta_j,\xi^0_j,\eta_j)\\
  =&\delta'_j\Big(\frac{1}{q+1}+\frac{\beta\varepsilon}{(q+1)^2}\Big)\frac{\partial}{\partial{\delta_j}}\int\limits_{\mathbb{R}^N}\tilde{a}(\delta_jz+\delta_j\eta_j)U_{1,0}^{q+1}\Big(1-\frac{\delta_j^2}{4}\sum\limits_{s,t,r=1}^N\frac{\partial^2g_{\xi^0_j}^{rr}}{\partial{y_s}\partial{y_t}}(0)(z_s+\eta_{js})(z_t+\eta_{jt})\Big)dz\\
  &-\frac{\beta \delta'_j\varepsilon}{q+1}\frac{\partial}{\partial{\delta_j}}\int\limits_{\mathbb{R}^N}\tilde{a}(\delta_jz+\delta_j\eta_j)U_{1,0}^{q+1}\log \big(\delta_j^{-\frac{N}{q+1}}U_{1,0}\big)\Big(1-\frac{\delta_j^2}{4}\sum\limits_{s,t,r=1}^N\frac{\partial^2g_{\xi^0_j}^{rr}}{\partial{y_s}\partial{y_t}}(0)(z_s+\eta_{js})(z_t+\eta_{jt})\Big)dz+o(\varepsilon)\\
  =&\frac{\delta'_j}{q+1}\int\limits_{\mathbb{R}^N}\sum\limits_{s=1}^N\frac{\partial\tilde{a}}{\partial y_s}(\delta_jz+\delta_j\eta_j)(z_s+\eta_{js})U_{1,0}^{q+1}dz+\frac{N\beta \delta'_j\varepsilon}{(q+1)^2\delta_j}\int\limits_{\mathbb{R}^N}\tilde{a}(\delta_jz+\delta_j\eta_j)U_{1,0}^{q+1}dz
  \\&-\frac{\delta_j'\delta_j}{2(q+1)}\int\limits_{\mathbb{R}^N}\tilde{a}(\delta_jz+\delta_j\eta_j)U_{1,0}^{q+1}\sum\limits_{s,t,r=1}^N\frac{\partial^2g_{\xi^0_j}^{rr}}{\partial{y_s}\partial{y_t}}(0)(z_s+\eta_{js})(z_t+\eta_{jt})dz
+o(\varepsilon)\\
=&\frac{\varepsilon}{2(q+1)}\sum\limits_{s=1}^N\frac{\partial\tilde{a}}{\partial y_s^2}(0)\int\limits_{\mathbb{R}^N}z_s^2U_{1,0}^{q+1}dz+\frac{\varepsilon}{2(q+1)}\sum\limits_{s,t=1}^N\frac{\partial\tilde{a}}{\partial y_s\partial y_t}(0)\eta_{js}\eta_{jt}\int\limits_{\mathbb{R}^N}U_{1,0}^{q+1}dz+\frac{N\beta \tilde{a}(0)\varepsilon}{2(q+1)^2t_j}\int\limits_{\mathbb{R}^N}U_{1,0}^{q+1}dz\\
&-\frac{\tilde{a}(0)\varepsilon}{4(q+1)}\sum\limits_{s,r=1}^N\frac{\partial^2g_{\xi^0_j}^{rr}}{\partial{y_s}^2}(0)\int\limits_{\mathbb{R}^N}U_{1,0}^{q+1}z_s^2dz-\frac{\tilde{a}(0)\varepsilon}{4(q+1)}\sum\limits_{s,t,r=1}^N\frac{\partial^2g_{\xi^0_j}^{rr}}{\partial{y_s}\partial{y_t}}(0)\eta_{js}\eta_{jt}\int\limits_{\mathbb{R}^N}U_{1,0}^{q+1}dz
+o(\varepsilon).
\end{align*}
 Using \eqref{ge1} and \eqref{ge2} again, we complete the $C^1$-estimate with respect to $\bar{t}$ in compact subsets of $(\mathbb{R}^+)^k$.

In a similar way, we consider the derivative with respect to $\bar{\eta}$ in $(\mathbb{R}^N)^k$. For any $1\leq s  \leq N$, we have
\begin{align*}
  \partial _{\eta_{js}}A(\delta_j,\xi^0_j,\eta_j)=&\frac{\partial}{\partial \eta_{js}}\int\limits_{\mathbb{R}^N}\tilde{a}(\delta_jz+\delta_j\eta_j)\sum\limits_{a,b=1}^Ng^{ab}_{\xi^0_j}(\delta_j z+\delta_j \eta_j)\partial_{z_a}\big(\chi_{\delta_j,\eta_j}U_{1,0}(z)\big)\partial_{z_b}
  \big(\chi_{\delta_j,\eta_j}V_{1,0}(z)\big)\\
  &\times \big|g_{\xi^0_j}(\delta_j z+\delta_j \eta_j)\big|^{1/2}dz\\
  =&\delta_j\int\limits_{\mathbb{R}^N}\sum\limits_{s=1}^N\frac{\partial \tilde{a}}{\partial y_s}(\delta_jz+\delta_j\eta_j)\sum\limits_{a,b=1}^Ng^{ab}_{\xi^0_j}(\delta_j z+\delta_j \eta_j)\partial_{z_a}U_{1,0}(z)\partial_{z_b}
  V_{1,0}(z)dz\\
  &+\delta_j\int\limits_{\mathbb{R}^N}\tilde{a}(\delta_jz+\delta_j\eta_j)\sum\limits_{a,b,s=1}^N\frac{\partial g^{ab}_{\xi^0_j}}{\partial y_s}(\delta_j z+\delta_j \eta_j)\partial_{z_a}U_{1,0}(z)\partial_{z_b}
  V_{1,0}(z)dz
  \\&-\frac{\delta_j^2}{2}\int\limits_{\mathbb{R}^N}\tilde{a}(\delta_jz+\delta_j\eta_j)\sum\limits_{a,b=1}^Ng^{ab}_{\xi^0_j}(\delta_j z+\delta_j \eta_j)\partial_{z_a}U_{1,0}(z)\partial_{z_b}
  V_{1,0}(z) \\
  &\times\sum\limits_{s,t,r=1}^N\frac{\partial^2g_{\xi^0_j}^{rr}}{\partial{y_s}\partial{y_t}}(0)(z_t+\eta_{jt})dz+o(\varepsilon)\\
  =&\delta_j^2\int\limits_{\mathbb{R}^N}\sum\limits_{s,t=1}^N\frac{\partial^2 \tilde{a}}{\partial y_s\partial y_t}(0)(z_t+\eta_{jt})\nabla U_{1,0}\cdot \nabla V_{1,0}dz\\&+
  \delta_j^2\int\limits_{\mathbb{R}^N}\tilde{a}(0)\sum\limits_{a,b,s,t=1}^N\frac{\partial^2 g^{ab}_{\xi^0_j}}{\partial y_s\partial y_t}(0)(z_t+ \eta_{jt})\partial_{z_a}U_{1,0}(z)\partial_{z_b}
  V_{1,0}(z)dz\\
  &-\frac{\delta_j^2}{2}\int\limits_{\mathbb{R}^N}\tilde{a}(0)\nabla U_{1,0} \cdot \nabla V_{1,0} \sum\limits_{s,t,r=1}^N\frac{\partial^2g_{\xi^0_j}^{rr}}{\partial{y_s}\partial{y_t}}(0)(z_t+\eta_{jt})dz+o(\varepsilon)\\
  =&\varepsilon t_j\sum\limits_{s,t=1}^N\frac{\partial^2 \tilde{a}}{\partial y_s\partial y_t}(0)\eta_{jt}\int\limits_{\mathbb{R}^N}\nabla U_{1,0}\cdot \nabla V_{1,0}dz\\&+
  \varepsilon t_j\tilde{a}(0)\sum\limits_{a,s,t=1}^N\frac{\partial^2 g^{aa}_{\xi^0_j}}{\partial y_s\partial y_t}(0) \eta_{jt}\int\limits_{\mathbb{R}^N}\frac{U'_{1,0}(z)V'_{1,0}(z)}{|z|^2}z_a^2dz\\
  &-\frac{ \varepsilon t_j \tilde{a}(0)}{2}\sum\limits_{s,t,r=1}^N\frac{\partial^2g_{\xi^0_j}^{rr}}{\partial{y_s}\partial{y_t}}(0)\eta_{jt}\int\limits_{\mathbb{R}^N}\nabla U_{1,0} \cdot \nabla V_{1,0} dz+o(\varepsilon),
\end{align*}
\begin{align*}
  \partial _{\eta_{js}}B(\delta_j,\xi^0_j,\eta_j)=&\frac{\partial}{\partial \eta_{js}}\int\limits_{\mathbb{R}^N}\delta_j^2\tilde{a}(\delta_jz+\delta_j\eta_j)\tilde{h}(\delta_jz+\delta_j\eta_j)\chi_{\delta_j,\eta_j}^2U_{1,0}(z)V_{1,0}(z)\big|g_{\xi^0_j}(\delta_j z+\delta_j \eta_j)\big|^{1/2}dz=o(\varepsilon),
\end{align*}
and
\begin{align*}
   &\partial _{\eta_{js}}C(\delta_j,\xi^0_j,\eta_j)\\
   =&\Big(\frac{1}{p+1}+\frac{\alpha\varepsilon}{(p+1)^2}\Big)\frac{\partial}{\partial{\eta_{js}}}\int\limits_{\mathbb{R}^N}\tilde{a}(\delta_jz+\delta_j\eta_j)V_{1,0}^{p+1}\Big(1-\frac{\delta_j^2}{4}\sum\limits_{s,t,r=1}^N\frac{\partial^2g_{\xi^0_j}^{rr}}{\partial{y_s}\partial{y_t}}(0)(z_s+\eta_{js})(z_t+\eta_{jt})\Big)dz\\
  &-\frac{\alpha\varepsilon}{p+1}\frac{\partial}{\partial{\eta_{js}}}\int\limits_{\mathbb{R}^N}\tilde{a}(\delta_jz+\delta_j\eta_j)V_{1,0}^{p+1}\log \big(\delta_j^{-\frac{N}{p+1}}V_{1,0}\big)\Big(1-\frac{\delta_j^2}{4}\sum\limits_{s,t,r=1}^N\frac{\partial^2g_{\xi^0_j}^{rr}}{\partial{y_s}\partial{y_t}}(0)(z_s+\eta_{js})(z_t+\eta_{jt})\Big)dz+o(\varepsilon)\\
  =&\frac{\delta_j}{p+1}\int\limits_{\mathbb{R}^N}\sum\limits_{s=1}^N\frac{\partial\tilde{a}}{\partial y_s}(\delta_jz+\delta_j\eta_j)V_{1,0}^{p+1}
 dz-\frac{\delta_j^2}{2(p+1)}\int\limits_{\mathbb{R}^N}\tilde{a}(\delta_jz+\delta_j\eta_j)V_{1,0}^{p+1}
 \sum\limits_{s,t,r=1}^N\frac{\partial^2g_{\xi^0_j}^{rr}}{\partial{y_s}\partial{y_t}}(0)(z_t+\eta_{jt})dz\\
 =&\frac{\delta_j^2}{p+1}\int\limits_{\mathbb{R}^N}\sum\limits_{s,t=1}^N\frac{\partial^2\tilde{a}}{\partial y_s\partial y_t}(0)(z_t+\eta_{jt})V_{1,0}^{p+1}
 dz-\frac{\delta_j^2}{2(p+1)}\int\limits_{\mathbb{R}^N}\tilde{a}(0)V_{1,0}^{p+1}
 \sum\limits_{s,t,r=1}^N\frac{\partial^2g_{\xi^0_j}^{rr}}{\partial{y_s}\partial{y_t}}(0)(z_t+\eta_{jt})dz\\
 =&\frac{\varepsilon t_j}{p+1}\sum\limits_{s,t=1}^N\frac{\partial^2\tilde{a}}{\partial y_s\partial y_t}(0)\eta_{jt}\int\limits_{\mathbb{R}^N}
 V_{1,0}^{p+1}
 dz-\frac{\varepsilon t_j \tilde{a}(0)}{2(p+1)}\sum\limits_{s,t,r=1}^N\frac{\partial^2g_{\xi^0_j}^{rr}}{\partial{y_s}\partial{y_t}}(0)\eta_{jt}\int\limits_{\mathbb{R}^N}V_{1,0}^{p+1}
 dz,
\end{align*}
\begin{align*}
   &\partial _{\eta_{js}}D(\delta_j,\xi^0_j,\eta_j)\\
   =&
   \Big(\frac{1}{q+1}+\frac{\beta\varepsilon}{(q+1)^2}\Big)\frac{\partial}{\partial{\eta_{js}}}\int\limits_{\mathbb{R}^N}\tilde{a}(\delta_jz+\delta_j\eta_j)U_{1,0}^{q+1}\Big(1-\frac{\delta_j^2}{4}\sum\limits_{s,t,r=1}^N\frac{\partial^2g_{\xi^0_j}^{rr}}{\partial{y_s}\partial{y_t}}(0)(z_s+\eta_{js})(z_t+\eta_{jt})\Big)dz\\
  &-\frac{\beta\varepsilon}{q+1}\frac{\partial}{\partial{\eta_{js}}}\int\limits_{\mathbb{R}^N}\tilde{a}(\delta_jz+\delta_j\eta_j)U_{1,0}^{q+1}\log \big(\delta_j^{-\frac{N}{q+1}}U_{1,0}\big)\Big(1-\frac{\delta_j^2}{4}\sum\limits_{s,t,r=1}^N\frac{\partial^2g_{\xi^0_j}^{rr}}{\partial{y_s}\partial{y_t}}(0)(z_s+\eta_{js})(z_t+\eta_{jt})\Big)dz+o(\varepsilon)\\
  =&\frac{\delta_j}{q+1}\int\limits_{\mathbb{R}^N}\sum\limits_{s=1}^N\frac{\partial\tilde{a}}{\partial y_s}(\delta_jz+\delta_j\eta_j)U_{1,0}^{q+1}
 dz-\frac{\delta_j^2}{2(q+1)}\int\limits_{\mathbb{R}^N}\tilde{a}(\delta_jz+\delta_j\eta_j)U_{1,0}^{q+1}
 \sum\limits_{s,t,r=1}^N\frac{\partial^2g_{\xi^0_j}^{rr}}{\partial{y_s}\partial{y_t}}(0)(z_t+\eta_{jt})dz\\
 =&\frac{\delta_j^2}{q+1}\int\limits_{\mathbb{R}^N}\sum\limits_{s,t=1}^N\frac{\partial^2\tilde{a}}{\partial y_s\partial y_t}(0)(z_t+\eta_{jt})U_{1,0}^{q+1}
 dz-\frac{\delta_j^2}{2(q+1)}\int\limits_{\mathbb{R}^N}\tilde{a}(0)U_{1,0}^{q+1}
 \sum\limits_{s,t,r=1}^N\frac{\partial^2g_{\xi^0_j}^{rr}}{\partial{y_s}\partial{y_t}}(0)(z_t+\eta_{jt})dz\\
 =&\frac{\varepsilon t_j}{q+1}\sum\limits_{s,t=1}^N\frac{\partial^2\tilde{a}}{\partial y_s\partial y_t}(0)\eta_{jt}\int\limits_{\mathbb{R}^N}
 U_{1,0}^{q+1}
 dz-\frac{\varepsilon t_j \tilde{a}(0)}{2(q+1)}\sum\limits_{s,t,r=1}^N\frac{\partial^2g_{\xi^0_j}^{rr}}{\partial{y_s}\partial{y_t}}(0)\eta_{jt}\int\limits_{\mathbb{R}^N}U_{1,0}^{q+1}
 dz.
\end{align*}
So we have the $C^1$-estimate with respect to $\bar{\eta}$ in $(\mathbb{R}^N)^k$.
\end{proof}

We now give the asymptotic expansion of the function $\widetilde{\mathcal{J}}_\varepsilon$ defined in \eqref{defj} as $\varepsilon\rightarrow0$.
\begin{lemma}\label{topr}
Under the assumptions  of Theorem \ref{th}, if $\bar{\delta}$ is as in \eqref{solu'}, then there holds
\begin{align*}
 \widetilde{\mathcal{J}}_{\varepsilon}(\bar{t},\bar{\eta})= \mathcal{J}_{\varepsilon}(\mathcal{W}_{\bar{\delta},\bar{\xi^0},\bar{\eta}},\mathcal{H}_{\bar{\delta},\bar{\xi^0},\bar{\eta}})+o(\varepsilon),
\end{align*}
as $\varepsilon\rightarrow0$,  $C^1$-uniformly with respect to $\bar{\eta}$ in $(\mathbb{R}^N)^k$ and to $\bar{t}$ in compact subsets of $(\mathbb{R}^+)^k$.
\end{lemma}
\begin{proof}
It's holds that
\begin{align*}
  &\widetilde{\mathcal{J}}_{\varepsilon}(\bar{t},\bar{\eta})- \mathcal{J}_{\varepsilon}(\mathcal{W}_{\bar{\delta},\bar{\xi^0},\bar{\eta}},\mathcal{H}_{\bar{\delta},\bar{\xi^0},\bar{\eta}})\\
  =&\int\limits_{\mathcal{M}}a(x)\big(\nabla_g \mathcal{W}_{\bar{\delta},\bar{\xi^0},\bar{\eta}}\cdot \nabla \Phi_{\varepsilon,\bar{t},\bar{\xi^0},\bar{\eta}}+h\mathcal{W}_{\bar{\delta},\bar{\xi^0},\bar{\eta}}\Phi_{\varepsilon,\bar{t},\bar{\xi^0},\bar{\eta}}-f_\varepsilon(\mathcal{H}_{\bar{\delta},\bar{\xi^0},\bar{\eta}})\Phi_{\varepsilon,\bar{t},\bar{\xi^0},\bar{\eta}}\big)d v_g\\
  &+\int\limits_{\mathcal{M}}a(x)\big(\nabla_g \mathcal{H}_{\bar{\delta},\bar{\xi^0},\bar{\eta}}\cdot \nabla \Psi_{\varepsilon,\bar{t},\bar{\xi^0},\bar{\eta}}+h\mathcal{H}_{\bar{\delta},\bar{\xi^0},\bar{\eta}}\Psi_{\varepsilon,\bar{t},\bar{\xi^0},\bar{\eta}}-g_\varepsilon(\mathcal{W}_{\bar{\delta},\bar{\xi^0},\bar{\eta}})\Psi_{\varepsilon,\bar{t},\bar{\xi^0},\bar{\eta}}\big)d v_g\\
  &+\int\limits_{\mathcal{M}}a(x)\big(\nabla_g \Psi_{\varepsilon,\bar{t},\bar{\xi^0},\bar{\eta}}\cdot \nabla _g \Phi_{\varepsilon,\bar{t},\bar{\xi^0},\bar{\eta}}+h\Psi_{\varepsilon,\bar{t},\bar{\xi^0},\bar{\eta}}\Phi_{\varepsilon,\bar{t},\bar{\xi^0},\bar{\eta}}\big)dv_g\\
  &-\int\limits_{\mathcal{M}}a(x)\big(F_\varepsilon(\mathcal{H}_{\bar{\delta},\bar{\xi^0},\bar{\eta}}+\Phi_{\varepsilon,\bar{t},\bar{\xi^0},\bar{\eta}})-
  F_\varepsilon(\mathcal{H}_{\bar{\delta},\bar{\xi^0},\bar{\eta}})-f_\varepsilon(\mathcal{H}_{\bar{\delta},\bar{\xi^0},\bar{\eta}}) \Phi_{\varepsilon,\bar{t},\bar{\xi^0},\bar{\eta}}\big)dv_g\\
  &-\int\limits_{\mathcal{M}}a(x)\big(G_\varepsilon(\mathcal{W}_{\bar{\delta},\bar{\xi^0},\bar{\eta}}+\Psi_{\varepsilon,\bar{t},\bar{\xi^0},\bar{\eta}})-
  G_\varepsilon(\mathcal{W}_{\bar{\delta},\bar{\xi^0},\bar{\eta}})-g_\varepsilon(\mathcal{W}_{\bar{\delta},\bar{\xi^0},\bar{\eta}}) \Psi_{\varepsilon,\bar{t},\bar{\xi^0},\bar{\eta}}\big)dv_g,
\end{align*}
where $F_\varepsilon(u)=\int\limits_{0}^uf_\varepsilon(s)ds$, $G_\varepsilon(u)=\int\limits_{0}^ug_\varepsilon(s)ds$.
Since
\begin{equation*}
  \int\limits_{\mathcal{M}}a(x)\nabla_g \mathcal{W}_{\bar{\delta},\bar{\xi^0},\bar{\eta}}\cdot \nabla \Phi_{\varepsilon,\bar{t},\bar{\xi^0},\bar{\eta}}dv_g=-\int\limits_{\mathcal{M}}\Phi_{\varepsilon,\bar{t},\bar{\xi^0},\bar{\eta}}\nabla_g a(x)\cdot\nabla_g\mathcal{W}_{\bar{\delta},\bar{\xi^0},\bar{\eta}} dv_g-\int\limits_{\mathcal{M}}a(x)\Phi_{\varepsilon,\bar{t},\bar{\xi^0},\bar{\eta}}\Delta_g \mathcal{W}_{\bar{\delta},\bar{\xi^0},\bar{\eta}}dv_g,
\end{equation*}
and
\begin{equation*}
  \int\limits_{\mathcal{M}}a(x)\nabla_g \mathcal{H}_{\bar{\delta},\bar{\xi^0},\bar{\eta}}\cdot \nabla \Psi_{\varepsilon,\bar{t},\bar{\xi^0},\bar{\eta}}dv_g=-\int\limits_{\mathcal{M}}\Psi_{\varepsilon,\bar{t},\bar{\xi^0},\bar{\eta}}\nabla_g a(x)\cdot\nabla_g\mathcal{H}_{\bar{\delta},\bar{\xi^0},\bar{\eta}} dv_g-\int\limits_{\mathcal{M}}a(x)\Psi_{\varepsilon,\bar{t},\bar{\xi^0},\bar{\eta}}\Delta_g \mathcal{H}_{\bar{\delta},\bar{\xi^0},\bar{\eta}}dv_g,
\end{equation*}
we need estimates of
\begin{equation*}
  \int\limits_{\mathcal{M}}\Phi_{\varepsilon,\bar{t},\bar{\xi^0},\bar{\eta}}\nabla_g a(x)\cdot\nabla_g\mathcal{W}_{\bar{\delta},\bar{\xi^0},\bar{\eta}} dv_g\quad \text{and}\quad \int\limits_{\mathcal{M}}\Psi_{\varepsilon,\bar{t},\bar{\xi^0},\bar{\eta}}\nabla_g a(x)\cdot\nabla_g\mathcal{H}_{\bar{\delta},\bar{\xi^0},\bar{\eta}} dv_g.
\end{equation*}
By \eqref{new1}, \eqref{new2} and Proposition \ref{propo1}, we have
\begin{equation*}
  \int\limits_{\mathcal{M}}\Phi_{\varepsilon,\bar{t},\bar{\xi^0},\bar{\eta}}\nabla_g a(x)\cdot\nabla_g\mathcal{W}_{\bar{\delta},\bar{\xi^0},\bar{\eta}} dv_g\leq \|\nabla_g a(x)\cdot\nabla_g\mathcal{W}_{\bar{\delta},\bar{\xi^0},\bar{\eta}}\|_{\frac{p+1}{p}}\|\Phi_{\varepsilon,\bar{t},\bar{\xi^0},\bar{\eta}}\|_{p+1}=o(\varepsilon),
\end{equation*}
and
\begin{equation*}
\int\limits_{\mathcal{M}}\Psi_{\varepsilon,\bar{t},\bar{\xi^0},\bar{\eta}}\nabla_g a(x)\cdot\nabla_g\mathcal{H}_{\bar{\delta},\bar{\xi^0},\bar{\eta}} dv_g\leq \|\nabla_g a(x)\cdot\nabla_g\mathcal{H}_{\bar{\delta},\bar{\xi^0},\bar{\eta}} dv_g\|_{\frac{q+1}{q}}\|\Psi_{\varepsilon,\bar{t},\bar{\xi^0},\bar{\eta}}\|_{q+1}=o(\varepsilon).
\end{equation*}
By Proposition \ref{propo1}, Lemma \ref{error}, using the H\"{o}lder and Sobolev inequalities, we get
\begin{align*}
  &\int\limits_{\mathcal{M}}a(x)\big(-\Delta_g  \mathcal{W}_{\bar{\delta},\bar{\xi^0},\bar{\eta}}+h\mathcal{W}_{\bar{\delta},\bar{\xi^0},\bar{\eta}}-f_\varepsilon(\mathcal{H}_{\bar{\delta},\bar{\xi^0},\bar{\eta}})\big)
  \Phi_{\varepsilon,\bar{t},
  \bar{\xi^0},\bar{\eta}}d v_g\\
  \leq& C\|-\Delta_g  \mathcal{W}_{\bar{\delta},\bar{\xi^0},\bar{\eta}}+h\mathcal{W}_{\bar{\delta},\bar{\xi^0},\bar{\eta}}-f_\varepsilon(\mathcal{H}_{\bar{\delta},\bar{\xi^0},\bar{\eta}})
  \|_{\frac{p+1}{p}}\|\Phi_{\varepsilon,\bar{t},
  \bar{\xi^0},\bar{\eta}}\|_{p+1}=o(\varepsilon),
\end{align*}
\begin{align*}
  &\int\limits_{\mathcal{M}}a(x)\big(-\Delta_g  \mathcal{H}_{\bar{\delta},\bar{\xi^0},\bar{\eta}}+h\mathcal{H}_{\bar{\delta},\bar{\xi^0},\bar{\eta}}-g_\varepsilon(\mathcal{W}_{\bar{\delta},\bar{\xi^0},\bar{\eta}})\big)
  \Psi_{\varepsilon,\bar{t},\bar{\xi^0},\bar{\eta}}d v_g\\
  \leq& C \|-\Delta_g  \mathcal{H}_{\bar{\delta},\bar{\xi^0},\bar{\eta}}+h\mathcal{H}_{\bar{\delta},\bar{\xi^0},\bar{\eta}}-g_\varepsilon(\mathcal{W}_{\bar{\delta},\bar{\xi^0},\bar{\eta}})\|_{\frac{q+1}{q}}
  \|\Psi_{\varepsilon,\bar{t},\bar{\xi^0},\bar{\eta}}\|_{q+1}=o(\varepsilon),
  \end{align*}
  and
\begin{align*}
  &\int\limits_{\mathcal{M}}a(x)\big(\nabla_g \Psi_{\varepsilon,\bar{t},\bar{\xi^0},\bar{\eta}}\cdot \nabla _g \Phi_{\varepsilon,\bar{t},\bar{\xi^0},\bar{\eta}}+h\Psi_{\varepsilon,\bar{t},\bar{\xi^0},\bar{\eta}}\Phi_{\varepsilon,\bar{t},\bar{\xi^0},\bar{\eta}}\big)dv_g\\
  \leq &C\|\nabla_g \Psi_{\varepsilon,\bar{t},\bar{\xi^0},\bar{\eta}}\|_{p^*}\|\nabla_g \Phi_{\varepsilon,\bar{t},\bar{\xi^0},\bar{\eta}}\|_{q^*}+C\|\Psi_{\varepsilon,\bar{t},\bar{\xi^0},\bar{\eta}}\|_2\|\Phi_{\varepsilon,\bar{t},\bar{\xi^0},\bar{\eta}}\|_2=o(\varepsilon).
\end{align*}
Moreover, by the mean value formula, Lemma \ref{gs},  we obtain
\begin{align*}%\label{three}
  &\int\limits_{\mathcal{M}}a(x)\big(F_\varepsilon(\mathcal{H}_{\bar{\delta},\bar{\xi^0},\bar{\eta}}+\Phi_{\varepsilon,\bar{t},\bar{\xi^0},\bar{\eta}})-
  F_\varepsilon(\mathcal{H}_{\bar{\delta},\bar{\xi^0},\bar{\eta}})-f_\varepsilon(\mathcal{H}_{\bar{\delta},\bar{\xi^0},\bar{\eta}}) \Phi_{\varepsilon,\bar{t},\bar{\xi^0},\bar{\eta}}\big)dv_g\nonumber\\
  \leq& C\int\limits_{\mathcal{M}}\mathcal{H}_{\bar{\delta},\bar{\xi^0},\bar{\eta}}^{p-1-\alpha\varepsilon}\Phi^2_{\varepsilon,\bar{t},\bar{\xi^0},\bar{\eta}}dv_g+
  C\int\limits_{\mathcal{M}}\Phi_{\varepsilon,\bar{t},\bar{\xi^0},\bar{\eta}}^{p+1-\alpha\varepsilon}dv_g\nonumber\\
  \leq& C\|\Phi_{\varepsilon,\bar{t},\bar{\xi^0},\bar{\eta}}\|_{p+1-\alpha\varepsilon}^2\sum\limits_{j=1}^k\|H_{\delta_j,\xi^0_j,\eta_j}\|_{p+1-\alpha\varepsilon}^{p-1-\alpha\varepsilon}+C\|\Phi_{\varepsilon,\bar{t},\bar{\xi^0},\bar{\eta}}\|_{p+1-\alpha\varepsilon}^{p+1-\alpha\varepsilon}
  =o(\varepsilon),
\end{align*}
and
\begin{align*}%\label{four}
  &\int\limits_{\mathcal{M}}a(x)\big(G_\varepsilon(\mathcal{W}_{\bar{\delta},\bar{\xi^0},\bar{\eta}}+\Psi_{\varepsilon,\bar{t},\bar{\xi^0},\bar{\eta}})-
  G_\varepsilon(\mathcal{W}_{\bar{\delta},\bar{\xi^0},\bar{\eta}})-g_\varepsilon(\mathcal{W}_{\bar{\delta},\bar{\xi^0},\bar{\eta}}) \Psi_{\varepsilon,\bar{t},\bar{\xi^0},\bar{\eta}}\big)dv_g\nonumber\\
  \leq& C\int\limits_{\mathcal{M}}\mathcal{W}_{\bar{\delta},\bar{\xi^0},\bar{\eta}}^{q-1-\beta\varepsilon}\Psi^2_{\varepsilon,\bar{t},\bar{\xi^0},\bar{\eta}}dv_g+
  C\int\limits_{\mathcal{M}}\Psi_{\varepsilon,\bar{t},\bar{\xi^0},\bar{\eta}}^{q+1-\beta\varepsilon}dv_g\nonumber\\
  \leq& C\|\Psi_{\varepsilon,\bar{t},\bar{\xi^0},\bar{\eta}}\|_{q+1-\beta\varepsilon}^2\sum\limits_{j=1}^k\|W_{\delta_j,\xi^0_j,\eta_j}\|_{q+1-\beta\varepsilon}^{q-1-\beta\varepsilon}+
  C\|\Psi_{\varepsilon,\bar{t},\bar{\xi^0},\bar{\eta}}\|_{q+1-\beta\varepsilon}^{q+1-\beta\varepsilon}=o(\varepsilon).
\end{align*}
This completes the proof of the $C^0$-estimate.

Next, we prove the $C^1$-estimate.
For any $(\varphi,\psi)\in \mathcal{X}_{p,q}(\mathcal{M})$, by Proposition \ref{propo1}, there exist constants $c_{01},c_{02},\cdots,c_{0k}, c_{11},c_{12},\cdots,c_{1k},\cdots,c_{N1},c_{N2},\cdots,c_{Nk}$ such that
\begin{equation}\label{real}
  \mathcal{J}_\varepsilon'(\mathcal{W}_{\bar{\delta},\bar{\xi^0},\bar{\eta}}+\Psi_{\varepsilon,\bar{t},\bar{\xi^0},\bar{\eta}},\mathcal{H}_{\bar{\delta},\bar{\xi^0},\bar{\eta}}+\Phi_{\varepsilon,\bar{t},\bar{\xi^0},\bar{\eta}})(\varphi,\psi)=
  \sum\limits_{l=0}^N\sum\limits_{m=1}^kc_{lm}\big\langle\big(\Psi_{\delta_m,\xi^0_m,\eta_m}^l,\Phi^l_{\delta_m,\xi^0_m,\eta_m}\big),(\varphi,\psi)\big\rangle_h.
\end{equation}
Moreover, by \cite[Lemma 5.4]{CW}, we have
\begin{equation}\label{cla}
  \sum\limits_{l=0}^N\sum\limits_{m=1}^k|c_{lm}|=O(\varepsilon^{\vartheta}),
\end{equation}
for any $\vartheta\in (0,1)$. By \eqref{ch1} and \eqref{ch2}, for any $1\leq l \leq N$ and $1\leq m\leq k$, we can compute
  \begin{align*}
  &\partial_{t_m}\widetilde{\mathcal{J}}_{\varepsilon}(\bar{t},\bar{\eta})-\partial_{t_m} \mathcal{J}_\varepsilon\big(\mathcal{W}_{\bar{\delta},\bar{\xi^0},\bar{\eta}},\mathcal{H}_{\bar{\delta},\bar{\xi^0},\bar{\eta}}\big)\nonumber\\
  %=&\widetilde{\mathcal{J}}_{\varepsilon}(t,\xi)- \mathcal{J}_{\varepsilon}(W_{\delta_{\varepsilon}(t),\xi},H_{\delta_{\varepsilon}(t),\xi})\\
  =&\big(\widetilde{\mathcal{J}}'_{\varepsilon}(\bar{t},\bar{\eta})- \mathcal{J}'_\varepsilon\big(\mathcal{W}_{\bar{\delta},\bar{\xi^0},\bar{\eta}},\mathcal{H}_{\bar{\delta},\bar{\xi^0},\bar{\eta}}\big)\big)\big(\partial_{t_m}\mathcal{W}_{\bar{\delta},\bar{\xi^0},\bar{\eta}},\partial_{t_m}\mathcal{H}_{\bar{\delta},\bar{\xi^0},\bar{\eta}}\big)
+\widetilde{\mathcal{J}}'_{\varepsilon}(\bar{t},\bar{\eta})\big(\partial_{t_m}\Psi_{\varepsilon,\bar{t},\bar{\xi^0},\bar{\eta}},\partial_{t_m}\Phi_{\varepsilon,\bar{t},\bar{\xi^0},\bar{\eta}}\big)\nonumber  \\=&-\frac{1}{t_m}\big(\widetilde{\mathcal{J}}'_{\varepsilon}(\bar{t},\bar{\eta})- \mathcal{J}'_\varepsilon\big(\mathcal{W}_{\bar{\delta},\bar{\xi^0},\bar{\eta}},\mathcal{H}_{\bar{\delta},\bar{\xi^0},\bar{\eta}}\big)\big)\Big(\Psi^0_{\delta_m,\xi^0_m,\eta_m}+\sum\limits_{l=1}^N\eta_{ml}\Psi^l_{\delta_m,\xi^0_m,\eta_m},\Phi^0_{\delta_m,\xi^0_m,\eta_m}+\sum\limits_{l=1}^N\eta_{ml}\Phi^l_{\delta_m,\xi^0_m,\eta_m}\Big)
\nonumber\\&+\mathcal{J}_\varepsilon'(\mathcal{W}_{\bar{\delta},\bar{\xi^0},\bar{\eta}}+\Psi_{\varepsilon,\bar{t},\bar{\xi^0},\bar{\eta}},\mathcal{H}_{\bar{\delta},\bar{\xi^0},\bar{\eta}}+\Phi_{\varepsilon,\bar{t},\bar{\xi^0},\bar{\eta}})\big(\partial_{t_m}\Psi_{\varepsilon,\bar{t},\bar{\xi^0},\bar{\eta}},\partial_{t_m}\Phi_{\varepsilon,\bar{t},\bar{\xi^0},\bar{\eta}}\big),
\end{align*}
  and
  \begin{align*}
  &\partial_{\eta_{ml}}\widetilde{\mathcal{J}}_{\varepsilon}(\bar{t},\bar{\eta})-\partial_{\eta_{ml}} \mathcal{J}_\varepsilon\big(\mathcal{W}_{\bar{\delta},\bar{\xi^0},\bar{\eta}},\mathcal{H}_{\bar{\delta},\bar{\xi^0},\bar{\eta}}\big)\nonumber\\
  =&\big(\widetilde{\mathcal{J}}'_{\varepsilon}(\bar{t},\bar{\eta})- \mathcal{J}'_\varepsilon\big(\mathcal{W}_{\bar{\delta},\bar{\xi^0},\bar{\eta}},\mathcal{H}_{\bar{\delta},\bar{\xi^0},\bar{\eta}}\big)\big)\big(\partial_{\eta_{ml}}\mathcal{W}_{\bar{\delta},\bar{\xi^0},\bar{\eta}},\partial_{\eta_{ml}}\mathcal{H}_{\bar{\delta},\bar{\xi^0},\bar{\eta}}\big)
+\widetilde{\mathcal{J}}'_{\varepsilon}(\bar{t},\bar{\eta})\big(\partial_{\eta_{ml}}\Psi_{\varepsilon,\bar{t},\bar{\xi^0},\bar{\eta}},\partial_{\eta_{ml}}\Phi_{\varepsilon,\bar{t},\bar{\xi^0},\bar{\eta}}\big)\nonumber  \\=&\big(\widetilde{\mathcal{J}}'_{\varepsilon}(\bar{t},\bar{\eta})- \mathcal{J}'_\varepsilon\big(\mathcal{W}_{\bar{\delta},\bar{\xi^0},\bar{\eta}},\mathcal{H}_{\bar{\delta},\bar{\xi^0},\bar{\eta}}\big)\big)\big(\Psi^l_{\delta_m,\xi^0_m,\eta_m},\Phi^l_{\delta_m,\xi^0_m,\eta_m}\big)
\nonumber\\&+\mathcal{J}_\varepsilon'(\mathcal{W}_{\bar{\delta},\bar{\xi^0},\bar{\eta}}+\Psi_{\varepsilon,\bar{t},\bar{\xi^0},\bar{\eta}},\mathcal{H}_{\bar{\delta},\bar{\xi^0},\bar{\eta}}+\Phi_{\varepsilon,\bar{t},\bar{\xi^0},\bar{\eta}})\big(\partial_{\eta_{ml}}\Psi_{\varepsilon,\bar{t},\bar{\xi^0},\bar{\eta}},\partial_{\eta_{ml}}\Phi_{\varepsilon,\bar{t},\bar{\xi^0},\bar{\eta}}\big).
\end{align*}
For any $0\leq l \leq N$ and $1\leq m\leq k$, we have
\begin{align*}
  &\big(\widetilde{\mathcal{J}}'_{\varepsilon}(\bar{t},\bar{\eta})- \mathcal{J}'_\varepsilon\big(\mathcal{W}_{\bar{\delta},\bar{\xi^0},\bar{\eta}},\mathcal{H}_{\bar{\delta},\bar{\xi^0},\bar{\eta}}\big)\big)\big(\Psi^l_{\delta_m,\xi^0_m,\eta_m},\Phi^l_{\delta_m,\xi^0_m,\eta_m}\big)\\
  =&-\int\limits_{\mathcal{M}}\Phi_{\varepsilon,\bar{t},\bar{\xi^0},\bar{\eta}}\nabla_g a(x)\cdot \nabla _g\Psi^l_{\delta_m,\xi^0_m,\eta_m}dv_g-\int\limits_{\mathcal{M}}\Psi_{\varepsilon,\bar{t},\bar{\xi^0},\bar{\eta}}\nabla_g a(x)\cdot \nabla _g\Phi^l_{\delta_m,\xi^0_m,\eta_m}dv_g\\&+
  \int\limits_{\mathcal{M}}a(x)\big(-\Delta_g \Psi^l_{\delta_m,\xi^0_m,\eta_m}+h\Psi^l_{\delta_m,\xi^0_m,\eta_m}-f'_\varepsilon(\mathcal{H}_{\bar{\delta},\bar{\xi^0},\bar{\eta}})\Phi^l_{\delta_m,\xi^0_m,\eta_m}\big)\Phi_{\varepsilon,\bar{t},\bar{\xi^0},\bar{\eta}}d v_g\nonumber\\
  &+\int\limits_{\mathcal{M}}a(x)\big(-\Delta_g \Phi^l_{\delta_m,\xi^0_m,\eta_m}+h\Phi^l_{\delta_m,\xi^0_m,\eta_m}-g'_\varepsilon(\mathcal{W}_{\bar{\delta},\bar{\xi^0},\bar{\eta}})\Psi^l_{\delta_m,\xi^0_m,\eta_m}\big)\Psi_{\varepsilon,\bar{t},\bar{\xi^0},\bar{\eta}}d v_g\nonumber\\
  &-\int\limits_{\mathcal{M}}a(x)\big(f_\varepsilon(\mathcal{H}_{\bar{\delta},\bar{\xi^0},\bar{\eta}}+\Phi_{\varepsilon,\bar{t},\bar{\xi^0},\bar{\eta}})-
  f_\varepsilon(\mathcal{H}_{\bar{\delta},\bar{\xi^0},\bar{\eta}})-f'_\varepsilon(\mathcal{H}_{\bar{\delta},\bar{\xi^0},\bar{\eta}}) \Phi_{\varepsilon,\bar{t},\bar{\xi^0},\bar{\eta}}\big) \Phi^l_{\delta_m,\xi^0_m,\eta_m}dv_g\nonumber\\
  &-\int\limits_{\mathcal{M}}a(x)\big(g_\varepsilon(\mathcal{W}_{\bar{\delta},\bar{\xi^0},\bar{\eta}}+\Psi_{\varepsilon,\bar{t},\bar{\xi^0},\bar{\eta}})-
  g_\varepsilon(\mathcal{W}_{\bar{\delta},\bar{\xi^0},\bar{\eta}})-g'_\varepsilon(\mathcal{W}_{\bar{\delta},\bar{\xi^0},\bar{\eta}}) \Psi_{\varepsilon,\bar{t},\bar{\xi^0},\bar{\eta}}\big) \Psi^l_{\delta_m,\xi^0_m,\eta_m}dv_g.
\end{align*}
Similar to \eqref{new1} and \eqref{new2}, we have
\begin{align*}
  \int\limits_{\mathcal{M}}\big|\nabla_g a(x)\cdot \nabla _g\Psi^l_{\delta_m,\xi^0_m,\eta_m}\big|^{\frac{p+1}{p}}dv_g=O(\varepsilon^{\frac{p+1}{p}})\quad \text{and}\quad
  \int\limits_{\mathcal{M}}\big|\nabla_g a(x)\cdot \nabla _g\Phi^l_{\delta_m,\xi^0_m,\eta_m}\big|^{\frac{q+1}{q}}dv_g=O(\varepsilon^{\frac{q+1}{q}}).
\end{align*}
This with Proposition \ref{propo1} yields
\begin{equation*}
 \int\limits_{\mathcal{M}}\Phi_{\varepsilon,\bar{t},\bar{\xi^0},\bar{\eta}}\nabla_g a(x)\cdot \nabla _g\Psi^l_{\delta_m,\xi^0_m,\eta_m}dv_g=o(\varepsilon) \quad \text{and}\quad\int\limits_{\mathcal{M}}\Psi_{\varepsilon,\bar{t},\bar{\xi^0},\bar{\eta}}\nabla_g a(x)\cdot \nabla _g\Phi^l_{\delta_m,\xi^0_m,\eta_m}dv_g=o(\varepsilon).
\end{equation*}
By Proposition \ref{propo1}, using the H\"{o}lder and Sobolev inequalities,  arguing as Lemma \ref{error}, for any $0\leq l \leq N$ and $1\leq m\leq k$, we have
\begin{align*}
  &\int\limits_{\mathcal{M}}a(x)\big(-\Delta_g \Psi^l_{\delta_m,\xi^0_m,\eta_m}+h\Psi^l_{\delta_m,\xi^0_m,\eta_m}-f'_\varepsilon(\mathcal{H}_{\bar{\delta},\bar{\xi^0},\bar{\eta}})\Phi^l_{\delta_m,\xi^0_m,\eta_m}\big)
  \Phi_{\varepsilon,\bar{t},\bar{\xi^0},\bar{\eta}}d v_g\\
  \leq &
  C\|-\Delta_g \Psi^l_{\delta_m,\xi^0_m,\eta_m}+h\Psi^l_{\delta_m,\xi^0_m,\eta_m}-f'_\varepsilon(\mathcal{H}_{\bar{\delta},\bar{\xi^0},\bar{\eta}})\Phi^l_{\delta_m,\xi^0_m,\eta_m}
  \|_{\frac{p+1}{p}}\|\Phi_{\varepsilon,\bar{t},\bar{\xi^0},\bar{\eta}}\|_{p+1}=o(\varepsilon),
\end{align*}
and
\begin{align*}
  &\int\limits_{\mathcal{M}}a(x)\big(-\Delta_g \Phi^l_{\delta_m,\xi^0_m,\eta_m}+h\Phi^l_{\delta_m,\xi^0_m,\eta_m}-g'_\varepsilon(\mathcal{W}_{\bar{\delta},\bar{\xi^0},\bar{\eta}})\Psi^l_{\delta_m,\xi^0_m,\eta_m}\big)
  \Psi_{\varepsilon,\bar{t},\bar{\xi^0},\bar{\eta}}d v_g\\
  \leq & C\|-\Delta_g \Phi^l_{\delta_m,\xi^0_m,\eta_m}+h\Phi^l_{\delta_m,\xi^0_m,\eta_m}-g'_\varepsilon(\mathcal{W}_{\bar{\delta},\bar{\xi^0},\bar{\eta}})\Psi^l_{\delta_m,\xi^0_m,\eta_m}
  \|_{\frac{q+1}{q}}\|\Psi_{\varepsilon,\bar{t},\bar{\xi^0},\bar{\eta}}\|_{q+1}=o(\varepsilon).
\end{align*}
Moreover, by the mean value formula, Lemma \ref{gs},  we obtain
\begin{align*}
  &\int\limits_{\mathcal{M}}a(x)\big(f_\varepsilon(\mathcal{H}_{\bar{\delta},\bar{\xi^0},\bar{\eta}}+\Phi_{\varepsilon,\bar{t},\bar{\xi^0},\bar{\eta}})-
  f_\varepsilon(\mathcal{H}_{\bar{\delta},\bar{\xi^0},\bar{\eta}})-f'_\varepsilon(\mathcal{H}_{\bar{\delta},\bar{\xi^0},\bar{\eta}}) \Phi_{\varepsilon,\bar{t},\bar{\xi^0},\bar{\eta}}\big) \Phi^l_{\delta_m,\xi^0_m,\eta_m}dv_g\\
  \leq& C\int\limits_{\mathcal{M}}\mathcal{H}_{\bar{\delta},\bar{\xi^0},\bar{\eta}}^{p-2-\alpha\varepsilon}\Phi_{\varepsilon,\bar{t},\bar{\xi^0},\bar{\eta}}^{2}\Phi^l_{\delta_m,\xi^0_m,\eta_m}dv_g \leq C\|\Phi^l_{\delta_m,\xi^0_m,\eta_m}\|_{p+1}\|\Phi_{\varepsilon,\bar{t},\bar{\xi^0},\bar{\eta}}\|_{p+1}^{2}\sum\limits_{j=1}^k\|H_{\delta_j,\xi^0_j,\eta_j}
  \|_{\frac{(p-2-\alpha\varepsilon)(p+1)}{p-2}}^{p-2-\alpha\varepsilon}=o(\varepsilon),
\end{align*}
and
\begin{align*}
  &\int\limits_{\mathcal{M}}a(x)\big(g_\varepsilon(\mathcal{W}_{\bar{\delta},\bar{\xi^0},\bar{\eta}}+\Psi_{\varepsilon,\bar{t},\bar{\xi^0},\bar{\eta}})-
  g_\varepsilon(\mathcal{W}_{\bar{\delta},\bar{\xi^0},\bar{\eta}})-g'_\varepsilon(\mathcal{W}_{\bar{\delta},\bar{\xi^0},\bar{\eta}}) \Psi_{\varepsilon,\bar{t},\bar{\xi^0},\bar{\eta}}\big)\Psi^l_{\delta_m,\xi^0_m,\eta_m}dv_g
  \\ \leq& \left\{
    \begin{array}{ll}
   \displaystyle C\int\limits_{\mathcal{M}}\mathcal{W}_{\bar{\delta},\bar{\xi^0},\bar{\eta}}^{q-2-\beta\varepsilon}\Psi^2_{\varepsilon,\bar{t},\bar{\xi^0},\bar{\eta}}\Psi^l_{\delta_m,\xi^0_m,\eta_m}dv_g+
  C\int\limits_{\mathcal{M}}\Psi_{\varepsilon,\bar{t},\bar{\xi^0},\bar{\eta}}^{q-\beta\varepsilon}\Psi^l_{\delta_m,\xi_m,\eta_m}dv_g,\quad &\text{if $q>2$},\\
  \displaystyle C\int\limits_{\mathcal{M}}\mathcal{W}_{\bar{\delta},\bar{\xi^0},\bar{\eta}}^{q-2-\beta\varepsilon}\Psi^2_{\varepsilon,\bar{t},\bar{\xi^0},\bar{\eta}}\Psi^l_{\delta_m,\xi^0_m,\eta_m}dv_g,\quad &\text{if $q\leq 2$},
  \end{array}
  \right.
  \\ \leq& \left\{
    \begin{array}{ll}
   \displaystyle C\|\Psi^l_{\delta_m,\xi^0_m,\eta_m}\|_{q+1}\|\Psi_{\varepsilon,\bar{t},\bar{\xi^0},\bar{\eta}}\|_{q+1}^{2}\sum\limits_{j=1}^k\|W_{\delta_j,\xi^0_j,\eta_j}\|_{\frac{(q-2-\beta\varepsilon)(q+1)}{q-2}}^{q-2-\beta\varepsilon}\\
   +
   C\|\Psi^l_{\delta_m,\xi^0_m,\eta_m}\|_{q+1}\|\Psi_{\varepsilon,\bar{t},\bar{\xi^0},\bar{\eta}}\|_{\frac{(q-\beta\varepsilon)(q+1)}{q}}^{q-\beta\varepsilon},\quad &\text{if $q>2$},\\
  C\|\Psi^l_{\delta_m,\xi^0_m,\eta_m}\|_{q+1}\|\Psi_{\varepsilon,\bar{t},\bar{\xi^0},\bar{\eta}}\|_{q+1}^{2}\sum\limits_{j=1}^k\|W_{\delta_j,\xi^0_j,\eta_j}\|_{\frac{(q-2-\beta\varepsilon)(q+1)}{q-2}}^{q-2-\beta\varepsilon},\quad
  &\text{if $q\leq 2$},
  \end{array}
  \right.
  \\
  =&o(\varepsilon).
  \end{align*}
Finally, with the aid of \eqref{e54}-\eqref{e55} and \eqref{real}-\eqref{cla}, we get
\begin{align*}
  &\mathcal{J}_\varepsilon'(\mathcal{W}_{\bar{\delta},\bar{\xi^0},\bar{\eta}}+\Psi_{\varepsilon,\bar{t},\bar{\xi^0},\bar{\eta}},\mathcal{H}_{\bar{\delta},\bar{\xi^0},\bar{\eta}}+\Phi_{\varepsilon,\bar{t},\bar{\xi^0},\bar{\eta}})\big(\partial_{t_m}\Psi_{\varepsilon,\bar{t},\bar{\xi^0},\bar{\eta}},\partial_{t_m}\Phi_{\varepsilon,\bar{t},\bar{\xi^0},\bar{\eta}}\big)
  \\=&\sum\limits_{i=0}^N\sum\limits_{j=1}^kc_{ij}\big\langle\big(\Psi_{\delta_j,\xi^0_j,\eta_j}^i,\Phi^i_{\delta_j,\xi^0_j,\eta_j}\big),\big(\partial_{t_m}\Psi_{\varepsilon,\bar{t},\bar{\xi^0},\bar{\eta}},\partial_{t_m}\Phi_{\varepsilon,\bar{t},\bar{\xi^0},\bar{\eta}}\big)\big\rangle_h
  \\=&-\sum\limits_{i=0}^N\sum\limits_{j=1}^kc_{ij}\delta_{jm}\big\langle\big(\partial_{t_m}\Psi_{\delta_j,\xi^0_j,\eta_j}^i,\partial_{t_m}\Phi^i_{\delta_j,\xi^0_j,\eta_j}\big),\big(\Psi_{\varepsilon,\bar{t},\bar{\xi^0},\bar{\eta}},\Phi_{\varepsilon,\bar{t},\bar{\xi^0},\bar{\eta}}\big)\big\rangle_h=o(\varepsilon),
\end{align*}
and
\begin{align*}
  &\mathcal{J}_\varepsilon'(\mathcal{W}_{\bar{\delta},\bar{\xi^0},\bar{\eta}}+\Psi_{\varepsilon,\bar{t},\bar{\xi^0},\bar{\eta}},\mathcal{H}_{\bar{\delta},\bar{\xi^0},\bar{\eta}}+\Phi_{\varepsilon,\bar{t},\bar{\xi^0},\bar{\eta}})\big(\partial_{\eta_{ml}}\Psi_{\varepsilon,\bar{t},\bar{\xi^0},\bar{\eta}},\partial_{\eta_{ml}}\Phi_{\varepsilon,\bar{t},\bar{\xi^0},\bar{\eta}}\big)\\
  =&\sum\limits_{i=0}^N\sum\limits_{j=1}^kc_{ij}\big\langle\big(\Psi_{\delta_j,\xi^0_j,\eta_j}^i,\Phi^i_{\delta_j,\xi^0_j,\eta_j}\big),\big(\partial_{\eta_{ml}}\Psi_{\varepsilon,\bar{t},\bar{\xi^0},\bar{\eta}},\partial_{\eta_{ml}}\Phi_{\varepsilon,\bar{t},\bar{\xi^0},\bar{\eta}}\big)\big\rangle_h\\
  =&-\sum\limits_{i=0}^N\sum\limits_{j=1}^kc_{ij}\delta_{jm}\big\langle\big(\partial_{\eta_{ml}}\Psi_{\delta_j,\xi^0_j,\eta_j}^i,\partial_{\eta_{ml}}\Phi^i_{\delta_j,\xi^0_j,\eta_j}\big),\big(\Psi_{\varepsilon,\bar{t},\bar{\xi^0},\bar{\eta}},\Phi_{\varepsilon,\bar{t},\bar{\xi^0},\bar{\eta}}\big)\big\rangle_h
=o(\varepsilon).
\end{align*}
This concludes the proof.
\end{proof}

\noindent{\bfseries Acknowledgements}:

The authors were supported by Natural Science Foundation of Chongqing, China cstc2021ycjh-bgzxm0115.

\noindent{\bfseries Declarations}:

{\bf Conflict of Interest}: The authors declare that they have no conflict of interest.

\noindent
{\bf Data Availability:}

Date sharing is not applicable to this article as no new data were created analyzed in this study.

\end{document}